\documentclass[a4paper, 11pt]{amsart}
\usepackage{geometry}
\usepackage{amssymb}
\usepackage[all]{xy}
\usepackage{hyperref,aliascnt}
\usepackage{enumerate}
\usepackage{stmaryrd}
\usepackage{mathtools}
\usepackage{xcolor}

\numberwithin{equation}{section}

\newtheorem{lma}{Lemma}[section]

\newaliascnt{thmCt}{lma}
\newtheorem{thm}[thmCt]{Theorem}
\aliascntresetthe{thmCt}

\newaliascnt{corCt}{lma}
\newtheorem{cor}[corCt]{Corollary}
\aliascntresetthe{corCt}

\newaliascnt{prpCt}{lma}
\newtheorem{prp}[prpCt]{Proposition}
\aliascntresetthe{prpCt}

\newtheorem*{thm*}{Theorem}
\newtheorem*{cor*}{Corollary}
\newtheorem*{prp*}{Proposition}

\theoremstyle{definition}

\newaliascnt{pgrCt}{lma}
\newtheorem{pgr}[pgrCt]{}
\aliascntresetthe{pgrCt}

\newaliascnt{dfnCt}{lma}
\newtheorem{dfn}[dfnCt]{Definition}
\aliascntresetthe{dfnCt}

\newaliascnt{rmkCt}{lma}
\newtheorem{rmk}[rmkCt]{Remark}
\aliascntresetthe{rmkCt}

\newaliascnt{rmksCt}{lma}

\aliascntresetthe{rmksCt}

\newaliascnt{qstCt}{lma}
\newtheorem{qst}[qstCt]{Question}
\aliascntresetthe{qstCt}

\newaliascnt{pbmCt}{lma}

\aliascntresetthe{pbmCt}

\newaliascnt{exaCt}{lma}
\newtheorem{exa}[exaCt]{Example}
\aliascntresetthe{exaCt}

\newaliascnt{exasCt}{lma}

\aliascntresetthe{exasCt}

\newaliascnt{conjCt}{lma}

\aliascntresetthe{conjCt}

\newaliascnt{ntnCt}{lma}
\newtheorem{ntn}[ntnCt]{Notation}
\aliascntresetthe{ntnCt}

\newcommand{\NN}{\mathbb{N}}

\newcommand{\ZZ}{\mathbb{Z}}

\newcommand{\ca}{$C^*$-al\-ge\-bra}
\newcommand{\Cs}{$C^*$-al\-ge\-bra}

\newcommand{\axiomO}[1]{(O#1)}

\DeclareMathOperator{\Cu}{Cu}
\DeclareMathOperator{\Lsc}{Lsc}

\usepackage{stmaryrd}

\newcommand{\CatCu}{\ensuremath{\mathrm{Cu}}}

\newcommand{\andSep}{\,\,\,\text{ and }\,\,\,}
\newcommand{\orSep}{\,\,\,\text{ or }\,\,\,}

\newcommand{\CuSgp}{$\CatCu$-sem\-i\-group}

\newcommand{\CuMor}{$\CatCu$-mor\-phism}

\newcounter{theoremintro}

\newtheorem{thmIntro}[theoremintro]{Theorem}

\DeclareMathOperator{\Int}{Int}

\newcommand\LscI{\Lsc ([0,1],\overline{\mathbb{N}})}

\newcommand{\Msubsection}[2]{\subsection{\texorpdfstring{#1}{#2}}}

\newcommand{\Msubsubsection}[2]{\subsubsection{\texorpdfstring{#1}{#2}}}

\title{A local characterization for the Cuntz semigroup of AI-algebras}

\author{Eduard Vilalta}
\address{E.~Vilalta, Departament de Matem\`{a}tiques,
Universitat Aut\`{o}noma de Barcelona,
08193 Bellaterra, Barcelona, Spain}
\email{evilalta@mat.uab.cat}
\thanks{
The author was partially supported by MINECO (grant No.\ PRE2018-083419 and No.\ MTM2017-83487-P), and by the Comissionat per Universitats i Recerca de la Generalitat de Catalunya (grant No.\ 2017SGR01725).
}
\subjclass[2010]%
{Primary
46L05, % General theory of C*-algebras
46L85, % Noncommutative topology
06F05. % Ordered semigroups and monoids
Secondary 
06A06, 
46L80, 
19K99.
}
\date{\today}
\begin{document}

%==========================================================================================
\begin{abstract}
We give a local characterization for the Cuntz semigroup of AI-algebras building upon Shen's characterization of dimension groups. Using this result, we provide an abstract characterization for the Cuntz semigroup of AI-algebras.
\end{abstract}

\maketitle

%==========================================================================================

\section{Introduction}

The Cuntz semigroup of a \ca{} $A$, denoted by $\Cu (A)$, is a powerful invariant introduced by Cuntz in \cite{Cun78DimFct} that generalizes the construction of the Murray-von Neumann semigroup of projections. In fact, for the class of simple, nuclear $\mathcal{Z}$-stable \ca{s}, the Cuntz semigroup functor is, suitably interpreted, equivalent to the Elliott invariant (see \cite{AntDadPerSan14RecoverElliott}).

There are a number of regularity properties formulated for Cuntz semigroups that appear in the classification program of simple \ca{s}. Among them, almost unperforation stands out, as it appeared in Toms' example \cite{Tom08InfFamily} to distinguish two simple \ca{s} that otherwise agreed on their Elliott invariant and many reasonable extensions of it. It also features prominently in the Toms-Winter conjecture (see \cite{Win12NuclDimZstable}, \cite{Ror04StableRealRankZ}, \cite{CasEviTikWhiWin19arX:NucDimSimple}, \cite{KirRor14CentralSeq}, \cite{Sat12arx:TraceSpace} and \cite{TomWhiWin15ZStableFdBauer}). Recently, new structural properties for the Cuntz semigroup of \ca{s} of stable rank one have been uncovered, which lead to the solution of three open problems for this class \cite{AntPerRobThi21:CuntzSR1}.

In \cite{CowEllIva08CuInv}, some of the properties of $\Cu (A)$ were abstracted into a category of semigroups, termed $\Cu$, in order to reflect its continuous nature; see \autoref{pgr:O1O4} below. One of the relevant notions while studying the objects in $\Cu$, often called $\Cu$-semigroups, is that of compact containment, in symbols $\ll$, which is the analogue of the compact containment relation in the lattice of open sets of a compact topological space (see, for example, \cite[Proposition~I-1.22.1]{GieHof+03Domains}). Projections in a \ca{} are the natural examples (and in relevant cases the only examples) of compact elements, that is, elements that are $\ll$-below themselves. For a \ca{} $A$ of stable rank one, $\Cu (A)$ further satisfies a weaker form of cancellation, as proved in \cite[Theorem~4.3]{RorWin10ZRevisited}; see \autoref{pgr:O5O6}.

The Cuntz semigroup of a \ca{} is not usually algebraically ordered (it is,  for example, in the finite-dimensional situation). However, an appropriate substitute for this property was proved in \cite{RorWin10ZRevisited}, and has been termed since axiom \axiomO{5} (see also \cite[Proposition~5.1.1]{Rob13Cone}).

The Cuntz semigroup has also been succesfully used in the classification of certain non-simple \ca{s}. Namely, Robert used it in \cite{Rob12LimitsNCCW} to classify, up to approximate unitary equivalence, *-homomorphisms from a limit of one-dimensional NCCW complexes to a \ca{} of stable rank one. He subsequently obtained a classification of $1$-dimensional NCCW complexes with trivial $K_1$-group using their Cuntz semigroup.

 In \cite{CiupElli08}, Ciuperca and Elliott established a one-to-one  correspondence between  the so-called Thomsen semigroup and the Cuntz semigroup, which as a consequence yielded a classification of all separable AI-algebras by means of their Cuntz semigroup; see also \cite{Tho92:IndLimIntAlg}.

We focus in this paper on (separable) AI-algebras, and more concretely on the range problem for the Cuntz semigroup for this class, that is, to determine a natural set of properties that a \CuSgp{} $S$ must satisfy in order to be isomorphic to $\Cu (A)$ for such a \ca{} $A$. In order to ease the notation, we will say that a \ca{} is an AI-algebra if it is *-isomorphic to an inductive limit of the form $\lim_n C[0,1]\otimes F_n$ with $F_n$ finite dimensional for every $n$.

Our line of attack consists of adapting the strategy used to obtain a local characterization of dimension groups due to Shen \cite[Theorem~3.1]{Shen79}, which was then utilized in the Effros-Handelman-Shen theorem \cite[Theorem~2.2]{EffHanShe80}. More explicitly, recall that as a combination of these two results one obtains that a countable unperforated ordered (abelian) group $G$ is order isomorphic to the ordered $K_0$-group of an AF-algebra if and only if, for every order homomorphism $\varphi\colon\ZZ^r\to G$ and element $\alpha \in\ker (\varphi)$, there exist order homomorphisms $\theta ,\phi$ such that the diagram
 \[
  \xymatrix{
  \ZZ^{r} \ar[r]^-{\varphi} \ar[d]_-{\theta} & G\\
  \ZZ^{s} \ar[ru]_-{\phi} & 
  }
 \]
 commutes and $\alpha\in\ker (\theta)$.
 
Following the structure of the proof of the abovementioned result (but with additional care), we give a local characterization for the Cuntz semigroup of AI-algebras; see Theorem \ref{thmIntro:Main} below. We briefly discuss some of the details:

%=============================================
There are two key ingredients that lead to the proof of our main result. The first one is a suitable analogue for \CuSgp{s} of the well known fact that every element of a group is the image of $1\in\ZZ$ through a group homomorphism.

\begin{thmIntro}[cf \ref{lifting_th}]
Let $S$ be a \CuSgp{} satisfying (O5) and weak cancellation, and let $\LscI$ be the $\Cu$-semigroup of lower-semicontinuous functions $[0,1]\to\overline{\NN}$. Then, given an element $s$ and a compact element $p$ in $S$ such that $s\leq p$, there exists a $\Cu$-morphism $\LscI\to S$ mapping $\chi_{(0,1]}$ to $s$ and $1$ to $p$.
\end{thmIntro}

%=============================================
Given a \CuSgp{} $S$, in \autoref{sec:Cauchy} we study Cauchy sequences $(\varphi_i)_i$ of \CuMor{s}, where $\varphi_i\colon \LscI\to S$ for each $i$, with respect to the distance introduced in \cite{CiupElli08} (see also \cite{CiupElliSant11} and \cite{RobSan10}). We show that, under certain conditions, such sequences have a unique limit.

\begin{thmIntro}[cf \ref{Gen_Cauchy}]
 Let $S$ be a weakly cancellative \CuSgp{} satisfying \axiomO{5}, and consider a sequence of \CuMor{s} $\varphi_{i}\colon  \LscI\to S$ such that $d(\varphi_{i},\varphi_{i+1})<\epsilon_{i}$ with $(\epsilon_{i})_i$ strictly decreasing and $\sum_{i=1}^{\infty} \epsilon_{i}<\infty$. Also, assume that $\varphi_{i}(1)=\varphi_{i+1}(1)$ for each $i$.
 
 Then, there exists a unique $\Cu$-morphism $\varphi\colon \LscI\to S$ with  $d(\varphi , \varphi_{i})\to 0$ and $\varphi (1)=\varphi_{i}(1)$ for every $i$.
\end{thmIntro}

%=============================================
Since countably generated \CuSgp{s} are generally far bigger than countably generated groups (the former being usually uncountable and the latter being always countable), one cannot hope to obtain an exact analogue of Shen's theorem \cite[Theorem~3.1]{Shen79} with a commutative diagram. Instead, what one does get is an approximate version of it.

\begin{thmIntro}[{\ref{Main_EHS}, \ref{Main_OneRel}}]\label{thmIntro:Main}
 Let $S$ be a countably based $\Cu$-semigroup satisfying weak cancellation and (O5) where every compactly bounded element is bounded by a compact. Then, $S$ is $\Cu$-isomorphic to the Cuntz semigroup of an AI-algebra if and only if for every \CuMor{} $\varphi\colon \LscI^{r}\to S$, every $\epsilon >0$ and every triple $x,x', y$ in $\LscI^r$ such that $x\ll x'$ with $\varphi (x')\ll \varphi (y)$, there exist $\Cu$-morphisms $\theta ,\phi$ such that the diagram
 \[
  \xymatrix{
  \LscI^{r} \ar[r]^-{\varphi} \ar[d]_-{\theta} & S\\
  \LscI^{s} \ar[ru]_-{\phi} & 
  }
 \]
satisfies:
\begin{enumerate}[(i)]
 \item $d(\phi\theta ,\varphi )< \epsilon$.
 \item $\theta (x)\ll\theta (y)$.
 \item $\varphi (1_{j})=\phi\theta (1_{j})$ for every $1\leq j\leq r$.
\end{enumerate}
\end{thmIntro}

%=============================================
Finally, in \autoref{sec:AbsCha} we introduce \emph{property I} (see \autoref{PropertyI}) and provide, using a discrete version of Theorem \ref{thmIntro:Main} (see \autoref{MainTh}), an abstract characterization for the Cuntz semigroup of  AI-algebras.

\begin{thmIntro}[\ref{thm:IiffAI}]
  Let $S$ be a \CuSgp{}. Then, $S$ is $\Cu$-isomorphic to the Cuntz semigroup of an AI-algebra if and only if $S$ is countably based, compactly bounded and it satisfies (O5), (O6), weak cancellation, and property $\rm{I}$.
\end{thmIntro}

\textbf{Acknowledgments.} This paper constitutes a part of the Ph.D. dissertation of the author. He is grateful to his advisor  Francesc Perera for his guidance and  many fruitful discussions on the subject.

%===============================================

\section{Preliminaries}

Let $S$ be a positively ordered monoid and let $x,y$ be elements in $S$. Recall that we write $x\ll y$, and say that $x$ is \emph{compactly contained in} (or \emph{way-below}) $y$, if for every increasing sequence $(z_n)_n$ whose supremum exists and is such that $y\leq \sup z_n$, we have $x\leq z_n$ for some $n\in\NN$.

%====================================================
\begin{pgr}\label{pgr:O1O4}
As defined in \cite{CowEllIva08CuInv}, a \emph{\CuSgp{}} $S$ is a positively ordered monoid that satisfies the following properties:
\begin{itemize}
 \item[\axiomO{1}] Every increasing sequence in $S$ has a supremum.
 \item[\axiomO{2}] Every element in $S$ can be written as the supremum of an $\ll$-increasing sequence.
 \item[\axiomO{3}] For every $x'\ll x$ and $y'\ll y$, we have $x'+y'\ll x+y$.
 \item[\axiomO{4}] For every pair of increasing sequences $(x_n)_n$ and $(y_n)_n$ we have $\sup_n x_n +\sup_n y_n = \sup_n (x_n+y_n)_n$.
\end{itemize}

%====================================================
The category $\Cu$ is defined as the subcategory of positively ordered monoids whose objects are \CuSgp{s} and whose morphisms are positively ordered monoid morphisms preserving the way-below relation and suprema of increasing sequences. The morphisms in $\Cu$ are usually referred to as \emph{\CuMor{s}}.
\end{pgr}

%====================================================
\begin{pgr}
It was proven in \cite{CowEllIva08CuInv} that there exists a functor between the category of \Cs{s} and $\Cu$. We briefly recall the construction here:

Given two positive elements $a,b$ in a \ca{} $A$, we say that $a$ is \emph{Cuntz subequivalent} to $b$, denoted by $a\precsim b$, if there exists a sequence $(r_n)_n$ in $A$ such that $a=\lim r_n b r_n^*$. The elements $a,b$ are said to be \emph{Cuntz equivalent} if $a\precsim b$ and $b\precsim a$.

The \emph{Cuntz semigroup} of $A$, denoted by $\Cu (A)$, is defined as the quotient $(A\otimes \mathcal{K})_+/\sim$. Writting the class of a positive element $a$ by $[a]$, the Cuntz semigroup of $A$ becomes a \CuSgp{} when endowed with the order induced by $\precsim$ and the addition induced by $[a]+[b]=\left[\begin{psmallmatrix}a&0\\ 0&b\end{psmallmatrix}\right]$.

Given a *-homomorphism $\varphi\colon A\to B$, let $\varphi$ also denote the amplification $\varphi\colon A\otimes \mathcal{K}\to B\otimes\mathcal{K}$. Then, $\varphi$ induces a \CuMor{} between $\Cu (A)$ and $\Cu (B)$, denoted by $\Cu(\varphi)$, by sending an element $[a]\in \Cu (A)$ to  $[\varphi (a)]$.

By \cite[Corollary~3.2.9]{AntoPereThie18}, the category $\Cu$ has inductive limits and the functor $\Cu\colon C^*\to \Cu$ is continuous.
\end{pgr}

%====================================================
\begin{pgr}\label{pgr:O5O6}
The Cuntz semigroup of a \ca{} always satisfies the following additional properties (see \cite[Proposition~4.6]{AntoPereThie18} and \cite[Proposition~5.1.1]{Rob13Cone} respectively):
\begin{enumerate}
 \item[\axiomO{5}] Given $x+y\leq z$, $x'\ll x$ and $y'\ll y$, there exists $c$ such that $x'+c\leq z\leq x+c$ and $y'\ll c$.
 \item[\axiomO{6}] Given $x'\ll x\leq y+z$ there exist elements $v\leq x,y$ and $w\leq x,z$ such that $x'\leq v+w$.
\end{enumerate}

A \CuSgp{} is said to be \emph{countably based} if it contains a countable sup-dense subset, that is to say a countable subset such that each element in the semigroup can be written as the supremum of an increasing sequence of elements in the subset. Every separable \ca{} has a countably based Cuntz semigroup; see, for example, \cite{AntPerSan11PullbacksCu}.

Also recall that a \CuSgp{} is \emph{weakly cancellative } if $x\ll y$ whenever $x+z\ll y+z$. It was proven in \cite[Theorem~4.3]{RorWin10ZRevisited} that stable rank one \ca{s} have weakly cancellative Cuntz semigroups.
\end{pgr}
%====================================================
\begin{pgr}
 We say that a \ca{} is an AI-algebra if it is *-isomorphic to an inductive limit of the form $\lim_n C[0,1]\otimes F_n$ with $F_n$ finite dimensional for every $n\in\NN$.
 
 AI-algebras were classified in \cite{CiupElli08} using the Cuntz semigroup. In fact, following the proof of \cite[Proposition~7.2.8]{LarLauRor00KThy} and using the result from \cite{CiupElli08} one can also prove the next theorem, where recall that $\Cu (C[0,1])$ is $\Cu$-isomorphic to $\LscI$, the \CuSgp{} of lower-semicontinuous functions $[0,1]\to\overline{\NN}$ (see, for example, \cite[Theorem~3.4]{AntPerSan11PullbacksCu}).

\begin{thm}\label{AI_Lsc}
 The Cuntz semigroup of an $AI$-algebra is $Cu$-isomorphic to the inductive limit of a system of the form $(\LscI^{k_i},\varphi_{i})$. Conversely, for every inductive system $(\LscI^{k_i},\varphi_{i})$ there exists an $AI$-algebra such that its Cuntz semigroup is $Cu$-isomorphic to the limit of the system.
\end{thm}
\end{pgr}

%====================================================
It is readily checked that finite sums of elements of the form
\[
 \chi_{(t,1]}, \quad 
 \chi_{(s,t)}, \quad 
 \chi_{[0,t)}, \andSep
 1
\]
are a basis for $\LscI$.

\begin{dfn}\label{basic_ind}
 An element in $\LscI$ will be called a \emph{basic indicator function} if it is of the form $\chi_{(t,1]}$, $\chi_{(s,t)}$, $\chi_{[0,t)}$ or $1$ for some $s,t$.
 
 Also, we will say that an element $f\in \LscI$ is \emph{basic} if it is the finite sum of basic indicator functions.
 
 Given any $M\in\mathbb{N}$, we say that an element in $\LscI^{M}$ is a \emph{basic indicator function} (resp. \emph{basic element}) if each of its components is a basic indicator function (resp. basic element).
\end{dfn}

\begin{rmk}\label{Free_basic_el}
 Let $F$ be the free abelian semigroup generated by the basic indicator functions (as symbols) on $\LscI$. Given two elements $g,h\in F$, we write $g\sim_0 h$ if and only if there exist $f\in F$ and two open intervals $U,V$ such that  $g= f+ \chi_{U}+\chi_{V}$ and $h=f+\chi_{U\cup V}+\chi_{U\cap V}$. We write $f\sim g$ if $f=g$, $f\sim_0 g$ or if $g\sim_0 f$.
 
 Let $\simeq $ be the transitive relation induced by $\sim$. Then, the quotient $F/\simeq$ is isomorphic to the monoid of basic elements in $\LscI$, since for any $f,g\in\LscI$ we have $f+g=f\vee g+f\wedge g$. Indeed, this is related to distributive lattice ordered semigroups, as defined in \cite[Definition~4.1]{Vila20}, and is generally true in $\Lsc (X,\overline{\mathbb{N}})$.
\end{rmk}

\begin{ntn}\label{retract_notation}
Given $\epsilon>0$ and $0\leq a<b\leq 1$, we define the $\epsilon$-retraction of the intervals $(a,b)$, $(a,1]$ and $[0,b)$ as 
 \[
  (a+\epsilon ,b-\epsilon ), \quad 
  (a+\epsilon ,1],\andSep
  [0,b-\epsilon )
 \]
respectively.

Given an interval $U$, we denote by $R_{\epsilon }(U)$ its $\epsilon $-retraction. Also, given a finite disjoint union of intervals, we define its $\epsilon $-retraction to be the finite disjoint union of the $\epsilon $-retractions of the intervals.

Given a basic indicator function $\chi_{U}$, we define $R_{\epsilon }(\chi_{U})=\chi_{R_{\epsilon }(U)}$. 
Whenever we do not need to specify $\epsilon $, we will simply write $R(\chi_{U})$. Also, if $\epsilon>0$ is such that $R_{\epsilon }(U)=\emptyset$, it will be understood that $R_{\epsilon }(\chi_{U})=0$.
\end{ntn}

%===============================================
\section{Lifting morphisms}\label{sec:LiftMor}

Given a group $G$ and an element $g\in G$, there always exists a group morphism $\mathbb{Z}\to G$ mapping $1$ to $g$. Even though this is trivial, it is a key feature in both Shen's and Effros-Handelman-Shen's theorems (see \cite{Shen79} and \cite{EffHanShe80} respectively).

In this section, we study an analogue of such a property in the category $\Cu$. Namely, given a \CuSgp{} $S$ and an element $s\in S$, we wish to prove that, under the right assumptions, there exists a \CuMor{} $\LscI\to S$ mapping $\chi_{(0,1]}$ to $s$. In fact, we will prove more:

\begin{thm}\label{lifting_th}
Let $S$ be a weakly cancellative $\Cu$-semigroup satisfying (O5). Then, given a finite $\ll$-increasing sequence $s_{1}\ll\cdots\ll s_{n}$ and a compact element $p$ such that $s_{n}\leq p$, there exists a $\Cu$-morphism $\phi\colon \LscI\to S$ such that $\phi (\chi_{(n-k/n,1]})=s_{k}$ and $\phi (1)=p$.
\end{thm}

Recall that in a locally small category an object $G$ is called a \emph{generator} if for every pair of morphisms $g,f\colon X\to Y$ there exists a morphism $h\colon G\to X$ such that $g\circ h\neq f\circ h$.

The proof of \autoref{lifting_th} will rely on the fact that the sub-$\Cu$-semigroup of $\LscI$ defined as $\mathcal{G}=\{  f\in \LscI\mid f(0)=0, \text{ $f$ increasing}\}$ is a generator for the category $\Cu$ (see \cite[Section 5.2]{Scho18}). In particular, \cite[Proposition 2.10]{AntPerThi20:AbsBivariantCu} and \cite[Lemma 5.16]{Scho18} imply that, given any \CuSgp{} $S$ and any finite $\ll$-increasing sequence $s_{1}\ll\cdots\ll s_{n}$, there always exists a $\Cu$-morphism $\mathcal{G}\to S$ mapping $\chi_{(n-k/n,1]}$ to $s_{k}$ for every $k$.

Thus, the theorem will follow if we can prove that certain $\Cu$-morphisms $\mathcal{G}\to S$ can be lifted to $\Cu$-morphisms $\LscI\to S$. This is done in \autoref{Lifting_Morphisms}.

Some remarks are in order:
\begin{rmk}
$ $
 \begin{enumerate}
  \item A $\Cu$-morphism from $\mathcal{G}$ to a \CuSgp{} $S$ may not always be liftable.
  
  For example, set $S=\LscI$. Take $\infty\in \LscI$ and let $\alpha\colon \mathcal{G}\to S$ be a $\Cu$-morphism with $\alpha(\chi_{(0,1]})=\infty$. Recall that such a map must exist because $\mathcal{G}$ is a generator.
  
  If such a map could be lifted, one would have 
  \[
   \infty=\alpha(\chi_{[0,1)})\ll \alpha (1)\leq\infty,
  \]
  which would imply that $\infty$ is compact, a contradiction.
  
  \item Even if a map can be lifted, the lift may not necessarily be unique. As an example, for every $n\in\NN$ define the $\Cu$-morphism $i_{n}\colon \LscI\to\LscI$ as 
  \[
   i_{n}(\chi_{[0,b)})=(n-1)+\chi_{[0,b)},\quad 
   i_{n}(\chi_{(a,1]})=\chi_{(a,1]},\quad
   i_{n}(\chi_{(a,b)})=\chi_{(a,b)},
  \]
  and $i_{n}(1)=n$.
  
  Then, $i_{n}|_{\mathcal{G}}=i_{n+1}|_{\mathcal{G}}$ for every $n$, so each $i_n$ is a lift of the inclusion $\mathcal{G}\to\LscI$.
 \end{enumerate}
\end{rmk}

Given a $\Cu$-semigroup $S$ satisfying weak cancellation, we know that for any compact element $p$ and any pair of elements $a,b\in S$ such that $p+a\leq p+b$, we have $a\leq b$. Indeed, taking $a'\ll a$, one has $p+a'\ll p+b$. Applying weak cancellation, we obtain $a'\ll b$ for any $a'\ll a$ and, consequently, $a\leq b$. 

In particular, this implies that if $p+a=p+b$, we must have $a=b$. This fact will be used repeatedly in the proof of the following lemma.

%===========================================

We will also consider the ordered set $\{\chi_{(t,1]} ,\chi_{[0,t)},1\}_{t\in [0,1]}$ with the order inherited from $\LscI$. Note that in this set suprema of increasing sequences always exist. 

Indeed, every increasing sequence $(s_{n})_{n}$ is of one of the following forms
\[
 s_{n}=1 \text{ for every } n ,\quad 
 s_{n}=\chi_{(a_{n},1]} \text{ with } a_{n+1}\leq a_{n} ,\orSep
 s_{n}=\chi_{[0,b_{n})}\text{ with }
 b_{n}\leq b_{n+1}\, ,
\]
with supremum $1$, $\chi_{(\inf (b_{n}),1]}$ and $\chi_{[0,\sup (a_{n}))}$ respectively. Note that all three elements belong to our set.

\begin{lma}\label{Lemma_Lifting_Tecnic}
 Let $S$ be a $\Cu$-semigroup satisfying $\mathrm{(O5)}$ and weak cancellation, and let 
 \[
 \phi\colon \{\chi_{(t,1]} ,\chi_{[0,t)},1\}_{t\in [0,1]}\to S
 \]
 be an order and suprema preserving map such that for every $t\leq s<t'$ we have
 \[
  \phi (\chi_{[0,t)})+\phi (\chi_{(s,1]})\leq \phi (1)\ll \phi (1)\leq 
  \phi (\chi_{[0,t')})+\phi (\chi_{(s,1]}).
 \]
 
 Then, there is a unique $\Cu$-morphism from $\LscI$ to $S$ lifting $\phi$. Conversely, for any $\Cu$-morphism $\phi$ from $\LscI$ to $S$, the previous   condition holds.
\end{lma}
\begin{proof}
 Necessity is clear, so we only need to prove the other implication. That is to say, we need to define a $\Cu$-morphism $\LscI\to S$, which we will also call $\phi$, extending the inital assignments. 
 
 First, note that $\phi (\chi_{(t,1]}),\phi(\chi_{[0,t)})\leq \phi (1)$ for every $t$. Also, given any $t'>t$ and $s'<s$, we can apply weak cancellation to the inequalities
 \[
  \begin{split}
   \phi (\chi_{[0,t)})+\phi (\chi_{(t,1]})\leq \phi (1)\ll \phi (1)\leq 
  \phi (\chi_{[0,t')})+\phi (\chi_{(t,1]}), \\
   \phi (\chi_{[0,s)})+\phi (\chi_{(s,1]})\leq \phi (1)\ll \phi (1)
   \leq 
   \phi (\chi_{[0,s)})+\phi (\chi_{(s',1]}),
  \end{split}
 \]
to obtain $\phi (\chi_{[0,t)})\ll \phi(\chi_{[0,t')})$ and $\phi (\chi_{(s,1]})\ll\phi(\chi_{(s',1]})$.
\vspace{0.2cm}

Let $s<t$. Then, since $\phi (1)\leq \phi (\chi_{(s,1]})+\phi (\chi_{[0,t)})$, we know by (O5) and using that $\phi (1)$ is compact that there exists an element $x\in S$ such that 
\[
 \phi (1)+x= \phi (\chi_{(s,1]})+\phi (\chi_{[0,t)}).
\]

Note that this element is unique by the remark preceding the lemma. Thus, we can define $\phi (\chi_{(s,t)})=x$.

Now let $s<t$ in $[0,1]$ and let $y\ll \phi (\chi_{(s,t)})$. Then, we have
\[
 \phi (1)+y\ll \phi(1)+\phi(\chi_{(s,t)})=
 \phi(\chi_{(s,1]}) + \phi(\chi_{[0,t)})
 \leq 
 \phi(\chi_{(s,1]})+\phi (1)\, , 
 \phi(\chi_{[0,t)})+\phi (1).
\]

By weak cancellation, one gets $y\ll \phi(\chi_{(s,1]}), \phi(\chi_{[0,t)})$. Taking the supremum on $y$, one obtains $\phi(\chi_{(s,t)})\leq  \phi(\chi_{(s,1]}), \phi(\chi_{[0,t)})$.

Also, given $s'<s<t<t'$, we have, using as proved above that $\phi (\chi_{(s,1]})\ll \phi (\chi_{(s',1]})$ and $\phi (\chi_{[0,t)})\ll \phi (\chi_{[0,t')})$,
\[
 \phi(1)+\phi(\chi_{(s,t)})=
  \phi(\chi_{(s,1]}) + \phi(\chi_{[0,t)}) \ll
  \phi(\chi_{(s',1]}) + \phi(\chi_{[0,t')})=
  \phi(1)+\phi(\chi_{(s',t')}).
\]

Applying weak cancellation once again, we get $\phi(\chi_{(s,t)})\ll \phi(\chi_{(s',t')})$.\vspace{0.2cm}

 Further, given any increasing sequence $t_{n}\to t$ and any decreasing sequence $s_{n}\to s$, we have 
 \[
  \begin{split}
   \phi(\chi_{(s,t)})+\phi (1)& =\phi (\chi_{[0,t)})+\phi (\chi_{(s,1]})=\sup_{n} (\phi (\chi_{[0,t_{n})})+\phi (\chi_{(s_{n},1]}))\\
   &= \sup_{n}(\phi (1)+\phi(\chi_{(s_{n},t_{n})}))=\phi (1)+\sup_{n}\phi(\chi_{(s_{n},t_{n})}).
  \end{split}
 \]
 
 By the remark preceding the lemma, one gets $\phi(\chi_{(s,t)})=\sup_{n}\phi(\chi_{(s_{n},t_{n})})$.
 \vspace{0.2cm}
 
 \noindent\textbf{Claim 1} \textit{Given any two open intervals $U,V$, one has $\phi (\chi_{U})+\phi (\chi_{V})=\phi (\chi_{U\cup V})+\phi (\chi_{U\cap V})$.}
 \begin{proof}
 Assume that $U,V$ are of the form $U=(s,t)$, $V=(s',t')$ with $s\leq s'\leq t\leq t'$.
 
 Then, one has
 \[
 \begin{split}
  \phi (\chi_{U})+\phi (\chi_{V})+2\phi (1)&=
  \phi (\chi_{[0,t)})+\phi (\chi_{(s,1]})+
  \phi (\chi_{[0,t')})+\phi (\chi_{(s',1]})\\
  &= \phi(\chi_{(s,t')}) + \phi (\chi_{(s',t)})+2\phi(1),
  \end{split}
 \]
which implies $\phi (\chi_{U})+\phi (\chi_{V})=\phi (\chi_{U\cup V})+\phi (\chi_{U\cap V})$, as required.

All the other cases are proven analogously.
\end{proof}
 
Let $B$ be the set of basic elements in $\LscI$. By \autoref{Free_basic_el} and Claim 1, we can now extend $\phi$ to a monoid morphism $\phi\colon B\to S$.

Since $\phi$ is additive and retractions  are applied to each connected component (see \autoref{retract_notation}), note that we have $\sup_{n}\phi (\chi_{R_{1/n}(U)})=\phi (\chi_{U})$ for every open subset $U\subset [0,1]$ that can be written as the finite disjoint union of intervals.
\vspace{0.2cm}

\noindent\textbf{Claim 2}
\textit{Given a finite strictly increasing sequence $s_{0}<\cdots <s_{n}$ in $[0,1]$, we have }
\[
\begin{split}
 \phi (\chi_{[0,s_{0}}))+\phi(\chi_{(s_{0},s_{1})})+\cdots +
 \phi(\chi_{(s_{n-1},s_{n})})+
 \phi(\chi_{(s_{n},1]})\leq \phi(1)\\
 \leq
 \phi (\chi_{[0,s_{1}}))+\phi(\chi_{(s_{0},s_{2})})+\cdots +
 \phi(\chi_{(s_{n-2},s_{n})})+
 \phi(\chi_{(s_{n-1},1]})
 \end{split}
\]

\textit{In particular, given $U, V$ open subsets that can be written as the finite disjoint union of open intervals, we have $\phi (\chi_{U})\leq\phi (\chi_{V})$ whenever $U\subset V$ and $\phi (\chi_{U})\ll\phi (\chi_{V})$ whenever $U\Subset V$.}
\begin{proof}
Using the equality $\phi (\chi_{(s,t)})+\phi (1)=\phi (\chi_{[0,t)})+\phi (\chi_{(s,1]})$ at the first step, and the inequality $\phi (\chi_{[0,s)})+\phi (\chi_{(s,1]})\leq \phi (1)$ at the second step, we get
\[
 \begin{split}
  \phi (\chi_{[0,s_{0}}))+\phi(\chi_{(s_{0},s_{1})})+\cdots +
 \phi(\chi_{(s_{n-1},s_{n})})+
 \phi(\chi_{(s_{n},1]})
 +n\phi (1)\\
 =\sum_{i=0}^{n}
 \phi (\chi_{[0,s_{i})})+\phi (\chi_{(s_{i},1]})
 \leq (n+1)\phi (1).
 \end{split}
\]

Similarly, but now using $\phi (\chi_{[0,t)})+\phi (\chi_{(s,1]})\geq \phi (1)$ for every $s<t$ at the second step, we have
\[
\begin{split}
 \phi (\chi_{[0,s_{1}}))+\phi(\chi_{(s_{0},s_{2})})+\cdots +
 \phi(\chi_{(s_{n-2},s_{n})})+
 \phi(\chi_{(s_{n-1},1]})+(n-1)\phi (1)\\
 =
 \sum_{i=1}^{n}
 \phi (\chi_{[0,s_{i+1})})+\phi (\chi_{(s_{i}-1,1]})\geq n\phi (1).
 \end{split}
\]

Since $\phi (1)$ is compact, we can apply  weak cancellation to cancel $n\phi (1)$ in the first inequality and $(n-1)\phi (1)$ in the second. This proves the first part of the claim.

Given a subset $Y\subset [0,1]$, let $\Int (Y)$ denote its interior. Then, given $U\Subset V$ as in the statement, and using the result above in both inequalities, one has 
\[
 \phi (\chi_{U})+\phi (\chi_{\Int ([0,1]-U)})\leq \phi (1)\leq 
 \phi (\chi_{V})+\phi (\chi_{\Int ([0,1]-U)}).
\]

Applying weak cancellation once again, we have $\phi (\chi_{U})\ll\phi (\chi_{V})$.

If $U\subset V$, take $n >0$ and note that $R_{1/n }(U)\Subset U\subset V$. Thus, we have $\phi (\chi_{R_{1/n }(U)})\ll \phi (\chi_{V})$. Taking the supremum on $n$ one obtains the required result.
\end{proof}

Now let $f,g\in B$, and note that $f\leq g$ (resp. $f\ll g$) if and only if $\{f\geq k\}\subset \{g\geq k\}$ for every $k\leq \sup (g)$ (resp. $\{f\geq k\}\Subset \{g\geq k\}$). Since $\phi$ is additive, it follows from Claim 2 that $\phi $ preserves both the order and the $\ll$-relation.

Take $f\in B$, which can be written as $f=\sum_{i} \chi_{U_{i}}$ with $U_{i}$ finite disjoint unions of intervals, and let $f_{n}$ be $\sum_{i} \chi_{R_{1/n}(U_{i})}$, which is still an element in $B$ with $f_n\ll f_{n+1}$ for all $n$. Clearly, since $\phi$ is additive and $\sup_{n}f_{n}=f$, we have that $\phi (f_{n})$ is an $\ll$-increasing sequence with supremum $\phi (f)$.

Thus, given an increasing sequence $(f'_{m})_m$ in $B$ with supremum $f$, we have that for every $n$ there exists some $m$ with
\[
 f_{n}\leq f'_{m}\leq f.
\]

Since $\phi$ preserves the order, one gets $\phi (f_{n})\leq \phi (f'_{m})\leq \phi (f)$. Taking suprema, it follows that $\sup_{m}\phi (f'_{m})=\phi (f)$ whenever $f'_{m}$ is an increasing sequence with supremum $f$.

Finally, since every element in $\LscI$ can be written as the supremum of an $\ll$-increasing sequence of elements in $B$, we can define $\phi\colon\LscI\to S$ as $\phi (g)=\sup_{n}\phi (g_{n})$ for $g_{n}$ $\ll$-increasing sequence in $B$ with supremum $g$.

It is now easy to check that $\phi $ is a $\Cu$-morphism (see, for example, the proof of  \cite[Theorem~4.40]{Vila20}). Note that it is a lift for our previously defined map $\phi\colon B\to S$ because we already know that $\sup_{m}\phi (f'_{m})=\phi (f)$ whenever $f'_{m}$ is an increasing sequence with supremum $f$.
\end{proof}
%===========================================

Recall that $\mathcal{G}$ is the sub-$\Cu$-semigroup of $\LscI$ defined as
\[\mathcal{G}=\{  f\in \LscI\mid f(0)=0, \text{ $f$ increasing}\}.
 \]

\begin{prp}\label{Lifting_Morphisms}
Let $S$ be a $\Cu$-semigroup satisfying (O5) and weak cancellation, and let $\alpha\colon \mathcal{G} \to S$ be a $\Cu$-morphism such that $\alpha (\chi_{(0,1]})\leq p$ with $p$ a compact element in $S$. 

 Then, there exists a unique $\Cu$-morphism $\phi\colon \LscI\to S$ extending $\alpha$ such that $\phi (1)=p$.
\end{prp}
\begin{proof}
We will show that there exists a map
 \[
 \phi\colon \{\chi_{(t,1]} ,\chi_{[0,t)},1\}_{t\in [0,1]}\to S
 \]
 satisfying the conditions in  \autoref{Lemma_Lifting_Tecnic}. The result will then follow from the application of this lemma.
 \vspace{0.1cm}
 
 First, set $\phi (1)=p$ and $\phi (\chi_{(s,1]})=\alpha (\chi_{(s,1]})$ for every $s$.
 
 Take $t\in (0,1]$ and consider a strictly increasing sequence $s_{n}\to t$. Since $\phi (\chi_{(s_{n+1},1]})\ll \phi (\chi_{(s_{n},1]})\ll \phi (1)$ for every $n$, we know by (O5) that there exist elements $x_{n}$ such that
 \[
  \phi (\chi_{(s_{n+1},1]}) + x_{n} \leq \phi(1) \leq \phi (\chi_{(s_{n},1]}) + x_{n}
 \]
 for every $n$. 
 
 Using $\phi(1)\ll \phi(1)$ and weak cancellation, one gets from
 \[
  \phi (\chi_{(s_{n+1},1]})+x_n\leq\phi (1)\leq \phi (\chi_{(s_{n+1},1]})+x_{n+1}
 \]
that $x_{n}\ll x_{n+1}$. Let $x=\sup x_{n}$.
 
 Now take $t'\geq t$, and consider an stricly increasing sequence $s'_{n}\to t'$. Using the same argument, we obtain an associated $\ll$-increasing sequence $(y_{n})_{n}$.
 
 Further, note that for every $n$ there must exist some $m\geq n$ such that $s'_{m}\geq s_{n+1}$. Using this at the third step we obtain
 \[
  \phi (\chi_{(s_{n+1},1]})+x_{n}\leq \phi (1)\leq
  \phi (\chi_{(s'_{m},1]})+y_{m}\leq 
  \phi (\chi_{(s_{n+1},1]}) +y_{m}.
 \]

 Applying weak cancellation and taking suprema over $m$, one obtains $x\leq \sup_{m} y_{m}$.
 
 This argument shows that we can define $\phi (\chi_{[0,t)})$ as $x$ (by taking $t'=t$), and that $\phi (\chi_{[0,t)})\leq \phi (\chi_{[0,t')})$ whenever $t'\geq t$. Thus, $\phi$ is order preserving.
 
 With the previous notation, note that, for any $t'<t$, we get
 \[
  \phi (\chi_{(t,1]}) + x_{n}\leq 
  \phi (\chi_{(s_{n+1},1]}) + x_{n}\leq 
  \phi (1)\leq 
  \phi (\chi_{(s_{n},1]}) + x_{n}\leq 
  \phi (\chi_{(t',1]})+x_{n},
 \]
 for some large enough $n$ such that $s_{n}>t'$.
 
This shows that $\phi (\chi_{(t,1]})+\phi (\chi_{[0,t)})\leq \phi (1)\leq \phi (\chi_{(t',1]})+\phi (\chi_{[0,t)})$ whenever $t'<t$.

Now take an increasing sequence $t_{m}\to t$. In order to show that $\sup_{m}\phi (\chi_{[0,t_{m})})=\phi (\chi_{[0,t)})$, we only need to prove  $\sup_{m}\phi (\chi_{[0,t_{m})})\geq \phi (\chi_{[0,t)})$, as we already know that $\phi $ is order preserving.

Using the same notation as above, let $s_{n}\to t$ be a strictly increasing sequence and let $x_{n}\in S$ be their associated elements. Since $s_{n}$ is strictly increasing, we must have that for every $n$ there exists $m\geq n$ such that $s_{n+1}< t_{m}$. Using that $\phi (1)\leq \phi (\chi_{(t',1]})+\phi (\chi_{[0,t)})$ whenever $t'<t$ in the second inequality, we get 
\[
 \phi (\chi_{(s_{n+1},1]})+x_{n}\leq \phi (1)\leq
  \phi (\chi_{(s_{n+1},1]})+\phi (\chi_{[0,t_{m})}).
\]

By weak cancellation and taking suprema, we have $\sup_{m}\phi (\chi_{[0,t_{m})})\geq \sup_{n} x_{n}=\phi (\chi_{[0,t)})$. This shows that $\phi $ also preserves suprema and, consequently, that \autoref{Lemma_Lifting_Tecnic} can be applied.

We now get a (unique) $\Cu$-morphism $\phi\colon \LscI\to S$ extending our map $\phi$. Since, by construction, $\phi (\chi_{(s,1]})=\alpha (\chi_{(s,1]})$ for every $s$ and the submonoid of $\mathcal{G}$ generated by $\{\chi_{(s,1]}\}_s$ is sup-dense in $\mathcal{G}$, it follows that $\phi$ extends $\alpha$ as desired.
\end{proof}

%===========================================

\section{Cauchy sequences and a local characterization}\label{sec:Cauchy}

We now turn our attention to the proof of the main result of this paper, \autoref{Main_EHS}, and its discrete counterpart, \autoref{MainTh}. Since countably based $\Cu$-semigroups are usually far bigger than countably generated groups (for instance, a countably based $\Cu$-semigroup is not generally countable), one cannot hope to get an exact analogue of \cite[Theorem 3.1]{Shen79}. Instead, one has to make do with an approximate version of it.

This is why we define a distance on the morphisms from $\LscI$ to a \CuSgp{} $S$.
%===========================================

\subsection{Distance between maps}

Let $\varphi_{1},\varphi_{2}\colon \LscI\to S$ be $\Cu$-morphisms with $\varphi_{1}(1)=\varphi_{2}(1)$. We  define the distance between $\varphi_1$ and $\varphi_2$ as
 \[
 d(\varphi_{1} ,\varphi_{2}):=\inf \{ \epsilon\in [0,1]\mid 
 \varphi_{1}(\chi_{(t+\epsilon ,1]})\leq  \varphi_{2}(\chi_{(t,1]})\,
  , \,  \varphi_{2}(\chi_{(t+\epsilon ,1]})\leq  \varphi_{1}(\chi_{(t,1]}) \quad \forall t\in [0,1]
 \}.
\]

Note that the distance between $\varphi_{1},\varphi_{2}$ is the distance between $\varphi_{1}|_{\Lsc((0,1],\overline{\NN}) },\varphi_{2}|_{\Lsc((0,1],\overline{\NN}) }$ considered in \cite{CiupElli08} and \cite{CiupElliSant11} (see also \cite{RobSan10}).
%===========================================

\begin{rmk}\label{rmk_dist_first}
 Let $s,t\in [0,1]$ with $s-t > d(\varphi_{1},\varphi_{2})$. Then, we have
 \[
  \varphi_{1}(\chi_{(s,1]})\ll  \varphi_{2}(\chi_{(t,1]})
  , \andSep  \varphi_{2}(\chi_{(s,1]})\ll  \varphi_{1}(\chi_{(t,1]}).
 \]
 
 Indeed, given $\eta >0$ such that $s-t> \eta >d(\varphi_1, \varphi_2)$, we have
 \[
  \varphi_{1}(\chi_{(s,1]})\ll \varphi_{1}(\chi_{(t+\eta ,1]})\leq \varphi_{2}(\chi_{(t ,1]}),\andSep
  \varphi_{2}(\chi_{(s,1]})\ll \varphi_{2}(\chi_{(t+\eta ,1]})\leq \varphi_{1}(\chi_{(t ,1]}).
 \]

 This remark will be used throughout the section.
\end{rmk}
%===========================================

Under the hypothesis of weak cancellation, Ciuperca and Elliott proved in \cite[Theorem 4.1]{CiupElli08} that their distance is a metric. To see that our distance is also a metric, we recall the following result:

\begin{lma}
 Let $S$ be a $\Cu$-semigroup satisfying weak cancellation and (O5). If two $\Cu$-morphisms $\varphi_{1}$, $\varphi_{2}$ from $\LscI$ to $S$ agree on $\Lsc((0,1],\overline{\NN})$ and $1$, they are the same.
\end{lma}
\begin{proof}
 If $\varphi_{1},\varphi_{2}$ agree on  $\Lsc((0,1],\overline{\NN})$, then they also agree on the sub-$\Cu$-semigroup $\mathcal{G}$ consisting of increasing lower-semicontinuous functions.
 
 Thus, since $\varphi_{1}(1)= \varphi_{2}(1)$, we know from \autoref{Lifting_Morphisms} that there exists a unique lift for $\varphi_{1}|_{\mathcal{G}}$ sending $1$ to $\varphi_{1}(1)$, which, by uniqueness, must be the same for $\varphi_{2}|_{\mathcal{G}}$ that sends $1$ to $\varphi_{2}(1)$.
 
 By \autoref{Lifting_Morphisms}, this shows $\varphi_{1}=\varphi_{2}$.
\end{proof}
%===========================================

Following \cite{CiupElliSant11}, one can generalize  the previous notion of distance to pairs of morphisms from finite direct sums of $\LscI$ to $S$.

\begin{dfn}\label{General_Dist}
 Let $S$ be a $\Cu$-semigroup. Let $L=\oplus_{i=1}^{r} L_{i}$ with $L_{i}=\LscI$ for each $i$, and consider a pair of morphisms $\varphi_{1},\varphi_{2}\colon L\to S$ with $\varphi_{1}(1_{i})= \varphi_{2}(1_{i})$ for each $i\leq r$.
 
 We define the distance between $\varphi_1$ and $\varphi_2$ as 
 \[
  d(\varphi_{1} ,\varphi_{2}):=\sup_{1\leq i\leq r}d(\varphi_{1}|_{L_{i}} ,\varphi_{2}|_{L_{i}}) .
 \]
\end{dfn}

\begin{rmk}
 Note that $d(\cdot ,\cdot )$ is clearly a metric as well.
\end{rmk}
%===========================================

\begin{lma}\label{basic_ind_dist}
Let $\varphi_{1},\varphi_{2}\colon \LscI \to S$ with $S$ a weakly cancellative $\Cu$-semigroup. If $d(\varphi_{1},\varphi_{2})\leq \epsilon$ with $\varphi_{1}(1)=\varphi_{2}(1)$, then for every basic indicator function $f\in\LscI$, as defined in \autoref{basic_ind}, we have
  \[
   \varphi_{1}(R_{\epsilon}(f))\leq  \varphi_{2}(f), \andSep  \varphi_{2}(R_{\epsilon}(f))\leq  \varphi_{1}(f),
 \]
  where $R_{\epsilon}$ denotes the $\epsilon$-retraction of $f$; see \autoref{retract_notation}.
\end{lma}
\begin{proof}
 We will only prove one of the inequalities, since the other one is proved analogously.

 If $f=1$ or $f=\chi_{(s,1]}$ for some $s\in [0,1]$, the result follows from the definition of distance.
 
 Else, if $f=\chi_{(s,t)}$ for some $s<t$, let $\eta >0$ and note that
 \[
 R_{\epsilon+\eta}(f)+\chi_{(t-\epsilon-\eta,1]}
  = \chi_{(s+\eta+\epsilon ,t-\eta -\epsilon)}+\chi_{(t-\eta-\epsilon,1]}\ll \chi_{(s+\epsilon ,1]}=
  R_{\epsilon }(\chi_{(s,1]}).
 \]

 Therefore, using that $\varphi_1,\varphi_2$ are $\Cu$-morphisms at the first and third steps and the assumption that $d(\varphi_1,\varphi_2)\leq \epsilon$ at the second and fourth steps, we get
 \begin{equation*}
 \begin{split}
   \varphi_{1}(R_{\epsilon+\eta}(f))+ \varphi_{1}(\chi_{(t-\epsilon-\eta,1]}) &\ll \varphi_1 (R_{\epsilon}(\chi_{(s,1]})\leq 
   \varphi_{2}(\chi_{(s,1]})\\
  & \leq 
  \varphi_{2}(\chi_{(s,t)}) +
  \varphi_{2}(\chi_{(t-\eta ,1]})\leq 
  \varphi_{2}(f) +
  \varphi_{1}(\chi_{(t-\epsilon -\eta,1]}).
 \end{split}
 \end{equation*}
 
Applying weak cancellation we get $\varphi_{1}(R_{\epsilon+\eta}(f))\ll \varphi_{2}(f)$. Since $\sup_{\eta\to 0} R_{\epsilon+\eta}(f)=R_{\epsilon }(f)$, it follows that $\varphi_{1}(R_{\epsilon }(f))\leq \varphi_{2}(f)$ as required.

Finally, if $f=\chi_{[0,t)}$ for some $t\in [0,1]$, let $\eta>0$ and note that, arguing as above and using $\varphi_1(1)=\varphi_2(1)$,
\[
\begin{split}
 \varphi_{1}(R_{\epsilon +\eta}(f) ) +
 \varphi_{2}(\chi_{(t-\eta ,1]}) &\leq 
 \varphi_{1}(R_{\epsilon +\eta}(f) ) +
 \varphi_{1}(\chi_{(t-\epsilon-\eta ,1]})\ll 
 \varphi_{1}(1)\\
 &=\varphi_{2}(1)\leq \varphi_{2}(\chi_{[0,t)}) + \varphi_{2}(\chi_{(t-\epsilon ,1]})= 
 \varphi_2(f)+\varphi_{2}(\chi_{(t-\epsilon ,1]})
\end{split}
\]

As before, we get $\varphi_1 (R_{\epsilon +\eta }(f))\ll \varphi_2 (f)$ for every $\eta >0$ and, consequently, $\varphi_1 (R_{\epsilon  }(f))\leq \varphi_2 (f)$.
\end{proof}
%===========================================

\subsection{Cauchy sequences and their limits}
 
In this section we will prove that suitable Cauchy sequences of $\Cu$-morphisms $\varphi_{i}\colon \LscI\to S$ have a limit. From this point onwards, and until the end of the section, $S$ will denote a \CuSgp{} satisfying (O5) and weak cancellation.

The following lemmas will be of importance:
%===========================================

\begin{lma}\label{Bound}
 Suppose that $f'\ll f$ in $\LscI$ and let $\phi\colon\LscI\to S$ be a $\Cu$-morphism. Then, there exists $\epsilon=\epsilon (f,f') >0$ and $f''\in \LscI$ such that 
 \begin{enumerate}[(i)]
  \item $f'\ll f''\ll f$.
  \item For every $\Cu$-morphism $\varphi\colon \LscI\to S$ with $\phi (1)=\varphi (1)$ and $d(\phi , \varphi)<\epsilon$, we have $\phi (f')\ll \varphi (f'')\ll \phi (f)$.
 \end{enumerate}
\end{lma}
\begin{proof}
 First, let us prove the result when $f',f$ are basic indicator functions. If $f'=f=1$, the result follows trivially.
 
 If $f'=\chi_{(t,1]}$ and $f=\chi_{(s,1]}$ with $t>s$, let $\epsilon >0$ be such that $s<s+2\epsilon <t$. Then, if $d(\phi, \varphi)<\epsilon$, one has, using \autoref{rmk_dist_first} at the last step,
 \[
  \phi (\chi_{(t,1]})\ll \phi (\chi_{(s+2\epsilon ,1]})\leq \varphi (\chi_{(s+\epsilon ,1]})\ll \phi
  (\chi_{(s,1]}),
 \]
so we can set $f''=\chi_{(s+\epsilon ,1]}$. 

If $f=\chi_{[0,t)}$ and $f'=\chi_{[0,s)}$ with $t>s$, we can take $\epsilon $ as before to get, if $d(\phi,\varphi )<\epsilon$ and using \autoref{basic_ind_dist} at the first step, that
\[
 \phi (\chi_{[0,s)})\ll \varphi (\chi_{[0,s+\epsilon )})\ll \phi (\chi_{[0,s+2\epsilon )})\ll \phi (\chi_{[0,t)})
\]
whenever $d(\phi ,\varphi)<\epsilon$. This shows that we can set $f''=\chi_{[0,s+\epsilon )}$.

Finally, given $f=\chi_{(s,t)}$ and $f'=\chi_{(s',t')}$ with $s<s'<t'<t$, let $\epsilon >0$ be such that $s<s'-2\epsilon $ and $t'<t'+2\epsilon <t$. Then, if $d(\phi,\varphi )<\epsilon$, we have
\[
\begin{split}
 \phi (1) + \phi (\chi_{(s',t')})
 &=\phi (\chi_{(s',1]})+\phi (\chi_{[0,t')})
 \ll
 \varphi (\chi_{(s'-\epsilon ,1]})+\varphi (\chi_{[0,t'+\epsilon )})\\
 &=\varphi (1) + \varphi (\chi_{(s'-\epsilon,t'+\epsilon )})
 =\varphi (\chi_{(s'-\epsilon,1]})
 +\varphi (\chi_{[0,t'+\epsilon )})
 \\
 &\ll 
 \phi (\chi_{(s'-2\epsilon ,1]})+\phi (\chi_{[0,t'+2\epsilon )})\ll 
 \phi (\chi_{(s,1]})+\phi (\chi_{[0,t)})=\phi (1) + \phi (\chi_{(s,t)}).
\end{split}
\]

In particular, one gets
\[
 \phi (1) + \phi (f')\ll 
 \varphi (1) + \varphi (\chi_{(s'-\epsilon,t'+\epsilon )})\ll 
 \phi (1) + \phi (f).
\]

Applying weak cancellation, we have
\[
 \phi (f')\ll 
 \varphi (\chi_{(s'-\epsilon,t'+\epsilon)})\ll
 \phi (f)
\]
and in this case we set  $f''=\chi_{(s'-\epsilon,t'+\epsilon)}$.

Now, given any pair $f'\ll f$, let $g,g'\in\LscI$ be such that $f'\ll g'\ll g\ll f$ with $g'=\sum_{i=1}^{n} h_{i}'$, $g=\sum_{i=1}^{n} h_{i}$ and $h_{i}'\ll h_{i}$ basic indicators for each $i$. This can be done because the finite sums of basic indicators are a basis for $\LscI$.

Since $h_i'\ll h_i$, it follows from our previous case that there exist $\epsilon_i =\epsilon_i (h_i ,h_i' )$ and $h_i''$ with $h_i'\ll h_i''\ll h_i$ and such that, whenever $d (\varphi ,\phi )<\epsilon_i $ and $\varphi (1)=\phi (1)$, we have
\[
 \phi (h_i')\ll \varphi (h_i'')\ll \phi (h_i).
\]

Let $\epsilon=\epsilon (f,f')=\min (\epsilon_i )$ and set $f''=\sum_i h_i''$.

Now let $\varphi$ be a $\Cu$-morphism such that $d(\varphi ,\phi )<\epsilon $ and $\phi (1)=\varphi (1)$. Since $d(\varphi ,\phi )<\epsilon \leq \epsilon_i$ for every $i$, we have
\[
 \phi (h_i')\ll \varphi (h_i'')\ll \phi (h_i)
\]
for every $i$.

Thus, one gets
\[
 \phi(f')\ll \phi(g')=\sum_i\phi (h_i')\ll \sum_i\varphi(h_i'')=\varphi (f'')\ll \sum_i\phi (h_i)=\phi (g)\ll \phi (f)
\]
as required.
\end{proof}
%===========================================

\begin{cor}\label{WB_Maps}
 Let $f,g,h\in\LscI$ and let $\phi\colon\LscI\to S$ be such that $f\ll h$ and $\phi (h)\ll \phi (g)$. Then, there exists $\epsilon =\epsilon (f,g,h)>0$ such that, for every $\varphi\colon \LscI\to S$ with $\phi (1)=\varphi (1)$ and $d(\phi , \varphi)<\epsilon $, we have $\varphi (f)\ll \varphi (g)$.
 
 In fact, there exists $f''$ such that $f\ll f''$ and $\varphi (f'')\ll \varphi (g)$.
\end{cor}
\begin{proof}
 Let $\epsilon (h,f)>0$ and $f''\in S$ be the number and element given by \autoref{Bound} applied to $f\ll h$. Also, for every $g'\ll g$, let $\epsilon (g,g')>0$ and $g''$ be a choice of number and element given by \autoref{Bound} applied to $g'\ll g$.
 
 Set $\epsilon = \epsilon (f,g,h)=\min (\epsilon (h,f) \, , \,  \sup\{ \epsilon (g,g')\, \mid \, \phi (h)\ll \phi (g') \})$.
 
 Take a $\Cu$-morphism $\varphi\colon\LscI\to S$ with $\phi (1)=\varphi (1)$ and $d(\phi ,\varphi )<\epsilon $. Thus, there exists an element $g'\ll g$ with $\phi (h)\ll \phi (g')$ such that $d(\phi ,\varphi )<\epsilon (g,g'),\epsilon (h,f)$.
 
 Using \autoref{Bound} at the first and third step, we have
 \[
  \varphi (f'')\ll \phi (h)\ll \phi (g')\ll \varphi (g'')\ll \varphi (g)
 \]
as desired. 

Note, in particular, that $\varphi (f)\ll \varphi (g)$ because $f\ll f''$.
\end{proof}
%===========================================

\begin{lma}\label{Partition}
 Let $n\in\NN$ and $\epsilon\in (0,1]$. Also, let $t_{i}=i/n$ be a partition of $[0,1]$ and consider two \CuMor{s} $\varphi_{1},\varphi_{2}\colon \LscI \to S$ such that $\varphi_{1}(1)=\varphi_{2}(1)$,  
 \begin{equation*}
  \varphi_{1}(\chi_{(t_{i}+\epsilon ,1]})\leq \varphi_{2}(\chi_{(t_{i},1]}),
  \andSep \varphi_{2}(\chi_{(t_{i}+\epsilon ,1]})\leq \varphi_{1}(\chi_{(t_{i},1]})
 \end{equation*}
for every $0\leq i\leq n$.

Then, $d(\varphi_{1},\varphi_{2})\leq \epsilon + 1/n$.
\end{lma}
\begin{proof}
 Let $t\in [0,1]$ and take $i$ minimal such that $t\leq t_{i}$. Thus, $t_{i}\leq t+1/n$.
 
 We have, using our assumptions and that $\varphi_i$ are $\Cu$-morphisms, 
 \[
  \varphi_{1}(\chi_{(t+1/n+\epsilon ,1]})\leq 
  \varphi_{1}(\chi_{(t_{i}+\epsilon ,1]})\leq
  \varphi_{2}(\chi_{(t_{i},1]}) \leq
  \varphi_{2}(\chi_{(t,1]}).
 \]
 
 By an analogous argument, one sees that $\varphi_{2}(\chi_{(t+1/n+\epsilon ,1]})\leq \varphi_{1}(\chi_{(t,1]})$ as required.
\end{proof}
%===========================================

\begin{prp}\label{Finite}
 Let $\epsilon >0$ and let $\phi\colon \LscI\to \LscI$ be a $\Cu$-morphism. There exists $\epsilon '>0$  such that, for any pair of $\Cu$-morphisms $\varphi_{1},\varphi_{2}\colon \LscI\to S$ with $\varphi_{1}(1)=\varphi_{2}(1)$ and at distance at most $\epsilon '$, we have
 \[
  d(\varphi_1\phi ,\varphi_2\phi)<\epsilon .
 \]

\end{prp}
 %Let $t,r>0$ and consider $\phi\colon \LscI\to \LscI$. There exists $r'>0$ such that, for any pair of maps $\varphi_{1},\varphi_{2}\colon \LscI\to S$ at distance at most $r'$ with $\varphi_{1}(1)= \varphi_{2}(1)$, we have 
 %\[
 % \varphi_{1}\phi (\chi_{(t+r,1]})\leq
 % \varphi_{2}\phi (\chi_{(t,1]})\quad , \quad
 % \varphi_{2}\phi (\chi_{(t+r,1]})%\leq
 % \varphi_{1}\phi (\chi_{(t,1]})
 %\]
 %The same proof applies for a finite number of pairs $(t_{n},r_{n})$.
\begin{proof}
We begin the proof with the following claim:

\noindent\textbf{Claim. }\textit{
 For every $t\in [0,1]$ there exists $\epsilon (t)>0$ such that, for every pair of $\Cu$-morphisms $\varphi_{1},\varphi_{2}\colon \LscI\to S$ with $\varphi_{1}(1)=\varphi_{2}(1)$ and at distance at most $\epsilon (t)$, we have
 \[
 \varphi_1 (\phi (\chi_{(t+\epsilon /2 ,1]}))\leq \varphi_2(\phi (\chi_{(t,1]} )) ,\andSep \varphi_2 (\phi (\chi_{(t+\epsilon /2 ,1]}))\leq \varphi_1(\phi (\chi_{(t,1]} )).
 \]}\vspace{-0.5cm}
\begin{proof}
 Let $t\in [0,1]$ and consider the basic indicators $\chi_{(t,1]}$ and $\chi_{(t+\epsilon /2 ,1]}$ of $\LscI$. Since $\chi_{(t,1]},\chi_{(t+\epsilon /2 ,1]}\ll 1$, it follows that $\phi (\chi_{(t,1]}),\phi (\chi_{(t+\epsilon /2 ,1]})\leq m $ for some $m\in \NN$.
  
 Thus, since we also know that $\phi (\chi_{(t+\epsilon /2,1]})\ll \phi (\chi_{(t,1]})$, there exist functions $f_i,g_i\in \LscI$ with $f_i\ll g_i\ll 1$ such that $\phi (\chi_{(t+\epsilon /2,1]})=\sum_{i\leq m} f_i$ and $\phi (\chi_{(t,1]})=\sum_{i\leq m} g_i$ (see, for example, \cite[Lemma~4.19]{Vila20}).

 Also, since $f_i\ll g_i\ll 1$ for every $i$, there exist basic indicators $a_{i,j}$ such that
 \[
  f_i\ll a_{i,1}+\cdots +a_{i,k(i)}\ll g_i
 \]
for every $i$.

Therefore, for every $i\leq m$ there exists $\delta_i$ such that
\[
 f_i\ll 
 R_{\delta_i}(a_{i,1})+\cdots +R_{\delta_i}(a_{i,k(i)})
 \ll a_{i,1}+\cdots +a_{i,k(i)}\ll g_i.
\]

Set $\epsilon (t) =\min (\delta_i) $. By \autoref{basic_ind_dist}, for every pair of $\Cu$-morphisms $\varphi_{1},\varphi_{2}$ with $\varphi_{1}(1)=\varphi_{2}(1)$ and at distance at most $\epsilon (t)$, we have 
\[
 \varphi_1 (f_i)\leq \varphi_1 (R_{\delta_i}(a_{i,1}))+\cdots +\varphi_1(R_{\delta_i}(a_{i,k(i)}))\leq 
 \varphi_2 (a_{i,1})+\cdots +\varphi_2 (a_{i,k(i)})\ll 
 \varphi_2 (g_i).
\]

 Adding these inequalities, one gets
 \[
  \varphi_1 (\phi (\chi_{(t+\epsilon /2 ,1]}))\leq \varphi_2(\phi (\chi_{(t,1]} ))
 \]
 and, by an analogous argument, we also have $\varphi_2 (\phi (\chi_{(t+\epsilon /2,1]}))\leq \varphi_1(\phi (\chi_{(t,1]} ))$ as required.
 \end{proof}
 
 Let $n\in \NN$ be such that $1/n\leq \epsilon /2$, and consider the partition $t_i =i/n$ of $[0,1]$.
 
 By the claim, for every $i$ there exists $\epsilon (t_i)$ such that
 \[
 \varphi_1 (\phi (\chi_{(t_i+\epsilon /2 ,1]}))\leq \varphi_2(\phi (\chi_{(t_i,1]} ))\, ,\, \varphi_2 (\phi (\chi_{(t_i+\epsilon /2 ,1]}))\leq \varphi_1(\phi (\chi_{(t_i,1]} ))
 \]
 whenever $d(\varphi_1 \phi ,\varphi_2\phi)<\epsilon (t_i)$.
 
 Define $\epsilon '=\min (\epsilon (t_i))$. Then, by \autoref{Partition}, we get $d(\varphi_1\phi,\varphi_2\phi )\leq \epsilon /2+1/n\leq \epsilon$ as desired.
\end{proof}
%===========================================

\begin{cor}\label{AI_Rmk}
 Let $L=\oplus_{i=1}^{r}\LscI$ and $L'=\oplus_{j=1}^{s}\LscI$. Then, for any $\epsilon >0$ and $\phi\colon L\to L'$, there exists $\epsilon '>0$ such that, for any pair of morphisms $\varphi_{1},\varphi_{2}\colon L'\to S$ with $d(\varphi_{1},\varphi_{2})<\epsilon '$ and $\varphi_{1}(1_{j})= \varphi_{2}(1_{j})$ for every $j$, we have
 \[
  d(\varphi_{1}\phi ,\varphi_{2}\phi )<\epsilon .
 \]
\end{cor}
%===========================================

Now let $\varphi_{i}\colon \LscI\to S$ be such that $d(\varphi_{i},\varphi_{i+1})< \epsilon_{i}$ with $\epsilon_{i}$ strictly decreasing and $R:=\sum_{i=1}^{\infty} \epsilon_{i}<\infty$. Also, assume that $\varphi_{i}(1)=\varphi_{i+1}(1)$ for every $i$.
\vspace{0.1cm}

 Define $R_{i}:=\sum_{k=1}^{i}\epsilon_{k}$ and let $t\in [0,1]$. Using \autoref{rmk_dist_first} we have
 \begin{equation*}
  \varphi_{1} (\chi_{ (t+R,1] })\ll
  \varphi_{2} (\chi_{ (t+R-R_{1},1] })\ll
  \varphi_{3} (\chi_{ (t+R-R_{2},1] })\ll
  \cdots
  \ll
  \varphi_{i+1} (\chi_{ (t+R-R_{i},1] })\ll
  \cdots 
 \end{equation*}
 because $(t+R-R_{i}) - (t+R-R_{i+1})=\epsilon_{i}>d(\varphi_{i},\varphi_{i+1})$.\vspace{0.1cm}

 Thus, the sequence $( \varphi_{i+1} (\chi_{ (t+R-R_{i},1] }) )_{i}$ is $\ll$-increasing, and we can consider its supremum. For each $t$, we define $\varphi (\chi_{(t,1]}):= \sup_i \varphi_{i+1} (\chi_{ (t+R-R_{i},1] })$.
 \vspace{0.3cm}
 %==========================================
 
 We will now see that $\varphi$ induces a $\Cu$-morphism, and that such morphism is the limit of our  sequence. That is to say, we will see that the Cauchy sequences with summable distances have a limit.
 
 \begin{prp}\label{prp_limit_Cauchy}
 Retain the above assumptions. Then:
 \begin{enumerate}[(i)]
  \item The sequence $(\varphi_{i})_{i}$ induces a $\Cu$-morphism $\varphi\colon\LscI\to S$.
  \item $d(\varphi , \varphi_{i})\to 0$ as $i$ tends to infinity.
 \end{enumerate}
\end{prp}
\begin{proof} We prove each claim separately:

 \noindent (i) First we will see that the map $\varphi\colon\{\chi_{(t,1]}\}_{t}\to S$ preserves order, suprema and the way-below relation. Thus, let $\chi_{(s,1]}\leq \chi_{(t,1]}$, which happens if and only if $s\geq t$. Then, since $s+R-R_{i}\geq t+R-R_{i}$ for each $i$, we have
 \[
  \varphi_{i+1} (\chi_{ (s+R-R_{i},1] })\leq \varphi_{i+1} (\chi_{ (t+R-R_{i},1] })
 \]
 for every $i$, and hence $\varphi(\chi_{(s,1]})\leq \varphi (\chi_{(t,1]})$.
 \vspace{0.1cm}
 
 If $\chi_{(s,1]}\ll \chi_{(t,1]}$, we know that $s-t=d$ for some $d>0$.
 
 Thus, since $R-R_{i}\to 0$, there exists some $k\in\mathbb{N}$ such that $d> 2(R-R_{k-1})$. Also, note that for every $i> k$, one has
 \[
 \begin{split}
  d (\varphi_{i+1} , \varphi_{k})&\leq d(\varphi_{i+1},\varphi_{i})+\cdots + d(\varphi_{k+1},\varphi_{k})\\  
  &<\epsilon_{i}+\cdots + \epsilon_{k}=R_i -R_{k-1}\leq R-R_{k-1}.
  \end{split}
 \]

 Therefore, one has
 \[
  \varphi_{i+1}(\chi_{(s+R-R_{i},1]})=
  \varphi_{i+1}(\chi_{(t+d+R-R_{i},1]})
  \ll
  \varphi_{k}(\chi_{( t+d+R_{k-1}-R_{i} ,1]})\leq
  \varphi_{k}(\chi_{( t+R-R_{k-1},1]}),
 \]
where in the second step we have used $(t+d+R-R_{i})-(t+d+R_{k-1}-R_{i})=R-R_{k-1}>d(\varphi_{i},\varphi_{k})$, and in the third step we have used $d\geq 2(R-R_{k-1})\geq (R-R_{k-1}) + (R_{i}-R_{k-1})$.

This shows that $\varphi( \chi_{(s,1]})\leq \varphi_{k}(\chi_{( t+R-R_{k-1},1]})\ll \varphi( \chi_{(t,1]})$.\vspace{0.1cm}
 
 Now let $(t_{n})_{n}$ be a decreasing sequence converging to $t$. Since $t\leq t_{n}$ for each $n$, it follows that
 \[
  \sup_{n}\varphi (\chi_{(t_{n},1]})\leq
  \varphi (\chi_{(t,1]}).
 \]

 Conversely, let $k\in\mathbb{N}$ be fixed and take $i>k$. Recall from the previous argument that we have $d(\varphi_{i+1},\varphi_{k+1})< R_{i}-R_{k}$. Set $\epsilon= R_{i}-R_{k} - d(\varphi_{i+1},\varphi_{k+1})>0$.
 
 Since $t_{n}-t$ is positive and tends to zero, there exists $n$ such that $\epsilon > t_{n}-t$. In particular, one gets
 \[
  (t+R-R_{k}) - (t_{n}+R-R_{i})= (t-t_{n}) + R_{i}-R_{k}> -\epsilon +R_{i}-R_{k}= d(\varphi_{i+1},\varphi_{k+1}).
 \]

 Using \autoref{rmk_dist_first} we get
 \[
  \varphi_{k+1}(\chi_{( t+R-R_{k} ,1]})\ll
  \varphi_{i+1}(\chi_{( t_{n}+R-R_{i} ,1]})\leq \varphi (\chi_{(t_n ,1]})\leq \sup_n \varphi (\chi_{(t_n,1]}).
 \]

 This shows that $\varphi (\chi_{(t,1]})\leq \sup_{n} \varphi ( \chi_{(t_{n},1]} )$ and, consequently, 
 $\sup_{n} \varphi ( \chi_{(t_{n},1]} )=\varphi (\chi_{(t,1]})$.
 \vspace{0.1cm}
 
Now recall that $\mathcal{G}$ is the sub-$\Cu$-semigroup of increasing lower-semicontinuous functions in $(0,1]$ (see \cite[Section 5.2]{Scho18}). Using a similar argument to that of \autoref{Lemma_Lifting_Tecnic}, one can see that any order, suprema and $\ll$-preserving map $\{\chi_{(t,1]}\}_t\to S$ can be lifted uniquely to a $\Cu$-morphism $\mathcal{G}\to S$. In fact, note that in this case the argument is simpler, since $\mathcal{G}$ can be seen to be the sup-completion of the free abelian semigroup generated by 
$\{\chi_{(t,1]}\}_t$.

Thus, let $\varphi\colon\mathcal{G}\to S$ be the unique $\Cu$-morphism lifting our map $\varphi\colon \{\chi_{(t,1]}\}_t\to S$.

Finally, we know by \autoref{Lifting_Morphisms} that $\varphi$ has a unique lift $\varphi\colon \LscI\to S$ such that $\varphi (1)=\varphi_{i}(1)$ for all $i$.
 \vspace{0.3cm}
 
 (ii) Fix $\epsilon\in (0,1]$. We will find $m\in\NN$ such that $d(\varphi ,\varphi_m)\leq \epsilon $.
 
 Since $\lim R_j = R$, there exists $j_0\in \NN$ such that for every $j\geq j_0$ we have $0\leq R-R_j\leq \epsilon /2$. Thus, one gets
 \[
  \varphi_j(\chi_{(t+\epsilon /2,1]})\leq 
  \varphi_j(\chi_{(t+R-R_j,1]})
  \leq
  \varphi (\chi_{(t,1]})
 \]
for every $j\geq j_0$ and $t\in [0,1]$.

Now let $n\in\NN$ be such that $1/n\leq \epsilon /2$ and consider the elements $t_i=i/n$ of $[0,1]$. For every $i$ we have $\varphi (\chi_{(t_i +\epsilon/2 ,1]})\ll \varphi (\chi_{(t_i ,1]})$ and thus, by the definition of $\varphi$, there exists $j_i$ such that
\[
 \varphi (\chi_{(t_i +\epsilon/2 ,1]})\leq \varphi_{j_i}(\chi_{(t_i+R-R_{j_i -1} ,1]}).
\]

Let $m=\max ( j_0,j_1,\cdots ,j_n )$. Then, 
\[
 \begin{split}
  \varphi_{m}(\chi_{(t+\epsilon /2,1]}) &\leq \varphi (\chi_{(t,1]}),\text{ since }m\geq j_0,\andSep\\
  \varphi(\chi_{(t_{i}+\epsilon /2,1]}) &\leq 
  \varphi_{j_{i}} (\chi_{(t_{i}+R-R_{j_{i}-1},1]})\leq \varphi_{m} (\chi_{(t_{i}+R-R_{m-1},1]})\leq 
  \varphi_{m} (\chi_{(t_{i},1]}) .
 \end{split}
\]

Thus, we have by \autoref{Partition} that $d(\varphi ,\varphi_m)\leq \epsilon/2 +1/n\leq \epsilon $.
\end{proof}
%===========================================

Using \autoref{prp_limit_Cauchy}, we can prove the following result.

\begin{thm}\label{Gen_Cauchy}
 Let $L=\oplus_{j=1}^{r} L_{j}$ with $L_{j}=\LscI$ for each $j$, and let $\varphi_{i}\colon L\to S$ be a sequence of \CuMor{s} such that $d(\varphi_{i},\varphi_{i+1})<\epsilon_{i}$ with $(\epsilon_{i})$ strictly decreasing and $\sum_{i=1}^{\infty} \epsilon_{i}<\infty$. Also, assume that $\varphi_{i}(1_{j})=\varphi_{i+1}(1_{j})$ for each $i,j$.
 
 Then, there exists a unique $\Cu$-morphism $\varphi\colon L\to S$ satisfying $d(\varphi , \varphi_{i})\to 0$ and $\varphi (1_{j})=\varphi_{i}(1_{j})$ for every $j\leq r$. 
\end{thm}
\begin{proof}
 For each fixed $j\leq r$, apply the previous proposition to the sequence $(\varphi_{i}\tau_{j})_{i}$, where $\tau_{j}\colon L_{j}\to L$ is the canonical inclusion. This produces a $\Cu$-morphism $\varphi^{(j)}$ such that $d(\varphi_{i}\tau_{j},\varphi^{(j)})\to 0$ for each $j$.

 The $\Cu$-morphism $\varphi:=\varphi^{(1)}\oplus\cdots\oplus \varphi^{(r)}\colon L\to S$ satisfies $d(\varphi , \varphi_{i})\to 0$ as required.

 To prove uniqueness, let $\phi\colon L\to S$ be a $\Cu$-morphism with $d(\varphi_{i},\phi)\to 0$ and $\phi (1_{j})=\varphi_{i}(1_{j})$ for every $j\leq r$. Using the triangle inequality, we obtain
 \[
  d(\varphi,\phi)\leq d(\varphi,\varphi_{i}) + d(\varphi_{i},\phi)\to 0.
 \]
This shows that $d(\varphi,\phi)=0$ and, consequently, $\varphi=\phi$.
\end{proof}
%===========================================

\subsection{A local characterization for AI-algebras}

In order to ease the notation, in this subsection we will denote the Cuntz semigroup $\LscI$ by $L$.

\begin{lma}\label{Limit_Morph}
 Let $S$ be a \CuSgp{} satisfying (O5) and weak cancellation. Let $(n_i)_i$ be a sequence in $\NN$ and consider a pair of sequences $\varphi_{i}\colon L^{n_{i}}\to S$ and $\sigma_{i+1,i}\colon L^{n_{i}}\to L^{n_{i+1}}$ of $\Cu$-morphisms. Let $\sigma_{i,j}$ denote the composition $\sigma_{j,j-1}\circ\cdots\circ \sigma_{i+1,i}$.
 
 If there exists a strictly decreasing sequence $(\epsilon_{i})_{i}$ with
 \[
  d(\varphi_{j}\sigma_{j,i}, \varphi_{j+1}\sigma_{j+1,i})< \epsilon_{i}/2^{j}, \andSep \varphi_{j+1}\sigma_{j+1,i}(1_{k})=\varphi_{j}\sigma_{j,i}(1_{k}) \text{ for each $k\leq n_{i}$},
 \]
 then we can find a $\Cu$-morphism $\phi\colon\lim (L^{n_{i}},\sigma_{i+1,i})\to S$ such that its canonical morphisms $\phi_{i}\colon L^{n_{i}}\to S$ are the limits of the sequences $(\varphi_{j}\sigma_{j,i})_{j}$.
\end{lma}
\begin{proof}
 By \autoref{Gen_Cauchy}, each sequence $(\varphi_{j}\sigma_{j,i})_{j}$ has a limit, which we denote by $\phi_{i}\colon L^{n_{i}}\to S$.
 
 We will now see that $\phi_{i+1}\sigma_{i+1,i}=\phi_{i}$ for each $i$. Thus, we will obtain a $\Cu$-morphism $\phi\colon\lim L^{n_{i}}\to S$, as required.
 
 Fix $i\in\mathbb{N}$ and take any $\epsilon >0$. Let $\delta>0$ be the distance given by \autoref{AI_Rmk} for the morphism $\sigma_{i+1,i}$ and distance $\epsilon$. Since $\phi_{i}$ is the limit of the  sequence $(\varphi_{j}\sigma_{j,i})_{j}$, we can take $j$ such that
\[
 d(\phi_{i}, \varphi_{j}\sigma_{j,i})<\epsilon
 ,\andSep
 d(\phi_{i+1}, \varphi_{j}\sigma_{j,i+1})<\delta.
\]

We have
\[
\begin{split}
 d (\phi_{i},\phi_{i+1}\sigma_{i+1,i})  & \leq 
 d (\phi_{i}, \varphi_{j}\sigma_{j,i})
 +
 d(\varphi_{j}\sigma_{j,i} ,\phi_{i+1}\sigma_{i+1,i})\\
 &=
 d (\phi_{i}, \varphi_{j}\sigma_{j,i})
 +
 d(\varphi_{j}\sigma_{j,i+1}\sigma_{i+1,i} ,\phi_{i+1}\sigma_{i+1,i})< 2\epsilon .
\end{split}
\]
Consequently, $\phi_{i}=\phi_{i+1}\sigma_{i+1,i}$ for every $i$ and this induces a morphism $\phi$ from the limit $\lim (L^{n_{i}},\sigma_{i+1,i})$ to $S$.
\end{proof}
%===========================================

We are now ready to prove the main result of the paper, which gives a characterization of the Cuntz semigroup of AI-algebras in terms of
 certain decompositions of suitable $\Cu$-morphisms.

%As mentioned in the beginning of the section, this can be seen as the analogue of \cite[Theorem 3.1]{Shen79}.

\begin{dfn}\label{dfn_comp_bound}
 We will say that a \CuSgp{} $S$ is \emph{compactly bounded} if every compactly contained element in $S$ is bounded by a compact. That is, for every element $a$ such that $a\ll b$ for some $b\in S$, there exists a compact element $p\geq a$.
\end{dfn}

The following proof combines some ideas from \cite[Theorem 3.1]{Shen79} and \cite[Chapter 12, Section 3]{Nadl92}.

\begin{thm}\label{Main_EHS}
 Let $S$ be a countably based and compactly bounded  $\Cu$-semigroup satisfying weak cancellation and (O5). Then, $S$ is $\Cu$-isomorphic to the Cuntz semigroup of an AI-algebra if and only if for every $\Cu$-morphism $\varphi\colon \LscI^{r}\to S$, for every finite subset $F\subset \LscI^{r}$ and every $\epsilon >0$, there exist $s\geq 1$ and $\Cu$-morphisms $\theta\colon L^r\to L^s$, $\phi\colon L^s\to S$, such that the diagram
 \begin{equation*}
  \xymatrix{
  \LscI^{r} \ar[r]^-{\varphi} \ar[d]_-{\theta} & S\\
  \LscI^{s} \ar[ru]_-{\phi} & 
  }
 \end{equation*}
satisfies:
\begin{enumerate}[(i)]
 \item $d(\phi\theta ,\varphi )< \epsilon$.
 \item For every $x,x',y\in F$, we have $\theta (x)\ll\theta (y)$ whenever $x\ll x'$ and $\varphi (x')\ll \varphi (y)$.
 \item $\varphi (1_{j})=\phi\theta (1_{j})$ for every $1\leq j\leq r$.
\end{enumerate}
\end{thm}
\begin{proof}
Let $S$ be isomorphic the Cuntz semigroup of an AI-algebra $A$, and let $\varphi$, $F$ and $\epsilon$ be as in the statement of the theorem. Since $S\cong \Cu (A)$, we know by \cite[Theorem 12.1]{CiupElli08} that this map lifts to a *-homomorphism $g\colon B\to A$, where $B$ is a direct sum of interval algebras and $A$ is the limit of an inductive system $(B_{i},f_{i+1,i})$ with $B_{i}$ a direct sum of interval algebras for each $i$.

By \autoref{WB_Maps}, for every triple of elements $x,x',y\in F$ with $x\ll x'$ and $\varphi (x')\ll \varphi (y)$, there exists $\epsilon( x,x',y )$ such that whenever $d(\varphi , \psi)<\epsilon( x,x',y )$ and $\varphi (1_{j})=\psi (1_{j})$ for every $j$, we have $\psi (x)\ll \psi (y)$.

Since the cardinality of $F$ is finite, so is the number of way-below relations between its elements. This means that the number $\epsilon_F=\min (\epsilon ,1 , \epsilon( x,x',y ))$ is strictly positive.

 Since $B_{i}$ is projective (see, for example, \cite[Section 3]{EffrKami86}), there is a *-homomorphism $h\colon B\to B_{i}$ such that $\Vert g(x)-f_{i}h(x)\Vert< \epsilon_F$ for every $x\in B$. Here, $f_{i}\colon B_{i}\to A$ is the canonical map. In particular, since $\epsilon_F <1$, it follows that $g(p)$ is Murray-Von Neumann equivalent to $f_{i}h(p)$ for every projection $p\in B$.

Applying the functor $\Cu$ we obtain the following diagram
\begin{equation*}
  \xymatrix{
  \LscI^{r} \ar[r]^-{\varphi} \ar[d]_-{\Cu (h)} & S\\
  \Cu (B_{i}) \ar[ru]_-{\Cu (f_{i})} & 
  }
 \end{equation*}
 
 Since the norm $\Vert g-f_{i}h\Vert$ is an upper bound for $d(\varphi , \Cu (h)\Cu (f_{i}))$ (see, for example, \cite[Lemma 1]{RobSan10}), we get $d(\varphi , \Cu (f_{i})\Cu (h))< \epsilon_F<\epsilon$. Moreover, since  being Murray-Von Neumann equivalent implies being Cuntz equivalent, one also gets $\varphi (1_k)=\Cu (f_{i})\Cu (h) (1_k)$ for every $k\leq r$.

 By the choice of $\epsilon_F$, we also have that $\Cu (f_{i})\Cu (h)(x)\ll \Cu (f_{i})\Cu (h)(y)$ for every triple $x,x',y\in F$ with $x\ll x'$ and $\varphi (x')\ll \varphi (y)$.
 
 Finally, note that $\Cu (f_{i})\colon \Cu (B_{i})\to S$ is the canonical morphism from $\Cu (B_{i})$ to the limit $\lim \Cu (B_{i})\cong \Cu (A)\cong S$. Thus, we know that $\Cu (f_{i})\Cu (h)(x)\ll \Cu (f_{i})\Cu (h)(y)$ if and only if there exists $i(x,y)\geq i$ with $\Cu (f_{i(x,y),i})\Cu (h)(x)\ll 
 \Cu (f_{i(x,y),i})\Cu (h)(y)$.
 
 Since $F$ is finite, so is the supremum $j$ of all the $i(x,y)$'s with $x,y\in F$. Setting $\theta:=\Cu (f_{j,i})\Cu (h)$ and $\phi:= \Cu (f_{j})$, we have
 \begin{enumerate}[(i)]
  \item $d(\phi\theta,\varphi)=
   d(\Cu (f_{j})\Cu (f_{j,i})\Cu (h),\varphi )= d(\Cu (f_i)\Cu (h),\varphi )<\epsilon$.
   \item For every triple $x,x',y\in F$ such that $x\ll x'$ and $\varphi (x')\ll \varphi (y)$, we get 
   \[
    \Cu (f_{i})\Cu (h)(x)\ll \Cu (f_{i})\Cu (h)(y).
   \]
   
   By our choice of $j$, it follows that 
   \[
   \theta (x)=\Cu (f_{j,i})\Cu (h)(x)\ll \Cu (f_{j,i})\Cu (h)(y)=\theta (y).
   \]
   \item $\phi\theta (1_k)=\Cu (f_i)\Cu (h) (1_k)=\varphi (1_k)$ for every $k\leq r$.
 \end{enumerate}
 as required.\vspace{0.3cm}
 
 We are now left to prove the other implication.\vspace{0.3cm}

 Let $s_{1},s_{2},\cdots$ be a countable basis for $S$, where we may assume $s_{i}\in S_{\ll}$ for each $i$, and consider a \CuMor{} $\psi_{i}\colon \LscI\to S$ such that $\psi_{i}(\chi_{(0,1]})=s_{i}$. Such a morphism can always be found by \autoref{lifting_th} and the fact that all compactly contained elements in $S$ are bounded by a compact. Also, denote by $\rho_{i}\colon\LscI^{i}\to S$ the direct sum $\rho_{i}=\psi_{1}\oplus\cdots\oplus \psi_{i}$.

 For every $j\in\mathbb{N}$, fix a countable and ordered basis for $L^{j}$. By "the first $i$ basic elements in $L^{j}$" we will mean the first $i$ elements of the fixed ordered basis of $L^{j}$.\vspace{0.1cm}
 
 The idea of the proof is as follows:
 \vspace{0.1cm}
 
 We will first define inductively $\Cu$-morphisms $\sigma_{i+1,i}\colon L^{n_{i}}\to L^{n_{i+1}}$ and $\varphi_{i}\colon L^{n_{i}}\to S$ such that:
 \begin{enumerate}[(i)']
  \item There exists a decreasing sequence of positive elements $(\epsilon_{i})$ tending to $0$ such that, for every $i$, there exists a $\Cu$-morphism $\theta_{i}\colon L^{n_{i-1}}\oplus L^{i}\to L^{n_{i}}$ with $d(\varphi_{i-1}\oplus\rho_{i},\varphi_{i}\theta_{i})<\epsilon_{i}$.
  \item For every fixed $i$, we have 
  \[
  d(\varphi_{j}\sigma_{j,i}, \varphi_{j+1}\sigma_{j+1,i})< \epsilon_{i}/2^{j}, \andSep \varphi_{j+1}\sigma_{j+1,i}(1_{k})=\varphi_{j}\sigma_{j,i}(1_{k})
  \]
 for each $k\leq n_{i}$. Here, $\sigma_{i,j}$ denotes the composition $\sigma_{j,j-1}\circ\cdots\circ \sigma_{i+1,i}$.
  \item For every fixed $k$, we also have $d(\varphi_{j}\sigma_{j,k}\theta_{k},\varphi_{j+1}\sigma_{j+1,k}\theta_{k})<\epsilon_{k}/2^{j}$.
  \item For each $i$, let $F_{i}$ be the finite set consisting of the images of the first $i$ basic elements of  $L^{n_{r}}$ through $\sigma_{i,r}$ for each $r\leq i$. Then, for every $x,x',y\in F_{i}$ satisfying $\varphi_{i}(x')\ll \varphi_{i}(y)$ with $x\ll x'$, we have $\sigma_{i+1,i}(x)\ll \sigma_{i+1,i}(y)$.
 \end{enumerate}

 Condition (ii)' and \autoref{Limit_Morph} will provide a limit morphism $\phi\colon\lim L^{n_{i}}\to S$ with the canonical morphisms $\phi_{i}\colon L^{n_{i}}\to S$ being the limits of the sequences $(\varphi_{j}\sigma_{j,i})_{j}$.
 
 Conditions (i)' and (iii)' will imply that $\phi$ is surjective. Condition (iv)' will be used to prove that $\phi$ is also an order embedding, thus showing the desired result.
 \vspace{0.3cm}
 
 Set $\varphi_{1}:=\rho_{1}$ and take some fixed $i\in\mathbb{N}$. Assume that for each $k\leq i-1$ the elements $\epsilon_{k}$, the morphisms $\sigma_{k,k-1}$, $\varphi_{k}$ and $\theta_{k}$ and the sets $F_{k-1}$ have been defined so that conditions (i)-(iv) above are satisfied.
 
 For every $k\leq i-1$, let $\delta_{k}$ be the distance given in \autoref{AI_Rmk} such that for any pair of morphisms $\zeta_{1},\zeta_{2}\colon L^{n_{i-1}}\to S$ at distance less than $\delta_{k}$, we have 
 \begin{equation}\label{dk}
  d( \zeta_{1}\sigma_{i-1,k}, \zeta_{2}\sigma_{i-1,k})< \epsilon_{k}/2^{i}, \andSep d(\zeta_{1}\sigma_{i-1,k}\theta_{k}, \zeta_{2}\sigma_{i-1,k}\theta_{k})< \epsilon_{k}/2^{i} .
 \end{equation}
 
 Set $\epsilon_{i}:=\min_{1\leq k\leq i-1}\{\delta_{k} ,\epsilon_{k}\}>0$. As defined above, let $F_{i-1}$ be the set that contains, for each $r\leq i-1$, the image through $\sigma_{i-1,r}$ of $i-1$ distinct basic elements of $L^{n_{r}}$.
 
 Let $\tau_{i-1}\colon L^{n_{i-1}}\to L^{n_{i-1}}\oplus L^{i}$ be the canonical inclusion, and let $F=\tau_{i-1}(F_{i-1})$.
 
 By our assumptions, we can find morphisms $\varphi_{i},\theta_{i}$ such that the diagram
 \begin{equation*}
 \xymatrixcolsep{4pc}
  \xymatrix{
  L^{n_{i-1}}\oplus L^{i} \ar[r]^-{\varphi_{i-1}\oplus \rho_{i}} \ar[d]_-{\theta_{i}} & S\\
  L^{n_{i}} \ar[ru]_-{\varphi_{i}} & 
  }
 \end{equation*}
satisfies conditions (i)-(iii) in the statement of the theorem with distance $\epsilon_{i}$ and finite set $F$.

Define $\sigma_{i,i-1}:=\theta_{i}\circ \tau_{i-1}$, and note that condition (i)' is immediately satisfied. Also, condition (iv)' is satisfied by construction.
\begin{equation*}
 \xymatrixcolsep{4pc}
  \xymatrix{
  L^{n_{i-1}} \ar[r]^-{\tau_{i-1}} 
  \ar@{-->}[rd]_-{\sigma_{i,i-1}}
  \ar@{-->}@/^2.5pc/[rr]^-{\varphi_{i-1}}
  & L^{n_{i-1}}\oplus L^{i} \ar[r]^-{\varphi_{i-1}\oplus \rho_{i}} \ar[d]_-{\theta_{i}} & S\\
  & L^{n_{i}} \ar[ru]_-{\varphi_{i}} & 
  }
 \end{equation*}

Furthermore, since $\tau_{i-1}$ is an inclusion, we have by \autoref{General_Dist} and $\rho_i\tau_{i-1}=0$ that
\begin{equation*}
 d(\varphi_{i}\sigma_{i,i-1},\varphi_{i-1})=
 d(\varphi_{i}\theta_{i}\tau_{i-1} , (\varphi_{i-1}\oplus\rho_{i})\tau_{i-1})
 \leq 
 d(\varphi_{i}\theta_{i}, \varphi_{i-1}\oplus\rho_{i})
 <\epsilon_{i}\leq \delta_{k}
\end{equation*}
for each $k$.

By the choice of $\delta_{k}$ made in (\ref{dk}), one gets
\[
 d(\varphi_{i}\sigma_{i,k},\varphi_{i-1}\sigma_{i-1,k})
 =d(\varphi_{i}\sigma_{i,i-1}\sigma_{i-1,k},\varphi_{i-1}\sigma_{i-1,k})<\epsilon_{k}/2^{i}.
\]

Moreover, as $\varphi_{i}\sigma_{i,i-1}(1_{j})=\varphi_{i-1}(1_{j})$ for every $j\leq n_{i-1}$, condition (ii)' also holds. An analogous argument shows that (iii)' holds.

This finishes the inductive argument.
\vspace{0.3cm}

By \autoref{Limit_Morph}, condition (ii)' induces a $\Cu$-morphism $\phi\colon \lim L^{n_{i}}\to S$ with the canonical morphisms $\phi_{i}\colon L^{n_{i}}\to S$ being the limits of the sequences $(\varphi_{j}\sigma_{j,i})_{j}$.
\vspace{0.3cm}

To see that $\phi$ is surjective, first note that, for every $i\in\mathbb{N}$ and for every $\epsilon>0$, there exists $j\in\mathbb{N}$ such that $d(\phi_{i}\theta_{i},\varphi_{j}\sigma_{j,i}\theta_{i})<\epsilon$. This is due to \autoref{AI_Rmk} and because $d(\phi_{i} ,\varphi_{j}\sigma_{j,i})$ tends to $0$ as $j$ tends to infinity.

Thus, we get
\[
\begin{split}
 d(\varphi_{i}\theta_{i},\phi_{i}\theta_{i}) &\leq 
 d(\varphi_{i}\theta_{i},\varphi_{i+1}\sigma_{i+1,i}\theta_{i}) + \cdots +
 d(\varphi_{j-1}\sigma_{j-1,i}\theta_{i},\varphi_{j}\sigma_{j,i}\theta_{i})+
 d(\varphi_{j}\sigma_{j,i}\theta_{i},\phi_{i}\theta_{i})
 \\
 &\leq
 \frac{\epsilon_{i}}{2^{i}}+\cdots +
 \frac{\epsilon_{i}}{2^{j-1}}+\epsilon
 \leq 
 2\epsilon_{i} +\epsilon ,
\end{split}
\]
where we have used property (iii)' to bound all but the last element.

Since this holds for every $\epsilon$, one gets $d(\varphi_{i}\theta_{i},\phi_{i}\theta_{i})\leq 2\epsilon_{i}$. In particular, the distance is strictly decreasing on $i$.

Now let $s_{i}$ be a basic element of $S$ and let $x\in S$ be such that $x\ll s_{i}$. We have $x \ll \psi_{i}(\chi_{(0,1]})$ and, consequently, there exists $s,t\in (0,1]$ such that
\[
 x\ll \psi_{i}(\chi_{(t+2s,1]})\ll
 \psi_{i}(\chi_{(t,1]})\ll 
 \psi_{i}(\chi_{(0,1]}).
\]

Note that, for every $k>i$, we have
\[
 d(\varphi_{k-1}\oplus \rho_{k},\phi_{k}\theta_{k})\leq
 d(\varphi_{k-1}\oplus \rho_{k},\varphi_{k}\theta_{k})
 +
 d(\varphi_{k}\theta_{k},\phi_{k}\theta_{k})<
 \epsilon_{k} + 2\epsilon_{k}=3\epsilon_{k},
\]
where in the previous bound we have used condition (i)' and the inequality obtained above.

Thus, there exists a large enough $k$ so that $d(\varphi_{k-1}\oplus \rho_{k},\phi_{k}\theta_{k})<s$. 

Since $k>i$, we know that $(\varphi_{k-1}\oplus \rho_{k})\eta_{i}=\psi_{i}$, where $\eta_{i}\colon L\to L^{n_{k-1}}\oplus L^{k}$ is the canonical inclusion in the $(n_{k-1}+i)$-th $L$-summand.

Thus, we have
\[
\begin{split}
 x &\ll \psi_{i}(\chi_{(t+2s,1]})=
 (\varphi_{k-1}\oplus \rho_{k})\eta_{i}(\chi_{(t+2s,1]})\ll
 \phi_{k}\theta_{k}\eta_{i}(\chi_{(t+s,1]})
 \\
 &\ll
 (\varphi_{k-1}\oplus \rho_{k})\eta_{i}(\chi_{(t,1]})= 
 \psi_{i}(\chi_{(t,1]}).
 \end{split}
\]

This shows that for every $i$ and every $x\ll s_{i}$, there exist some $k$ with $x\ll \phi_{k}(l)\ll s_{i}$ with $l\in L^{n_{k}}$. Consequently, $\phi$ is surjective.
\vspace{0.3cm}

We will now prove that $\phi$ is an order-embedding. To do this, we will denote by $[a]$ the elements in $\lim L^{n_{i}}$ coming from some block $L^{n_{i}}$ of the direct limit.

Take $x,y\in \lim L^{n_{i}}$ such that $\phi (x)\leq \phi (y)$. Let $a\in L^{n_{s}}$ be a basic element with $[a]\ll x$. Also, take $[z],[z']$ such that $[a]\ll [z']\ll [z]\ll x$ with $z,z'\in L^{n_{s'}}$ basic elements as well. Finally, take $b\in L^{n_{k}}$ a basic element such that $[b]\ll y$ and $\phi ([z])\ll \phi ([b])\ll\phi (y)$.

We can assume that, for a large enough $i$, we have $\sigma_{i,s}(a),\sigma_{i,k}(b),\sigma_{i,s'}(z),\sigma_{i,s'}(z') \in L^{n_{i}}$ with 
\[
 \sigma_{i,s}(a) \ll \sigma_{i,s'}(z')\ll \sigma_{i,s'}(z), \andSep
  \phi_{i}(\sigma_{i,s'}(z)) \ll \phi_{i} (\sigma_{i,k}(b)).
 \]

Thus, since we have $d(\phi_{i}, \varphi_{j}\sigma_{j,i})\to 0$, it follows from \autoref{WB_Maps} that 
\[
\varphi_{j}\sigma_{j,i}(\sigma_{i,s'}(z'))\ll \varphi_{j}\sigma_{j,i}(\sigma_{i,k}(b))
\]
for every sufficiently large $j$. 

Also, since $z,z',a,b$ are basic elements in their respective blocks, we can take $j$ large enough so that we also have $\sigma_{j,s'}(z),\sigma_{j,s'}(z'),\sigma_{j,s}(a),\sigma_{j,k}(b)\in F_j$.

Therefore, since $\sigma_{j,s}(a)\ll \sigma_{j,s'}(z')$ and $\varphi_{j} (\sigma_{j,s'}(z'))\ll \varphi_{j} (\sigma_{j,k}(b))$, it follows from condition (iv) that $\sigma_{j+1,j}(\sigma_{j,s}(a))\ll \sigma_{j+1,j}(\sigma_{j,k}(b))$. That is to say, we have
\[
 \sigma_{j+1,s}(a)\ll \sigma_{j+1,k}(b)
\]
and thus $[a]\ll [b]\ll y$.

Since $x$ can be written as the supremum of an $\ll$-increasing sequence $([x_n])$ with $x_n$ basic elements, it follows from the previous argument that $[x_n]\ll y$ for every $n$. Taking the supremum, one gets $x\leq y$ as required.

We now have $S\cong \lim_{i} L^{n_{i}}$, and the desired result follows from \autoref{AI_Lsc}.
\end{proof}
%===========================================

The following result shows that we only need to focus on one $\ll$-relation instead of working with a finite subset $F$ and all of its $\ll$-relations.

\begin{prp}\label{Main_OneRel}
 Let $S$ be a countably based and compactly bounded  $\Cu$-semigroup satisfying weak cancellation and (O5). Then, $S$ is $\Cu$-isomorphic to the Cuntz semigroup of an AI-algebra if and only if for every \CuMor{} $\varphi\colon \LscI^{r}\to S$, every $\epsilon >0$ and every triple $x,x', y$ in $\LscI^r$ such that $x\ll x'$ with $\varphi (x')\ll \varphi (y)$, there exist $\Cu$-morphisms $\theta ,\phi$ such that the diagram
 \[
  \xymatrix{
  \LscI^{r} \ar[r]^-{\varphi} \ar[d]_-{\theta} & S\\
  \LscI^{s} \ar[ru]_-{\phi} & 
  }
 \]
satisfies:
\begin{enumerate}[(i)]
 \item $d(\phi\theta ,\varphi )< \epsilon$.
 \item $\theta (x)\ll\theta (y)$.
 \item $\varphi (1_{j})=\phi\theta (1_{j})$ for every $1\leq j\leq r$.
\end{enumerate}
\end{prp}
\begin{proof}
 One direction follows trivially from \autoref{Main_EHS}, so we only need to prove the other. To do this, we will see that the conditions in \autoref{Main_EHS} are satisfied.
 
 Let $\varphi\colon \LscI^{r}\to S$ and define the following set
 \[
  \mathcal{R}\varphi =\{ (x,x',y)\in (\LscI^{r})^3 \mid x\ll x'\text{ with } \varphi (x')\ll\varphi (y)\}.
 \]

Let $R=\{ (x_{i},x_{i}',y_{i}) \}$ be a finite subset of $\mathcal{R}\varphi$ (note that some $x_{i}$'s, $x_i'$'s and $y_{j}$'s might coincide). We will prove by induction on the cardinality of $R$ that given any $\varphi\colon \LscI^{r}\to S$, any $\epsilon >0$ and any finite subset $R\subset \mathcal{R}\varphi$ there exist morphisms $\theta ,\phi$ such that 
\begin{enumerate}[(i)]
 \item $d(\phi\theta ,\varphi )< \epsilon$.
 \item $\theta (x)\ll\theta (y)$ for every $(x,x',y)\in R$.
 \item $\varphi (1_{j})=\phi\theta (1_{j})$ for every $1\leq j\leq r$.
\end{enumerate}
 
Note that, given any finite subset $F\subset \LscI^{r}$, the set $F^3\cap \mathcal{R}\varphi$ is finite. Thus, our result will follow from this fact.

For $n=1$, the results holds by assumption.

Now assume that for every $k\leq n-1$, the desired result has been proven for every $\Cu$-morphism $\varphi\colon \LscI^{r}\to S$, for every $\epsilon >0$ and every finite subset $R\subset \mathcal{R}\varphi$ such that $\vert R\vert =k$. Fix any $\Cu$-morphism $\varphi\colon \LscI^{r}\to S$, any $\epsilon >0$ and any finite subset $R\subset \mathcal{R}\varphi$ with $\vert R\vert =n$.

As before, write $R=\{ (x_{1},x_1',y_{1}),\cdots ,(x_{n},x_n',y_{n}) \}$. By \autoref{WB_Maps}, there exist elements $x_i''$ with $x_i\ll x_i''\ll x_i$ and a bound $\delta >0$ such that, for any $\Cu$-morphism $\psi\colon \LscI^{r}\to S$ with $d(\psi,\varphi)<\delta$ and $\psi (1_{j})=\phi (1_{j})$ for each $j$, we have $\psi (x''_{i})\ll \psi (y_{i})$.

Set $\epsilon ':=\min \{\epsilon /2,\delta\}>0$, and apply the induction hypothesis to $\varphi$, $\epsilon '$ and the set $\{(x_{1},x_{1}'',y_{1})\}$. Thus, we get $\Cu$-morphisms 
\[
 \theta_{1}\colon \LscI^{r}\to \LscI^{t},\andSep 
 \phi_{1}\colon \LscI^{t}\to S
\]
such that 
\begin{enumerate}[(i)]
 \item $d(\phi_{1}\theta_{1} ,\varphi )< \epsilon '$.
 \item $\theta_{1} (x_{1})\ll\theta_{1} (y_{1})$.
 \item $\varphi (1_{j})=\phi_{1}\theta_{1} (1_{j})$ for every $1\leq j\leq r$.
\end{enumerate}

In particular, note that (i) implies that $(\theta_{1}(x_{i}),\theta_{1}(x_{i}''),\theta_{1}(y_{i}))\in \mathcal{R}\phi_{1}$ for every $i\leq n$. Indeed, we have $\phi_1\theta_1 (x_i)\ll \phi_1\theta_1 (x_i'')$ and $\varphi (x_i'')\ll \varphi (y_i)$. Since $d(\phi_{1}\theta_{1} ,\varphi )< \epsilon '<\delta$, we get $\phi_1\theta_1 (x_i'')\ll \phi_1\theta (y_i)$. Therefore, we have
\[
 \theta_1 (x_i)\ll \theta_1 (x_i''), \andSep 
 \phi_1\theta_1 (x_i'')\ll \phi_1\theta (y_i),
\]
which shows that $(\theta_{1}(x_{i}),\theta_{1}(x_{i}''),\theta_{1}(y_{i}))\in \mathcal{R}\phi_{1}$ as desired.

Now take the finite subset $R_{1}=\{ (\theta_{1}(x_{i}),\theta_{1}(x_{i}''),\theta_{1}(y_{i})) \}_{2\leq i\leq n}$, which has cardinality $n-1$. Let $\nu$ be the bound given in \autoref{AI_Rmk} for the morphism $\theta_{1}$ and the constant $\epsilon /2$.

Then, applying the induction hypothesis to $\phi_{1}$, $\nu >0$ and $R_{1}$ we get  $\Cu$-morphisms $\theta_{2},\phi_{2}$ with 
\begin{enumerate}[(i)]
 \item $d(\phi_{2}\theta_{2} ,\phi_{1} )<\nu$.
 \item $\theta_{2}\theta_{1} (x_{i})\ll\theta_{2}\theta_{1} (y_{i})$ for every $1\leq i\leq n$. For $i=1$, this is because $\theta_1 (x_{1})\ll \theta_1 (y_{1})$.
 \item $\phi_{1} (1_{j})=\phi_{2}\theta_{2} (1_{j})$ for every $1\leq j\leq t$.
\end{enumerate}

Since $d(\phi_2\theta_2 ,\phi_1)<\nu$, we have by \autoref{AI_Rmk} that $d(\phi_{2}\theta_{2} \theta_{1},\phi_{1} \theta_{1})<\epsilon /2$. In this situation, we get 
\[
d(\varphi ,\phi_{2}\theta_{2} \theta_{1})\leq d(\varphi ,\phi_{1}\theta_{1}) + d(\phi_{2}\theta_{2} \theta_{1},\phi_{1} \theta_{1})<\epsilon.
\]

Finally, given any $1_{j}\in \LscI^{r}$, write $\theta_{1}(1_{j})=k_{1}1_{1}+\cdots + k_{t}1_{t}\in\LscI^{t}$. We have
\[
\begin{split}
 \varphi(1_{j})&=\phi_{1}\theta_{1}(1_{j})=
 \phi_{1}(k_{1}1_{1}+\cdots + k_{t}1_{t})=
 k_{1}\phi_{1}(1_{1})+\cdots +k_{t}\phi_{1}(1_{t})\\
 &= k_{1}\phi_{2}\theta_{2}(1_{1})+\cdots +k_{t}\phi_{2}\theta_{2}(1_{t})=\phi_{2}\theta_{2}\theta_{1}(1_{j}).
\end{split}
\]

This shows that $\phi:= \phi_{2}$ and $\theta:=\theta_{2}\theta_{1}$ satisfy the required properties.
\end{proof}
%===========================================

We will now reduce the hypothesis of the theorem further, by showing that we can discretize \autoref{Main_OneRel} and that one only needs to focus on a particular kind of elements $x,x',y$. More explicitly, we will later prove the following:
\vspace{0.2cm}

Given any $l\in\mathbb{N}$, let $C_{l}^{M}=\{ \chi^{s}_{(i/l,1]} \}_{i,s}\cup\{ 1_{s}\}\subset \LscI^{M}$, where $1_{s}$ and $\chi^{s}_{(i/l,1]}$ denote the unit and the element $\chi_{(i/l,1]}$ in the $s$-th summand respectively for every $s\leq M$. Also, let $B_{l}^{M}$ be the additive span of the elements in $C_{l}^{M}$.

\begin{thm}\label{Triangular_Discrete}\label{MainTh}
Let $S$ be a countably based and compactly bounded $\Cu$-semigroup satisfying weak cancellation and (O5). Then, $S$ is the Cuntz semigroup of an AI-algebra if and only if for every \CuMor{}
 $\varphi\colon \LscI^{M}\to S$, every $l\in\mathbb{N}$ and every triple of elements $x, x', y\in B_{l}^{M}$  such that $x\ll x'$ with $\varphi (x')\ll \varphi (y)$, there exist $\Cu$-morphisms $\theta ,\phi$ such that the diagram
 \begin{equation*}
  \xymatrix{
  \LscI^{M} \ar[r]^-{\varphi} \ar[d]_-{\theta} & S\\
  \LscI^{N} \ar[ru]_-{\phi} & 
  }
 \end{equation*}
satisfies:
\begin{enumerate}[(i)]
 \item $\varphi (\chi^{s}_{((i+1)/l,1]})\ll \phi\theta(\chi^{s}_{(i/l,1]})$ and $\phi\theta(\chi^{s}_{((i+1)/l,1]})\ll \varphi (\chi^{s}_{(i/l,1]})$ for every $\chi^{s}_{(i/l,1]}\in C_{l}^{M}$.
 \item $\theta (x)\ll\theta (y)$.
 \item $\varphi (1_{s})=\phi\theta (1_{s})$ for every $s\leq M$.
\end{enumerate}
\end{thm}
%===========================================

\begin{exa} 
 Let $\mathcal{Z}$ be the Jiang-Su algebra (as introduced in \cite{JiaSu99}) and let $Z$ be its Cuntz semigroup. It is known that $Z$ is $\Cu$-isomorphic to $\NN \sqcup (0,\infty ]$, where the elements in $\NN$ are compact and the ones in $(0,\infty ]$ are not. For more details, see \cite[Paragraph~7.3.2]{AntoPereThie18}.
 
 We will show that $Z$ does not satisfy the conditions in \autoref{Triangular_Discrete}. This is trivial since $\mathcal{Z}$ is not an AI-algebra, but we give here an explicit proof.
 
 Let $l\geq 3$ and let $\varphi\colon \LscI\to Z$ be any $\Cu$-morphism such that $\varphi (\chi_{(k/l,1]})=1'-k/l$ and $\varphi (1)=1$. At least one such morphism exists by \autoref{lifting_th}, since $(1'-k/l)_{k}\subset Z$ is a rapidly decreasing sequence bounded by the compact $1$.
 
 We have that $\varphi (1)=1\ll 3/2=3\varphi (\chi_{(1/2,1]})$. Then, if $Z$ were to satisfy the previous theorem, we would get morphisms $\theta\colon\LscI\to\LscI^{N}$ and $\phi\colon\LscI^{N}\to Z$ such that $d(\phi\theta ,\varphi)<1/2$, $\theta(1)\ll \theta (3\chi_{(1/2,1]} )$ and $\phi\theta (1)=\varphi (1)$.
 
 Note that one also has $\phi(\theta (1)\wedge 1)=\varphi (1)=1$, where the infimum is taken componentwise. Moreover,  since $\theta (1)\ll 3\theta (\chi_{(1/2,1]})$, we have $\theta (1)\wedge 1\ll 3\theta (\chi_{(1/2,1]})\wedge 1$ and, consequently, $\text{supp} (\theta (1)\wedge 1)\subset \text{supp} (\theta (\chi_{(1/2,1]})\wedge 1)$. This implies $\theta (\chi_{(1/2,1]})\wedge 1=\theta (1)\wedge 1$.
 
 But now we have
 \[
  \phi(\theta (\chi_{(1/2,1]})\wedge 1)=\phi (\theta (1)\wedge 1)=\varphi (1)\gg \varphi (\chi_{(0,1]})\gg \phi\theta (\chi_{(1/2,1]})
  \geq \phi(\theta (\chi_{(1/2,1]})\wedge 1),
 \]
where in the second inequality we have used $d(\varphi ,\phi\theta)<1/2$.

We obtain that $\varphi (\chi_{(0,1]})=1'$ is compact, a contradiction.
\end{exa}
%===========================================

Recall from \autoref{basic_ind} that an element in $\LscI$ is said to be basic if it can be written as a finite sum of basic indicator functions. In particular, a basic element $f$ satisfies $\min (f)=f(0)$ if and only if it can be written as a finite sum of elements of the form $1,\chi_{(\cdot ,1]}$ and $\chi_{(\cdot ,\cdot )}$. Similarly, $f$ is increasing if and only if it can be written as a finite sum of elements of the form $1$ and $\chi_{(\cdot ,1]}$.

\begin{lma}\label{Basic_reduc}
 Let $x\ll y$ be basic elements in $\LscI$. Then, there exists basic increasing elements $a$ and $d$ such that $x+a\ll d\ll y+a$.
 
 The same holds, component-wise, for elements $x\ll y$ in $\LscI^{M}_{\ll}$ for every $M\in\mathbb{N}$.
\end{lma}
\begin{proof}
 We will first prove the following claim.
 
\noindent\textbf{Claim} \textit{Given a basic element $x$ and a basic indicator $y$ such that $x\ll y\leq 1$ there exist basic increasing elements $a,d$ such that $x+a\ll d\ll y+a$.}
\begin{proof}
Write $x=\sum_{i} \chi_{U_{i}}$ and $y=\chi_{V}$ with $U_{i}, V$ intervals. Note that, since $x\leq 1$, we may assume the intervals $U_{i}$ to be pairwise disjoint.

If $V=[0,1]$ (i.e. $y=1$), we can take $a=0$ and $d=y$, so we assume otherwise.

If $V$ is of the form $[0,t)$ for some $t$, let $\epsilon >0$ be small enough so that $\sqcup_{i}U_{i}\Subset [0,t-\epsilon )$. Set $d=1$ and $a=\chi_{(t-\epsilon ,1]}$. One clearly has $\sum_{i}\chi_{U_{i}}+a\ll d\ll \chi_{V}+a$.

If there exists some $s$ so that $V=(s,1]$, let $\epsilon>0$ be such that $\sum_{i}\chi_{U_{i}}\ll \chi_{(s+\epsilon,1]}$. Set $a=0$ and $d=\chi_{(s+\epsilon,1]}$. 

Finally, if $V=(s,t)$ for some $s<t$, take $\epsilon,\delta >0$ small enough so that
\[
 \sqcup_{i}U_{i}\Subset (s+\epsilon,1], \andSep (\sqcup_{i}U_{i})\cap (t-\epsilon ,1]=\emptyset.
\]

Set $a=(t-\epsilon ,1]$ and $d=(s+\epsilon,1]$. By construction, one gets $\sum_{i}\chi_{U_{i}}+a\ll d\ll y+a$ as required.
\end{proof}

Given two basic elements $x\ll y$ in $\LscI$, we know that these can be written as $x=\sum_{i}\chi_{U_{i}}$ and $y=\sum_{i}\chi_{V_{i}}$ with $\chi_{U_{i}}\ll \chi_{V_{i}}$, where $\chi_{U_{i}}$ is a basic element (since $U_{i}$ may not be an interval) and $\chi_{V_{i}}$ is a basic indicator function.

For every $i$, the previous claim gives us basic increasing elements $a_{i},d_{i}$ such that 
\[
\chi_{U_{i}}+a_{i}\ll d_{i}\ll \chi_{V_{i}}+a_{i}.
\]

Set $a=\sum_{i} a_{i}$ and $d=\sum_{i} d_{i}$, where note that $x+a\ll d\ll y+a$. Since every sum of basic increasing elements is again basic and increasing, the result follows.

 To see that the same result holds in $\LscI^{M}$ for every $M$, simply apply everything componentwise.
\end{proof}
%===========================================
\begin{prp}\label{Further_red}
 Let $S$ be a countably based and compactly bounded $\Cu$-semigroup satisfying (O5) and weak cancellation.

 If conditions (i)-(iii) of \autoref{Main_OneRel} are satisfied for every \CuMor{} $\varphi$, for every $\epsilon>0$ and any triple of basic increasing elements $a,a',b$ such that $a\ll a'$ with $\varphi (a')\ll \varphi (b)$, then $S$ is $\Cu$-isomorphic to the Cuntz semigroup of an AI-algebra.
\end{prp}
\begin{proof}
Let $\varphi\colon \LscI^{r}\to S$ be a $\Cu$-morphism and take $\epsilon >0$.

 Also let $x,x',y$ be basic elements with $x\ll x'$ and $\varphi (x')\ll \varphi (y)$. We will prove that conditions (i)-(iii) in \autoref{Main_OneRel} hold for these elements. Since basic elements are sup-dense, this will show that $S$ is $\Cu$-isomorphic to the Cuntz semigroup of an AI-algebra.
 
 Take $y'\ll y$ with $\varphi (x')\ll \varphi (y')$. Now apply \autoref{Basic_reduc} to obtain basic increasing elements $d,f,a,b$ such that
 \[
   x+a\ll d\ll x'+a, \andSep y'+b\ll f\ll y+b,
 \]
where by an increasing element we mean an element that is increasing in each component.
 
 Thus, there exists $r,t>0$ such that $R_{2t}(a),R_{2r}(b)$ are still basic increasing elements with
 \[
   x+a\ll R(d)\ll d \leq x'+R_{2t}(a),\andSep 
   y'+b\ll f\leq y+R_{2r}(b)
 \]
for some retraction $R(d)$ of $d$.
 
Thus, we have the following
\[
 \varphi (d+b)\leq \varphi (x'+R_{2t}(a)+b)\leq 
 \varphi (y'+R_{2t}(a)+b) \ll \varphi (f+R_{t}(a)).
\]

Note that the elements $R(d)+R_{r}(b)$, $d+b$ and $f+R_{t}(a)$ are a triple of basic increasing elements such that
\[
 R(d)+R_{r}(b)\ll d+b,\andSep \varphi (d+b)\ll \varphi (f+R_{t}(a)).
\]

Thus, we get by assumption that there exist $\Cu$-morphisms $\theta$ and $\phi$ satisfying conditions (i)-(iii) in \autoref{Main_OneRel} for $x=R(d)+R_{r}(b)$ and $y=f+R_{t}(a)$. Using this at the second step, one obtains
\[
 \theta (x+a+R_{2r}(b))\ll \theta (R(d)+R_{r}(b))
 \ll \theta (f+R_{t}(a))
 \leq \theta (y+R_{2r}(b) +a).
\]

Applying weak cancellation to the previous inequality, we get $\theta (x)\ll \theta (y)$ as required.
\end{proof}
%===========================================
\begin{proof} 

[of \autoref{Triangular_Discrete}]

The forward implication follows trivially from \autoref{Main_OneRel} by taking $\epsilon <1/l$ and applying \autoref{basic_ind_dist} to prove condition (i).

To show the converse, we know by \autoref{Further_red} that it is enough to show that  conditions (i)-(iii) of \autoref{Main_OneRel} are satisfied for every morphism $\varphi\colon\LscI^M\to S$, for every $\epsilon>0$ and every triple of basic increasing elements $x,x',y$ such that $x\ll x'$ with $\varphi (x')\ll \varphi (y)$.

 Let $l\in\NN$ be large enough such that $2/l<\epsilon$ and $x,x',y\in B^M_l$. By assumption, we obtain $\Cu$-morphisms $\phi$ and $\theta$ satisfying conditions (i)-(iii) in \autoref{Triangular_Discrete}. In particular, condition (i) states that
 \[
  \varphi (\chi^{s}_{((i+1)/l,1]})\ll \phi\theta(\chi^{s}_{(i/l,1]}),\andSep 
  \phi\theta(\chi^{s}_{((i+1)/l,1]})\ll \varphi (\chi^{s}_{(i/l,1]})
 \]
for every $\chi^{s}_{(i/l,1]}\in C_{l}^{M}$.

Thus, it follows from \autoref{Partition} that $d(\varphi ,\phi\theta )< 1/l+1/l=2/l<\epsilon$. This shows that condition (i) of \autoref{Main_OneRel} is satisfied.

Conditions (ii) and (iii) from \autoref{Main_OneRel} and \autoref{Triangular_Discrete} coincide. Using \autoref{Further_red}, we have that $S$ is $\Cu$-isomorphic to the Cuntz semigroup of an AI-algebra, as desired.
\end{proof}
%===========================================

\section{An abstract characterization}\label{sec:AbsCha}

The aim of this section is to provide an abstract characterization for the Cuntz semigroups of AI-algebras using \autoref{MainTh}. The property used in this characterization will be Property $\rm{I}$, as defined below;  see \autoref{PropertyI}. One could possibly use \autoref{MainTh} to prove other, maybe simpler, characterizations; see \autoref{qst:SimplerCharac}.

\Msubsection{The sets $\Omega_{n}$ and $X_{n}$}{The sets On and Xn}
Let $\Omega_{n}=\{ -\infty ,0,\cdots ,n,\infty \}$. For $\alpha ,\alpha'\in\Omega_n$, we define $\alpha '\prec \alpha$ if $\alpha'=\alpha=-\infty$, or $\alpha '=\alpha=\infty$ or $\alpha '\leq \alpha$ with $\alpha '\neq \alpha$

Let $\text{udiag}(\Omega_{n}\times \Omega_{n})$ be the subset of $\Omega_{n}\times \Omega_{n}$ consisting of the pairs $(\alpha,\beta)$ with $\alpha\lneq \beta$. Define $X_{n}$ as  the free abelian monoid on $\text{udiag}(\Omega_{n}\times \Omega_{n})$. In $X_{n}$ we denote the unit by $(0,0)$.

\vspace{0.2cm}
Given $w,(\alpha ,\beta)\in X_{n}$, write $w\prec (\alpha,\beta)$ if and only if $w=(0,0)$ or else there exist elements $(\alpha_{i},\beta_{i})$ in $\text{udiag}(\Omega_{n}\times \Omega_{n})$ such that 
\[
\alpha\prec \alpha_{1}\prec \beta_{1}\prec \alpha_{2}\prec \cdots \prec \alpha_{m}\prec \beta_{m}\prec \beta
\]
and $w=\sum (\alpha_{i},\beta_{i})$. In particular, $(\alpha ',\beta ')\prec (\alpha ,\beta )$ if and only if $\alpha\prec \alpha '\prec \beta '\prec \beta$.

Also, set $(0,0)\prec (0,0)$.

More generally, given $w,v\in X_{n}$, we write $w\prec v$ if there exist (possibly repeated and zero) elements $w_{i}\in X_{n}$ and $(\alpha_{i},\beta_{i})\in \text{udiag}(\Omega_{n}\times \Omega_{n})\cup\{ (0,0)\}$  such that $w=\sum w_{i}$, $v=\sum (\alpha_{i},\beta_{i} )$ and $w_{i}\prec (\alpha_{i},\beta_{i} )$ as above for every $i$.

Note that $\prec$ is trivially compatible with addition.
%===========================================

\begin{lma}\label{Max_Elements}
 Let $n\in \mathbb{N}$ and let $w,v,(\alpha ,\beta )$ be elements in $X_n$. We have
 \begin{enumerate}[(i)]
 \item If $w=\sum (\gamma_j ,\delta_j)\prec (\alpha ,\beta )$, there exists at most one $\gamma_j$ with $\alpha= \gamma_j=-\infty $ (resp. at most one $\delta_j$ with $\beta =\delta_j=\infty $).
  \item $w+(-\infty ,\infty)\prec (-\infty ,\infty )$ if and only if $w=(0,0)$.
  \item $w+(-\infty ,\infty )\prec v+(-\infty ,\infty )$ if and only if $w\prec v$.
 \end{enumerate}
\end{lma}
\begin{proof}
 For (i), we may assume that $w=\sum_{j\leq m} (\gamma_j ,\delta_j)\prec (\alpha ,\beta )$ with $\alpha\prec\gamma_{1}\prec\cdots\prec \delta_{m}\prec\beta$ and $(\gamma_j ,\delta_j)\neq (0,0)$ for every $j$. If some $\gamma_j$ is such that $\gamma_j=-\infty$, we have 
 \[
 \alpha=\gamma_1=\delta_1=\cdots =\gamma_j=-\infty.
 \]
 However, we know that $\gamma_1\neq \delta_1$. This implies that $j=1$ and, consequently, that $\gamma_j$ is unique.
 
 An analogous argument proves that if $\delta_j=\infty$ for some $j$, we must have $j=m$.
 
 In particular this implies that, whenever $w\prec (\alpha ,\beta)$, $w$ has a summand of the form $(-\infty ,\infty )$ if and only if $w=(-\infty ,\infty )=(\alpha ,\beta )$.
 
 Thus, it follows that $w+(-\infty ,\infty )\prec (-\infty ,\infty )$ if and only if $w=(0,0)$.
 
 Finally, to prove (iii), let $v'=v+(-\infty ,\infty )$ and write $v'=\sum_{i\leq k}(\alpha_i ,\beta_i )$ and $w=\sum_{i\leq k} w_i$ in such a way that $w_{1}+(-\infty ,\infty )\prec (\alpha_1,\beta_1)$ and $w_i\prec (\alpha_i,\beta_i)$ for $i\geq 2$.
 
 From the first inequality and the definition of $\prec$ we have $w_1=(0,0)$ and $(\alpha_1,\beta_1)=(-\infty ,\infty )$. Therefore, we get $v=\sum_{i\geq 2}(\alpha_i ,\beta_i)$.
 
 Using that $\prec$ is trivially compatible with addition, it follows that 
 \[
 w=\sum_{i\geq 1}w_i=\sum_{i\geq 2}w_i\prec \sum_{i\geq 2}v_i=v
 \]
 as required.
\end{proof}
%===========================================

\begin{lma}\label{lma_prec_properties}
 With the above notation, $\prec $ is a transitive and antisymmetric relation.
\end{lma}
\begin{proof}
First note that, given any two pairs $w_{1}\prec v_{1}$ and $w_{2}\prec v_{2}$, one clearly has $w_{1}+w_{2}\prec v_{1}+v_{2}$.
 
 Thus, let $w\prec v$ and $v\prec u$, which by definition means that $w=\sum w_{i}$, $v=\sum (\alpha_{i},\beta_{i})= \sum_{j}v_{j}$ and $u=\sum_{j}(\gamma_{j},\delta_{j})$ such that
 \[
  w_{i}\prec (\alpha_{i},\beta_{i}),\andSep
  v_{j}\prec (\gamma_{j},\delta_{j})
 \]
for every $i,j$, and where we may assume $(\alpha_{i},\beta_{i}), (\gamma_{j},\delta_{j})$ non-zero for every $i,j$.

Also, for every $w_{i}\neq 0$, write $w_{i}=\sum_{k} (\zeta_{k,i}, \eta_{k,i})$. Define $\min(w_{i})$ and $\max(w_{i})$ as the minimum of the $\zeta_{k,i}$ and the maximum  of the $\eta_{k,i}$ with respect to $\prec$ respectively. Note that such elements exist by the definition of $w_{i}\prec (\alpha_{i},\beta_{i})$.

For every $j$, let $I_{j}$ be such that $\sum_{i\in I_{j}}(\alpha_{i},\beta_{i})=v_{j}$, and define $\tilde{I_{j}}=\{ i\in I_{j}\mid w_{i}\neq 0 \}$. 
Using our first observation, we get 
\[
 \sum_{i\in \tilde{I}_{j}}w_{i}=\sum_{i\in I_{j}}w_{i}\prec \sum_{i\in I_{j}}(\alpha_{i},\beta_{i})\prec (\gamma_{j},\delta_{j}).
\]

Note that if $I_{j}$ is empty, so is $\tilde{I}_{j}$, and so $\sum_{i\in I_{j}}w_{i}=(0,0)\prec (\gamma_{j},\delta_{j})$. Else, if $I_{j}$ is not empty, there exists an ordering $i_{1},\cdots ,i_{r}$ of $\tilde{I}_{j}$ such that
\[
\begin{split}
 \gamma_{j}&\prec
 \alpha_{i_{1}}\prec
 \min(w_{i_{1}})\prec 
 \zeta_{k,i_{1}}\prec \eta_{k,i_{1}}
 \prec
 \max(w_{i_{1}})\prec
 \beta_{i_{1}}\\
 &\prec
 \cdots
 \prec \alpha_{i_{r}}\prec
 \min(w_{i_{r}})\prec 
 \zeta_{k,i_{r}}\prec \eta_{k,i_{r}}
 \prec
 \max(w_{i_{r}})\prec
 \beta_{i_{r}}\prec
 \delta_{j}
 \end{split}
\]

Thus, given $\tilde{w_{j}}=\sum_{i\in I_{j}}w_{i}$, we have $\tilde{w_{j}}\prec (\gamma_{j},\delta_{j})$. Since $w=\sum \tilde{w_{j}}$, it follows that $w\prec u$. This shows that $\prec$ is transitive.

\vspace{0.2cm}
To see that $\prec$ is antisymmetric first note that, by transitivity and (i)-(ii) in \autoref{Max_Elements}, given elements $w=(\alpha ,\beta)$ and $v\in X_{n}$, we have $w\prec v\prec w$ if and only if $w=v=(-\infty , \infty )$ or if $w=v=0$.

Let $w\in X_{n}$ and let $m$ be the number of non-zero summands of $w$. We will now prove by induction on $m$ that $w\prec v\prec w$ if and only if $w=v=m(-\infty ,\infty )$. Thus, assume that for a fixed $m$ we have proven that $w\prec v\prec w$ implies $w=v=(m-1)(-\infty ,\infty )$ for any $w$ having $m-1$ non-zero summands and any $v$.

Take $w=\sum_{i\leq m}(\alpha_{i},\beta_{i})$ and $v=\sum_{j\leq k}(\gamma_{j},\delta_{j})$, where we assume that all summands are non-zero. Then, set $i_{1}=1$ and find $j_{1}$ such that $(\alpha_{i_{1}},\beta_{i_{1}})\prec (\gamma_{j_{1}},\delta_{j_{1}})$. This can be done because $w\prec v$.

Now let $i_{2}$ be such that
\[
 (\alpha_{i_{1}},\beta_{i_{1}})\prec (\gamma_{j_{1}},\delta_{j_{1}})\prec 
 (\alpha_{i_{2}},\beta_{i_{2}}),
\]
which exists because $v\prec w$. Note that, if $i_{2}=i_{1}$, we would have 
\[
 (\alpha_{i_{2}},\beta_{i_{2}})\prec
(\gamma_{j_{1}},\delta_{j_{1}})\prec
(\alpha_{i_{2}},\beta_{i_{2}}).
\]

This in turn would imply $(\alpha_{i_{2}},\beta_{i_{2}})=(\gamma_{j_{1}},\delta_{j_{1}})=(-\infty ,\infty )$. By \autoref{Max_Elements} (iii), we could cancel this summand from $w$ and $v$ and apply induction to conclude that $w=v=m(-\infty ,\infty )$.

Thus, we assume $i_{2}\neq i_{1}$. Following this construction, one obtains an ordering $i_{1},\cdots ,i_{n}$ of $\{ 1,\cdots ,n \}$ and pairwise different integers $j_{1},\cdots ,j_{n}$ such that
\[
 (\alpha_{i_{1}},\beta_{i_{1}})\prec (\gamma_{j_{1}},\delta_{j_{1}})\prec
 \cdots\prec
 (\alpha_{i_{n}},\beta_{i_{n}})\prec (\gamma_{j_{n}},\delta_{j_{n}}).
\]

Since $v\prec w$, there must exists some $s\leq n$ such that $(\gamma_{j_{n}},\delta_{j_{n}})\prec (\alpha_{s},\beta_{s})$ and, since $i_{1},\cdots ,i_{n}$ is an ordering of $\{ 1,\cdots ,n \}$, we must have $s=j_{l}$ for some $l\leq n$.

This implies
\[
 (\alpha_{i_{l}},\beta_{i_{l}})\prec (\gamma_{j_{l}},\delta_{j_{l}})\prec
 (\alpha_{i_{l}},\beta_{i_{l}})
\]
and, by the same argument as before, we get $w=v=m(-\infty ,\infty )$ as desired.
\end{proof}
%===========================================

In $X_{n}$, we write $w\approxeq v$ if $w=v$ or else there exist $z\in X_{n}$ and $\alpha\prec\gamma\prec\beta\prec\delta $ in $\Omega_{n}$ such that
\[
w=z+(\alpha ,\beta)+(\gamma ,\delta ),\andSep  v= z+ (\alpha ,\delta )+(\gamma ,\beta )
\]
or
\[
 v=z+(\alpha ,\beta)+(\gamma ,\delta ) ,\andSep  w= z+ (\alpha ,\delta )+(\gamma ,\beta ).
\]

Note that, with $z=0$, this yields $(\alpha ,\beta)+(\gamma ,\delta )\approxeq (\alpha ,\delta )+(\gamma ,\beta )$.

Given two elements $w,v\in X_{n}$, we write $w\simeq v$ if and only if there exists $w_{1},\cdots ,w_{m}\in X_{n}$ such that $w\approxeq w_{1}\approxeq\cdots\approxeq w_{m}\approxeq v$.

This construction tries to mimic the relation defined in \autoref{Free_basic_el}.

\begin{lma}\label{lma_simeq_properties}
 $\simeq$ is an equivalence relation compatible with addition.
\end{lma}
\begin{proof}
 First note that $\simeq$ is transitive,  symmetric and reflexive by construction, so we only need to prove that it is compatible with addition.
 
 To see this, first take $w\approxeq v$ with $w\neq v$, where we may assume that there exist $z\in X_{n}$ and $\alpha\prec\gamma\prec\beta\prec\delta $ in $\Omega_{n}$ such that
 \[
   w=z+(\alpha ,\beta)+(\gamma ,\delta ),\andSep  v= z+ (\alpha ,\delta )+(\gamma ,\beta ).
 \]

 Thus, given any $w'\in X_{n}$, we have
 \[
   w+w'=(z+w')+(\alpha ,\beta)+(\gamma ,\delta ),\andSep  v+w'= (z+w')+ (\alpha ,\delta )+(\gamma ,\beta ),
 \]
 which shows $w+w'\approxeq v+w'$. Trivially, if $w=v$, we also have $w+w'\approxeq v+w'$ for any $w'$.
 
 Therefore, given $w\approxeq v$ and $w'\approxeq v'$, we get $w+w'\approxeq v+w'\approxeq v+v'$, which implies $w+w'\simeq v+v'$.
 
 Now let $w\simeq v$ and $w'\simeq v'$. Then, there exist $w_1,\cdots ,w_m$ and $w'_1,\cdots ,w'_{m'}$ such that
 \[
   w\approxeq w_{1}\approxeq\cdots\approxeq w_{m}\approxeq v, \andSep
   w'\approxeq w'_{1}\approxeq\cdots\approxeq w'_{m'}\approxeq v'.
 \]
 
Since $w\approxeq v$ whenever $w=v$, we may assume $m=m'$. Thus, we have
\[
 w+w'\approxeq w_1+w_1'\approxeq\cdots\approxeq w_m+w_m'\approxeq v+v',
\]
which implies $w+w'\simeq v+v'$ as required.
\end{proof}
%===========================================

\subsection{Chainable subsets}
Given a \CuSgp{} $S$ with weak cancellation, we will say that an additive map $F \colon X_{n}\to S$ is an I-morphism if $F (q)\ll F (t)$ whenever $q\prec t$ and $F(q)=F(t)$ whenever $q\simeq t$.

Note that this implies $F((0,0))+F((-\infty ,\infty))=F((-\infty ,\infty)) \ll F((-\infty ,\infty))$ and, by weak cancellation, $F((0,0))=0$.

\begin{dfn}\label{chainable_subset}
Let $H$ be a subsemigroup of a \CuSgp{} $S$, and let $e$ be a compact element in $H$. We say that $H$ is an $(n,e)$-chainable subset if there exists a surjective I-morphism $F\colon X_{n}\to S$ with $F(X_n)\subset H$ satisfying the following properties:
\begin{enumerate}[(i)]
 \item $F((-\infty ,\infty ))=e$
 \item $F((\alpha ,\beta ))+ F((\beta ,\infty ))\leq F((\alpha ,\infty ))$
 \item For every $m\geq 1$ there exists an I-morphism $F'\colon X_{mn}\to S$ such that $F'((m\alpha ,m\beta ))=F((\alpha ,\beta ))$ for every $\alpha ,\beta\in\Omega_{n}$ with $\alpha\prec\beta$.
\end{enumerate}
\end{dfn}
%===========================================

\begin{rmk}\label{morphisms}
 Given a Cu-morphism $\varphi\colon S\to T$ and an $(n,e)$-chainable subsemigroup $H$ of $S$, it is clear that $\varphi (H)$ is a $(n,\varphi (e))$-chainable subset of $T$.
\end{rmk}

\begin{rmk}
 Let $H$ be a subsemigroup of $S$. Then, $H$ is $(n,e)$-chainable if and only if $H$ is the additive span of a finite subset $H_{0}$, containing and bounded by $e$, such that there exists a surjective map $F_{0}\colon \text{udiag}(\Omega_{n}\times\Omega_{n})\cup\{(0,0)\}\to H_{0}$ satisfying the following properties:
 \begin{enumerate}[(i)]
  \item $F_{0}(-\infty ,\infty )=e$.
  \item $F_{0}((\alpha ,\beta ))+ F_{0}((\beta ,\infty ))\leq F_{0}((\alpha ,\infty ))$.
  \item $F_{0}((\alpha ,\beta ))+F_{0}((\gamma ,\delta ))=F_{0}((\alpha ,\delta ))+F_{0}((\gamma ,\beta))$ whenever $\alpha\prec\gamma\prec\beta\prec\delta $.
  \item $\sum_{i=1}^{m}F_{0}((\alpha_{i},\beta_{i}))\ll F_{0}((\alpha ,\beta))$ whenever $\sum_{i=1}^{m} (\alpha_{i},\beta_{i})\prec (\alpha ,b)$.
  \item For every $m\geq 1$ there exists an I-morphism $F'\colon X_{mn}\to S$ such that $F'((m\alpha ,m\beta ))=F((\alpha ,\beta ))$ for every $\alpha ,\beta\in\Omega_{n}$.
 \end{enumerate}
\end{rmk}
%===========================================

\subsubsection{Examples}
\begin{exa}\label{algebraic}
 The additive span of a compact element $e$ (i.e. the set of finite multiples of $e$) is $(n,e)$-chainable for every $n$.
 \end{exa}
 \begin{proof}
  Given any $n\geq 0$, one can define the additive map $F\colon X_{n}\to H$ as $F((\alpha ,\beta ))=e$ if $\beta =\infty$ and $F((\alpha ,\beta ))=0$ otherwise.

  Using \autoref{Max_Elements} (i), it is easy to see that, if $w\prec v$, then the number of summands of the form $(\alpha ,\infty )$ of $w$ must be less than or equal to that of $v$. This shows $F(w)\ll F(v)$ whenever $w\prec v$.

  Also, given $\alpha\prec \gamma\prec \beta\prec \delta$ in $\Omega_{n}$, either $\delta \neq \infty$, in which case
  \[
   F((\alpha ,\beta)+(\gamma ,\delta ))=0=F((\alpha ,\delta )+(\gamma ,\beta )),
  \]
 or $\delta =\infty$ and
 \[
   F((\alpha ,\beta)+(\gamma ,\delta ))=e=F((\alpha ,\delta )+(\gamma ,\beta )).
  \]
  This implies $F(w)=F(v)$ whenever $w\simeq v$ and, consequently, that $F$ is an I-morphism.
  
  Note that properties (i) and (ii) of  \autoref{chainable_subset} follow by construction and that, for property (iii), we can simply define $F'\colon X_{nm}\to H$ analogously to $F$.
 \end{proof}
%===========================================

Let $S=\LscI$, and let $L_{n}$ be the additive span of $\{ 1,\chi_{(i/n,j/n)},\chi_{(i/n,1]},\chi_{[0,j/n)}\}_{i,j}$. We will now show that $L_{n}$ is $(n,1)$-chainable.

\begin{ntn}\label{Notation}
 Given two elements $\alpha\prec\beta$ in $\Omega_{n}$, we denote by $(\alpha/n,\beta/n)$ the interval 
 \[
 (\alpha/n,\beta/n)\cap [0,1]
 \]
 in $[0,1]$. 
 
 For example, $(\alpha/n,\infty /n)$ corresponds to $(\alpha/n,1]$ whenever $\alpha\neq -\infty$ and $(-\infty /n,\infty /n)$ corresponds to $[0,1]$.
\end{ntn}
%===========================================

\begin{prp}\label{Linear_span}
 Let $S=\LscI$ and let $n\in\mathbb{N}$. Then, $L_{n}$ is $(n,1)$-chainable.
 \end{prp}
 \begin{proof}
 Following \autoref{Notation}, define the additive map $F\colon X_{n}\to L_{n}$ as $F((\alpha ,\beta))=\chi_{(\alpha /n,\beta/n)}$. We will now check that $F$ is an I-morphism: 
 
 First, note that given $\alpha,\gamma\prec \beta ,\delta$ with $F((\alpha ,\beta)+(\gamma ,\delta ))=F((\alpha ,\delta )+(\gamma ,\beta))$ if and only if 
 \[
 \begin{split}
  (\alpha /n,\beta/n)\cup (\gamma /n,\delta /n) &=(\alpha /n ,\delta /n)\cup (\gamma /n,\beta /n),\\
  (\alpha /n,\beta /n)\cap (\gamma /n,\delta /n) &=(\alpha /n ,\delta /n)\cap (\gamma /n,\beta /n).
 \end{split}
 \]
 
 Using that $F$ is additive, one can check that this is equivalent to  
 $(\alpha ,\beta)+(\gamma ,\delta )\approxeq (\alpha ,\delta )+(\gamma ,\beta)$. Thus, using the additivity of $F$ once again, we get $F(q)=F(t)$ whenever $q\simeq t$.
 
 Now, let $\sum_{i=1}^{m} (\alpha_{i},\beta_{i})\prec (\alpha,\beta )$, which by definition implies that, after a possible reordering, we have $\alpha\prec \alpha_{1}\prec \beta_{1}\prec \alpha_{2}\prec\cdots\prec \alpha_{m}\prec \beta_{m} \prec \beta$. This implies 
 \[
 \cup (\alpha_{i}/n,\beta_{i}/n)\Subset (\alpha/n,\beta /n) \andSep (\alpha_{i}/n,\beta_{i}/n)\cap (\alpha_{j}/n,\beta_{j}/n)=\emptyset
 \]
 for every pair $i\neq j$.
 
 Thus, one gets
 \[
  F\left(\sum_{i=1}^{m} (\alpha_{i},\beta_{i})\right)=
  \sum_{i=1}^{m}\chi_{(\alpha_{i}/n,\beta_{i}/n)}\ll
  \chi_{(\alpha /n,\beta /n)}
  = F((\alpha ,\beta )),
 \]
 which shows that $F(q)\ll F(t)$ whenever $q\prec t$. Indeed, we know by definition that $q\prec t$ if and only if $q=\sum_{i\leq k} w_i$ and $t=\sum_{i\leq k} (\alpha_i ,\beta_i)$ with $w_i\prec (\alpha_i ,\beta_i)$ for every $i$. By our previous argument we have $F(w_i)\ll v_i$ and, since $F$ is additive, it follows that $F(q)\ll F(t)$.
 \vspace{0.2cm}
 
Finally, note that conditions (i)(for $e=1$) and (ii) in \autoref{chainable_subset} are satisfied by construction, and that for every $m\geq 1$ we can define the map $F'\colon X_{nm}\to L_{nm}$ analogously to $F$. This shows that $L_{n}$ is $(n,1)$-chainable.
\end{proof} 
%===========================================

Our aim now is to show that the I-morphism $F$ defined in the previous proposition satisfies $F(q)=F(t)$ if and only if $q\simeq t$.

\begin{lma}\label{lma_cup_cap}
 Let $U,V$ be open sets of $[0,1]$ such that
 \[
  U=\sqcup_{i\leq N}(\alpha_{i}/n,\beta_{i}/n),\andSep
  V=\sqcup_{j\leq M}(\gamma_{j}/n,\delta_{j}/n)
 \]
 with $\alpha_{i},\beta_{i},\gamma_{j},\delta_{j}\in\Omega_{n}$ and the convention of \autoref{Notation}.
 
 Then, given the finite number of elements $((\epsilon_{k}, \zeta_{k}))_{k}, ((\eta_{l}, \theta_{l}))_{l}\in X_{n}$ such that
 \[
   U\cup 
     V=
   \sqcup_{k} (\epsilon_{k}/n, \zeta_{k}/n),\andSep 
   U\cap 
   V=
   \sqcup_{l} (\eta_{l}/n, \theta_{l}/n),
 \]
we have $\sum_{i}(\alpha_{i},\beta_{i})+\sum_{j}(\gamma_{j},\delta_{j})\simeq \sum_{k}(\epsilon_{k}, \zeta_{k})+\sum_{l}(\eta_{l}, \theta_{l})$.
\end{lma}
\begin{proof}
 We will prove the lemma by induction on $M$, the number of connected components of $V$.
 
 Thus, first assume that $M=1$. If $U\cap V=\emptyset $, $U\subset V$ or $V\subset U$, the result follows trivially (and we actually get an equality instead of $\simeq$).
 
 If none of the above cases happens, note that there is at most one pair $(\alpha_{s},\beta_{s})$ such that
 \[
  \alpha_{s}\prec \gamma_{1}\prec \beta_{s}\prec \delta_{1}
 \]
and at most one pair $(\alpha_r ,\beta_r)$ such that
\[
 \gamma_{1}\prec \alpha_{r}\prec \delta_{1} \prec
 \beta_{r}.
\]

We will prove the result assuming that both such pairs exist, since the remaining cases can be checked analogously. Also, without loss of generality, we can assume that our first pair is
$(\alpha_{1},\beta_{1})$. We let the other pair be $(\alpha_{r},\beta_{r})$.

Then, by the definition of $\simeq$ and since it is compatible with addition (see \autoref{lma_simeq_properties}), we have 
\[
\begin{split}
 (\alpha_{1},\beta_{1})+(\gamma_{1},\delta_{1})+(\alpha_{r},\beta_{r})
 &\simeq
 (\alpha_{1},\delta_{1}) + (\gamma_{1},\beta_{1}) + (\alpha_{r},\beta_{r})\\
 &\simeq
 (\alpha_{1},\beta_{r}) + (\alpha_{r},\delta_{1}) + (\gamma_{1},\beta_{1})
 \end{split}
\]

Thus, we get
\[
 \sum_{i=1}^{N} (\alpha_{i},\beta_{i})+(\gamma_{1},\delta_{1})\simeq 
 \left( \sum_{i>r}(\alpha_{i},\beta_{i}) +(\alpha_{1},\beta_{r}) \right)+\left( 
 \sum_{1<i<r}(\alpha_{i},\beta_{i}) +
 (\gamma_{1},\beta_{1})+(\alpha_{r},\delta_{1} )
 \right)
\]
which gives us our required elements.
 
 Now fix some $M$ and assume that the result holds for every $U$ and $V$ as in the statement of the lemma with $V$ having at most $M-1$ connected components.
 
 Consider $U=\sqcup_{i\leq N}(\alpha_{i}/n,\beta_{i}/n)$ and $V=\sqcup_{j\leq M}(\gamma_{j}/n,\delta_{j}/n)$. Apply the induction hypothesis to $U$ and $V'=\sqcup_{j\leq M-1}(\gamma_{j}/n,\delta_{j}/n)$, so as to obtain elements $((\epsilon_{k}, \zeta_{k}))_{k}, ((\eta_{l}, \theta_{l}))_{l}$ in $X_{n}$ such that
 \[
   U\cup 
     V' =
   \sqcup_{k} (\epsilon_{k}/n, \zeta_{k}/n),\andSep 
   U\cap 
   V'= \sqcup_{l} (\eta_{l}/n, \theta_{l}/n),
 \]
 satisfying $\sum (\alpha_i,\beta_i) +\sum (\gamma_j ,\delta_j )\simeq \sum (\epsilon_k, \zeta_{k})+ \sum ((\eta_{l}, \theta_{l}))$, which we call relation (1).
 
 Note that $(\eta_{l}/n, \theta_{l}/n)\cap (\gamma_{M}/n,\delta_{M}/n)=\emptyset $ for every $l$, as $V'\cap (\gamma_{M}/n,\delta_{M}/n)=\emptyset$.
 
 Apply now the induction hypothesis to $\sqcup_{k}(\epsilon_{k}/n, \zeta_{k}/n)$ and $(\gamma_{M}/n,\delta_{M}/n)$. Thus, we obtain elements $((\epsilon'_{k'}, \zeta'_{k'}))_{k'}, ((\eta'_{l'}, \theta'_{l'}))_{l'}$ satisfying 
 \[
  \begin{split}
   \Big( \sqcup_{k}(\epsilon_{k}/n,\zeta_{k}/n)\Big)\cup (\gamma_{M}/n,\delta_{M}/n)&=
   \sqcup_{k'} (\epsilon'_{k'}/n, \zeta'_{k'}/n),
   \\
   \Big( \sqcup_{k}(\epsilon_{k}/n,\zeta_{k}/n)\Big)\cap (\gamma_{M}/n,\delta_{M}/n)&=
   \sqcup_{l'} (\eta'_{l'}/n, \theta'_{l'}/n),
  \end{split}
 \]
 and $(\gamma_{M},\delta_{M})+ \sum (\epsilon_{k},\zeta_{k})\simeq\sum  (\epsilon'_{k'}, \zeta'_{k'})+\sum (\eta'_{l'}, \theta'_{l'})$, which we will refer to as relation (2).
 
 Therefore, we get
 \[
  U\cup 
     V
     =
     \Big( \sqcup_{k}(\epsilon_{k}/n,\zeta_{k}/n)\Big)\cup (\gamma_{M}/n,\delta_{M}/n)\\
     =
     \sqcup_{k'}(\epsilon'_{k'}/n, \zeta'_{k'}/n)
 \]
 and 
 \[
 \begin{split}
     U\cap 
   V
   &=
  \Big(\sqcup_{l} (\eta_{l}/n, \theta_{l}/n)\Big) \sqcup    
  \Big( U\cap (\gamma_{M}/n,\delta_{M}/n) \Big)\\
     &=
     \Big(\sqcup_{l} (\eta_{l}/n, \theta_{l}/n)\Big) \sqcup
     \Big( (U\cup V')\cap (\gamma_{M}/n,\delta_{M}/n) \Big)\\
     &=
     \Big(\sqcup_{l} (\eta_{l}/n, \theta_{l}/n)\Big)\sqcup
  \Big(\sqcup_{l'}(\eta'_{l'}/n, \theta'_{l'}/n)\Big).
 \end{split}
\]
Now using relations (1), (2) and the fact that $\simeq $ is compatible with addition, we get:
\[
 \begin{split}
  \sum_{i}(\alpha_{i},\beta_{i})+\sum_{j<M}(\gamma_{j},\delta_{j}) + (\gamma_{M},\delta_{M}) &\simeq 
  \sum_{k}(\epsilon_{k}, \zeta_{k})+\sum_{l}(\eta_{l}, \theta_{l}) + (\gamma_{M},\delta_{M})\\
  &\simeq \sum_{k'}(\epsilon'_{k'}, \zeta'_{k'}) + \left(\sum_{l'}(\eta'_{l'}, \theta'_{l'}) + \sum_{l}(\eta_{l}, \theta_{l}) \right),
  \end{split}
 \]
 as desired.
\end{proof}
%===========================================

Recall that $L_n$ is the additive span of $\{ \chi_{(\alpha /n,\beta /n)}\}_{\alpha,\beta\in\Omega_n}$. Thus, given any element $f\in L_n$ and any $k\leq \sup (f)$, there exist $\alpha_{i,k},\beta_{i,k}\in\Omega_{n}$ such that
\[
 \{ f\geq k\}=\sqcup_{i} (\alpha_{i,k}/n,\beta_{i,k}/n).
\]

For every $f\in L_n$, we define the element $q_{f}\in X_{n}$ as $q_{f}=\sum_{i,k} (\alpha_{i,k},\beta_{i,k})$.

Also recall that, by \autoref{Linear_span}, $L_n$ is $(n,1)$-chainable and that the I-morphism $F\colon X_n\to L_n$ defined by $F((\alpha ,\beta ))=\chi_{(\alpha /n,\beta /n)}$ is surjective.

\begin{lma}\label{Canonical_element}
  With the above notation, let $q\in X_n$ and $f=F(q)$. Then, $q\simeq q_{f}$. 
 
 In particular, since $\simeq$ is transitive, one has $F(q)=F(t)$ if and only if $q\simeq t$.
\end{lma}
\begin{proof} 
 Let $q=\sum_{i\leq m} (\epsilon_{i},\zeta_{i})$ and $f=F(q)$. As above, let $\alpha_{i,k},\beta_{i,k}\in\Omega_{n}$ be such that
 \[
 \{ f\geq k\}=\sqcup_{i} (\alpha_{i,k}/n,\beta_{i,k}/n).
\]
for every $k$. We will prove the result by induction on $m$, the amount of summands of $q$:
 
 For $m=1$, the result is trivial since $q=q_{f}$. Thus, fix some $m$ and assume that the result has been proven for every element with $m-1$ summands. Then, let $q=\sum_{i\leq m} (\epsilon_{i},\zeta_{i})$ and, since $\simeq$ is compatible with  addition (see \autoref{lma_simeq_properties}), note that
 \[
  q\simeq q_{F(q')}+ (\epsilon_{1},\zeta_{1}),
 \]
where $q'=\sum_{1<i\leq m} (\epsilon_{i},\zeta_{i})$.

By construction, the element $q_{F(q')}$ can be written as a sum $\sum_{k}\sum_{j\in J_{k}}(\gamma_{j,k},\delta_{j,k})$ such that 
\[
\{F(q')\geq k\} = \sqcup_{j\in J_{k}}(\gamma_{j,k}/n,\delta_{j,k}/n).
\]

Set $v_{k}=\sum_{j\in J_{k}}(\gamma_{j,k},\delta_{j,k})$. By \autoref{lma_cup_cap} applied to $U=\{F(q')\geq 1\}$ and $V=(\epsilon_{1}/n,\zeta_{1}/n)$, we get elements  $w_{1}$ and $u_{1}$, corresponding to $U\cup V$ and $U\cap V$ respectively, such that 
\[
v_{1}+(\epsilon_{1},\zeta_{1})\simeq
w_{1}+u_{1}.
\]

Thus, we have $q\simeq q_{F(q')}+(\epsilon_1 ,\zeta_1 )=v_1+(\epsilon_1 ,\zeta_1 )+\sum_{1<k}v_k\simeq w_{1}+u_{1}+\sum_{1<k}v_k$.

Applying $F$ to $q\simeq q_{F(q')}+(\epsilon_1,\zeta_1)$, one gets $f=F(q')+\chi_{(\epsilon_1/n,\zeta_1/n )}$. Consequently $U\cup V= \{f\geq 1\}$ and therefore $w_{1}=\sum_{i}(\alpha_{i,1},\beta_{i,1})$. We also  get  $F(u_{1})=\chi_{U\cap V}=\chi_{\{F(q')\geq 1\}\cap (\epsilon_{1}/n,\zeta_{1}/n)}$.

Now apply \autoref{lma_cup_cap} again to $U'=U\cap V$ and $V'=\{ F(q')\geq 2\}$, thus getting $w_{2},u_{2}$, corresponding to $U'\cup V'$ and $U'\cap V'$ respectively, such that $u_1+v_2\simeq w_2+u_2$. Once again, we  have $w_{2}=\sum_{i}(\alpha_{i,2},\beta_{i,2})$ 
 because 
 \[
  U'\cup V'=\{ F(q')\geq 2\}\cup (\{F(q')\geq 1\}\cap (\epsilon_{1}/n,\zeta_{1}/n))
  =\{ f\geq 2\},
 \]
and therefore $q\simeq w_1+u_1+v_2+\sum_{k>2}v_k\simeq w_1+w_2+u_2+\sum_{k>2}v_k$.

Repeating this process a finite number of times, one gets
\[
\begin{split}
 q &\simeq q_{F(q')}+ (\epsilon_{1},\zeta_{1})
 \simeq w_{1}+u_{1}+\sum_{k\geq 2}v_{k}\\
 &\simeq
 w_{1} + w_{2} + u_{2}+\sum_{k\geq 3}v_{k}
 \simeq \cdots \simeq 
 w_{1}+\cdots +w_{k+1}= q_{f}
 \end{split}
\]
as desired.
\end{proof}
%===========================================

\begin{rmk}
With the notation of \autoref{Linear_span}, given $q,t$ with $F(q)\ll F(t)$, we do not necessarily have that $q\simeq q'\prec t'\simeq t$.
 
 Consider, for example, the elements $q=(1,2)+(2,3)$ and $t=(0,4)$ in $X_{4}$. Clearly, we have $F(q)\ll F(t)$ but it is not true that $q\simeq q'\prec t'\simeq t$.
\end{rmk}
%===========================================

Following the previous remark, we will now define a subset $L_{n}^{0}$ of $L_{n}$ such that, whenever $q,t\in X_{n}$ satisfy $F(q)\ll F(t)$, we have $q\simeq q'\prec t'\simeq t$.

First, we note the following:
\begin{rmk}\label{Rmk_Retraction}
 Let $\chi_{U},\chi_{V}\in L_{n}$ with $U,V$ intervals (i.e. of the form $(\alpha /n,\beta /n)$ for some $\alpha ,\beta\in\Omega_{n}$) and let $0< \epsilon<1/2n$. It is easy to check that $R_{\epsilon }(U\cup V)=R_{\epsilon }(U)\cup R_{\epsilon }(V)$ and $R_{\epsilon }(U\cap V)=R_{\epsilon }(U)\cap R_{\epsilon }(V)$, where $R_{\epsilon}(\cdot )$ denotes the $\epsilon$-retraction of unions of intervals as defined in \autoref{retract_notation}.
 
 Thus, given any $n>0$ and $f\in L_{n}$ with $f=\sum\chi_{U_{i}}$, we have that
 \[
 \cup_{\vert I\vert =k}\cap_{i\in I}R_{\epsilon } (U_{i})=
 R_{\epsilon }(\cup_{\vert I\vert =k}\cap_{i\in I}U_{i})=
 R_{\epsilon }(\{ f\geq k\}),
 \]
 and, consequently,
 \[
 \sum_i\chi_{R_{\epsilon } (U_{i})}=
 \sum_k \chi_{\cup_{\vert I\vert =k}\cap_{i\in I}R_{\epsilon } (U_{i})}=
   \sum_{k}\chi_{R_{\epsilon }(\{ f\geq k\})}
 \]

 This implies that, for every $0< \epsilon <1/2n$, the function $R_{\epsilon }(f)=\sum\chi_{R_{\epsilon } (U_{i})}$ does not depend on the expression of $f$ as $\sum \chi_{U_{i}}$. Note, however, that $R_{\epsilon }(f)$ may no longer belong to  $L_{n}$.
 
 Also note that, for every $f,g\in L_{n}$, we have $R_{\epsilon }(f)+R_{\epsilon }(g)=R_{\epsilon }(f+g)$.
\end{rmk}
%===========================================

\begin{dfn}
 For every $n>0$, define $L_{n}^{0}$ as the subset of $L_{n}$ consisting of the functions $f\in L_{n}$ such that, for every $k\leq \sup (f)$, the connected components (i.e. intervals) of the open subset $\{ f\geq k \}$ are at pairwise distance at least $2/n$.
\end{dfn}
%===========================================
\begin{lma}\label{Red_L0}
 Let $f$ be an element in $L_{n}$. For every rational $0< \epsilon<1/2n$, there exists some $m\in\mathbb{N}$ such that $R_{\epsilon }(f)\in L_{m}^{0}$.
\end{lma}
\begin{proof}
 Since $f\in L_{n}$, we know that $\{f\geq k\}$ is a finite disjoint union of open intervals for every $k\leq \sup (f)$. Thus, given any $\epsilon >0$, the $\epsilon $-retractions of these intervals are at pairwise distance at least $2 \epsilon $.
 
 If, additionally, $\epsilon <1/2n$ and rational, we can write $\epsilon =m_{1}/m_{2}$ for some $m_{1}$ and $m_{2}>1$. Set $m=nm_{2}>n$, and note that $R_{\epsilon }(f)\in L_{m}$.
 
 Further, the connected components of $\{R_{\epsilon }(f)\geq k\}$ are at pairwise distance at least $4\epsilon =4m_{1}/m_{2}>2/m$, as required.
\end{proof}
%===========================================

\begin{lma}\label{L0_comp}
Let $f,g\in L_{n}^{0}$. Then, $f\ll g$ if and only if $q_{f}\prec q_{g}$, where $q_{f},q_{g}$ are the elements defined in \autoref{Canonical_element}.
\end{lma}
\begin{proof}
 Let $f,g\in L_n^0$ and write, for every $k\leq \max\{\sup (f),\sup (g)\}$,
 \[
   \{ f\geq k\}=\sqcup_{i\in I_{k}} (\alpha_{i,k}/n,\beta_{i,k}/n),\andSep 
   \{ g\geq k\}=\sqcup_{j\leq J_{k}} (\gamma_{j,k}/n,\delta_{j,k}/n)
 \]
with $I_k, J_k$ finite sets and $\alpha_{i,k}$, $\beta_{i,k}$, $\gamma_{j,k}$, $\delta_{j,k}$ elements of $X_n$.
  
  Assume that $f\ll g$, which implies (see, for instance, \cite[Lemma~4.19]{Vila20}) that $\{ f\geq k\}\Subset\{ g\geq k\}$ for every $k$. Thus, for every fixed $k$ there exists a partition $\sqcup_{j\in J_{k}}B_{j,k}=I_{k}$ such that
  \[
   \sqcup_{i\in B_{j,k}} (\alpha_{i,k}/n,\beta_{i,k}/n) \Subset (\gamma_{j,k}/n,\delta_{j,k}/n)
  \]
 for every $j\in J_{k}$.
 
 This implies that there is an ordering $i_{1},\cdots ,i_{\vert B_{j,k}\vert }$ of $B_{j,k}$ such that
\[
 \gamma_{j,k}\prec \alpha_{i_{1},k}\prec  \beta_{i_{1},k} \leq \alpha_{i_{2},k} \prec\cdots \prec \beta_{i_{\vert B_{j,k}\vert -1},k}\leq \alpha_{i_{\vert B_{j,k}\vert },k}\prec  \beta_{i_{\vert B_{j,k}\vert },k}\prec 
 \delta_{j,k}.
\]
  
  Additionally, since $f\in L_{n}^{0}$, we know that we cannot have $\beta_{i_{s},k} = \alpha_{i_{s+1},k}$ for any $s< \vert B_{j,k}\vert $. Thus, one gets 
  \[
 \gamma_{j,k}\prec \alpha_{i_{1},k}\prec  \beta_{i_{1},k} \prec \alpha_{i_{2},k} \prec\cdots \prec \beta_{i_{\vert B_{j,k}\vert -1},k}\prec \alpha_{i_{\vert B_{j,k}\vert },k}\prec  \beta_{i_{\vert B_{j,k}\vert },k}\prec 
 \delta_{j,k}
\]
  and, consequently,
  \[
   \sum_{i\in B_{j,k}} (\alpha_{i,k},\beta_{i,k}) \prec (\gamma_{j,k},\delta_{j,k}).
  \]
 
 Since $\prec$ is compatible with  addition (see \autoref{lma_prec_properties}), we get
 \[ 
 \sum_{i\in I_{k}}(\alpha_{i,k},\beta_{i,k})
 =
 \sum_{j\in J_{k}}\sum_{i\in B_{j,k}} (\alpha_{i,k},\beta_{i,k}) 
 \prec \sum_{j\in J_k } (\gamma_{j,k},\delta_{j,k})
 \]
 and thus $q_{f}\prec q_{g}$ as desired.
 
 The other implication follows from \autoref{Canonical_element} and the fact that the map $F\colon X_{n}\to L_{n}$ defined in \autoref{Linear_span} is an I-morphism.
 \end{proof}
 
 \begin{cor}
   Let $n\in\NN$ and take $q,t\in X_{n}$ such that $F(q),F(t)\in L_{n}^{0}$. Then, $F(q)\ll F(t)$ if and only if $q\simeq q_{f}\prec q_{t}\simeq t$.
 \end{cor}
%===========================================

\subsubsection{Properties of chainable subsets}

Throughout this subsubsection, we will denote by $H$ an $(n,e)$-chainable subsemigroup of a \CuSgp{} $S$ and $F$ will be its associated I-morphism.

\begin{dfn}\label{rho}
 Given any $m\geq 1$, let $F\colon X_{m}\to L_{m}$ be the I-morphism defined in \autoref{Linear_span}. We will denote by $\rho_{m}$ the $\Cu$-morphism $\rho_{m}\colon\LscI\to S$ such that
\begin{equation*}
 \rho_{m}(\chi_{(\alpha /nm,1]})=F'((\alpha ,\infty ))
\end{equation*}
for every $\alpha\in\Omega_{nm}$. Here, $F'\colon X_{nm}\to S$ denotes the map given by (iii) in \autoref{chainable_subset}.
\end{dfn}

Observe that $\rho_m$ as defined above exists by \autoref{lifting_th}, as the sequence $(F'((\alpha ,\infty)))_{\alpha}$ is bounded by $e$ and is $\ll$-decreasing. This is because $(\alpha ,\infty )\prec (\alpha-1,\infty )$ for every $\alpha$.

Note that $\rho_{m}$ depends on $F'$ (and not only on $m$) but, since this will not be used in the notes, we omit it in order to ease the notation. 

Note that, using (iii) in \autoref{chainable_subset}, one has
\[
 \rho_{m}(\chi_{(\alpha /n,1]})= \rho_{m}(\chi_{(m\alpha /mn,1]})= F'((m\alpha ,\infty ))= F((\alpha ,\infty ))
\]
for every $m>0$ and $\alpha\in\Omega_{n}$.
%===========================================

\begin{lma}\label{TeBelow}
Let $\epsilon >0$ and let $m\geq \epsilon^{-1}$. Then, for any $\alpha ,\beta \leq n$ we have
 \begin{enumerate}[(i)]
 \item $\rho_{m} (\chi_{(\alpha /n+\epsilon ,1]})\ll F((\alpha ,\infty ))\ll \rho_{m} (\chi_{(\alpha /n-\epsilon ,1]}) $, trivially.
  \item $\rho_{m} (\chi_{(\alpha /n+\epsilon ,\beta /n-\epsilon )}) \ll F((\alpha ,\beta ))\ll \rho_{m} (\chi_{(\alpha /n-\epsilon ,\beta /n+\epsilon )})$.
  \item $\rho_{m} (\chi_{[0,\beta /n-\epsilon )})\ll F((-\infty ,\beta ))\ll \rho_{m} (\chi_{[0,\beta /n+\epsilon  )})$
 \end{enumerate}
\end{lma}
\begin{proof}
 We only prove (ii), since (i) is trivial and (iii) follows similarly:
 
 We know from (ii) in \autoref{chainable_subset}, and the observation prior to this lemma applied to $(\beta ,\infty )$, that $F((\alpha ,\beta ))+\rho_{m} (\chi_{(\beta /n,1]})\leq \rho_{m} (\chi_{(\alpha /n,1]})$. Thus, one gets
 \[
 \begin{split}
  F((\alpha ,\beta ))+\rho_{m} (\chi_{(\beta /n,1]})&\leq \rho_{m} (\chi_{(\alpha /n,1]})\ll \rho_{m} (\chi_{(\alpha /n-\epsilon ,1]})\\
  &\leq \rho_{m} (\chi_{(\alpha /n-\epsilon ,\beta /n+\epsilon )}) + \rho_{m} (\chi_{(\beta /n,1]})
 \end{split}
 \]
 for any $\epsilon >0$.
 
By weak cancellation, we have $F((\alpha ,\beta ))\ll \rho_{m} (\chi_{(\alpha /n-\epsilon ,\beta /n+\epsilon )})$.

To see the other inequality, note that, since 
\[
 (m\alpha ,m\beta )+(m\beta -n,\infty )\simeq (m\alpha ,\infty )+(m\beta -n,m\beta ),
\]
one has 
\[
\begin{split}
F((\alpha ,\beta ))+F'((m\beta -n,\infty ))&= F'((m\alpha ,m\beta ))+F'((m\beta -n,\infty ))\\
&=F'((m\alpha ,\infty ))+
F'(m\beta -n,m\beta )\\
&\geq 
F'((m\alpha ,\infty ))=F((\alpha ,\infty )).
\end{split}
\]

Using the fact that $\rho_m$ is a $\Cu$-morphism at the first step and the above inequality at the third step, we get
\[
\begin{split}
 \rho_{m} (\chi_{(\alpha /n+1/m,\beta /n-1/m )}) +\rho_{m} (\chi_{(\beta /n-1/m ,1]}) &\ll \rho_{m} (\chi_{(\alpha /n,1]})
 =F((\alpha ,\infty ))\\
 &\leq F((\alpha ,\beta ))+\rho_{m} (\chi_{(\beta /n -1/m,1]}).
 \end{split}
\]

Applying weak cancellation, one gets 
\[
F((\alpha ,\beta ))\gg \rho_{m} (\chi_{(\alpha /n+1/m,\beta /n-1/m )})\geq \rho_{m} (\chi_{(\alpha /n+\epsilon ,\beta /n-\epsilon )}),
\]
since $\epsilon >1/m$.
\end{proof}

Let $f\in X_{n}$ and write $f$ as  $\sum_{i=1}^{k}(\alpha_{i},\beta_{i})$. Also, consider $g=\sum_{i=1}^{k}\chi_{U_{i}}\in\LscI$ with $U_{i}=(\alpha_{i}/n,\beta_{i}/n)$. Then, the previous lemma shows that for every $\epsilon >0$ there exists a sufficiently large $m$ such that $\rho_{m}(R_{\epsilon }(g))\ll F(f)\ll \rho_{m}(N_{\epsilon }(g))$, where $R_{\epsilon }(g)$ and $N_{\epsilon }(g)$ denote the functions $\sum_{i=1}^{k}\chi_{R_{\epsilon }(U_{i})}$ and $\sum_{i=1}^{k}\chi_{N_{\epsilon } (U_{i})}$ respectively. Recall that $R_{\epsilon }(a,b)=(a+\epsilon ,b-\epsilon )$ and $N_{\epsilon }(a,b)=(a-\epsilon ,b+\epsilon )$.
%===========================================

\begin{lma}\label{Prec_Convert}
 Let $S$ be a \CuSgp{} and $H$ an $(n,e)$-chainable subset of $S$ with associated I-morphism $F\colon X_{n}\to H$. Then, given any finite family $\{q_{i}\prec t_{i}\}_{i=1,\cdots , N}$ in $X_{n}$, there exist functions $f_{i},g_{i}\in \LscI$ for $i\leq N$ such that
 \begin{enumerate}[(i)]
  \item $f_{i}\ll g_{i}$ for all $i$.
  \item $\rho_{m}(f_{i})\ll F(q_{i})\ll 
  \rho_{m}(g_{i})\ll F(t_{i})$ for all $i$ and for every sufficiently large $m$.
  \item $f_{i}=f_{j}$ (resp. $f_{i}=g_{j}$,  $g_{i}=g_{j}$) whenever $q_{i}=q_{j}$ (resp. $q_{i}=t_{j}$, $t_{i}=t_{j}$) for $i,j\leq N$.
 \end{enumerate}
 
 In fact, given the I-morphism $G\colon X_{n}\to\LscI$ defined in \autoref{Linear_span}, one can take $f_{i}=R_{\epsilon }(G(q_{i}))$ and $g_{i}=R_{\epsilon }(G(t_{i}))$ for $\epsilon <1/4n$.
\end{lma}
\begin{proof}
 For every $i$, write $q_{i}=\sum_{j\leq M,k\leq K}(\alpha_{j,k}^{i},\beta_{j,k}^{i})$ and $t_{i}=\sum_{j\leq M}(\alpha_{j}^{i},\beta_{j}^{i})$ such that
 \[
  \alpha_{j}^{i}\prec \alpha_{j,1}^{i}\prec \beta_{j,1}^{i}\prec\cdots\prec \beta_{j,K}^{i}\prec \beta_{j}^{i}
 \]
for every $j$, where it is understood that $K$ depends on $i$ and $j$, and $M$ depends on $i$ (see the definitions before \autoref{Max_Elements}).

Define the open subsets $U_{j,k}^{i}=(\alpha_{j,k}^{i}/n,\beta_{j,k}^{i}/n)$ and $U_{j}^{i}=(\alpha_{j}^{i}/n,\beta_{j}^{i}/n)$. Let $G\colon X_n\to\LscI$ be the I-morphism obtained in \autoref{Linear_span}, which is defined as $G(\alpha ,\beta )=\chi_{(\alpha /n ,\beta /n )}$.

Since 
\[
\sum_{k}(\alpha_{j,k}^{i},\beta_{j,k}^{i})\prec (\alpha_{j}^{i},\beta_{j}^{i}),
\]
we can apply $G$ to obtain
\[
 \sum_{k}\chi_{U_{j,k}^{i}}\ll \chi_{U_{j}^{i}}
\]
for every $i,j$.

Set $\epsilon<1/4n$, and consider $V_{j}^{i}=R_{\epsilon}(U_{j}^{i})$, the $\epsilon$-retraction of $U_{j}^{i}$. Note that, for every $\alpha\prec \alpha '\prec \beta '\prec \beta$ in $\Omega_{n}$, we have $(\alpha '/n-\epsilon ,\beta '/n+\epsilon )\Subset (\alpha /n+\epsilon , \beta /n-\epsilon )$ in $[0,1]$. This implies 
\[
 \sum_{k}\chi_{U_{j,k}^{i}}\ll \sum_{k}\chi_{N_{\epsilon}(U_{j,k}^{i})}\ll \chi_{V_{j}^{i}}
\]
for every $i,j$.

Now set $V_{j,k}^{i}=R_{\epsilon}(U_{j,k}^{i})$ and let $m\geq 4n$. Then, using \autoref{TeBelow} at the first, second, and last steps, and the fact that $\rho_{m}$ is a $\Cu$-morphism at the third, we get
\[
 \sum_{k}\rho_{m}(\chi_{V_{j,k}^{i}})\ll 
 \sum_{k}F((\alpha_{j,k}^{i},\beta_{j,k}^{i}))\ll
 \sum_{k}\rho_{m}(\chi_{N_{\epsilon}(U_{j,k}^{i})})\ll
 \rho_{m}(\chi_{V_{j}^{i}})\ll F((\alpha_{j}^{i},\beta_{j}^{i})).
\]

Define $f_{i}=\sum_{j\leq M,k\leq K}\chi_{V_{j,k}^{i}}$ and $g_{i}=\sum_{j\leq M}\chi_{V_{j}^{i}}$. By adding on $j$ in the previous inequality, one gets
\[
 \rho_{m}(f_{i})\ll F(q_{i})\ll \rho_{m}(g_{i})\ll F(t_{i}),
\]
which is condition (ii).

By the comments above, we also have $\sum_{k}\chi_{V_{j,k}^{i}}\ll \sum_{k}\chi_{U_{j,k}^{i}}\ll \chi_{V_{j}^{i}}$, which implies condition (i) in the same fashion.

Condition (iii) follows by construction, since we have applied the same retraction to all the elements.
\end{proof}
%===========================================

\begin{rmk}\label{Retraction_eq}
 Given $\sum_{i} (\alpha_{i},\beta_{i})\simeq \sum_{j} (\gamma_{j},\delta_{j})$ in $X_{n}$, we have 
 \[
 \sum_{i} \chi_{(\alpha_{i}/n+\epsilon ,\beta_{i}/n-\epsilon )}=\sum_{j} \chi_{(\gamma_{j}/n+\epsilon ,\delta_{j}/n-\epsilon )}
 \]
 in $\LscI$ for every positive $\epsilon <1/2n$.
 
 Indeed, $\sum_{i} (\alpha_{i}\beta_{i})\simeq \sum_{j} (\gamma_{j},\delta_{j})$ implies $\sum_{i} \chi_{(\alpha_{i}/n ,\beta_{i}/n )}=\sum_{j} \chi_{(\gamma_{j}/n ,\delta_{j}/n )}$ in $\LscI$ by \autoref{Linear_span}.

 Since $R_{\epsilon }(f)$ is well defined for any $f\in L_{n}$, we get 
 \[
  \sum_{i} \chi_{(\alpha_{i}/n+\epsilon ,\beta_{i}/n-\epsilon )}=
  R_{\epsilon }\left(\sum_{i} \chi_{(\alpha_{i}/n ,\beta_{i}/n )}\right)
  =
   R_{\epsilon }\left(\sum_{j} (\gamma_{j},\delta_{j})\right) =
   \sum_{j} \chi_{(\gamma_{j}/n+\epsilon ,\delta_{j}/n-\epsilon )}.
 \]
\end{rmk}
%===========================================
\Msubsection{Properties $\rm{I}_{0}$ and $\rm{I}$}{Properties I0 and I}

In this last section we introduce property $\rm{I}_{0}$, which provides a characterization for when a \CuSgp{} is $\Cu$-isomorphic to an inductive limit of the form $\lim_n\LscI$. We also introduce  property $\rm{I}$, which leads to a characterization of when a \CuSgp{} is $\Cu$-isomorphic to the Cuntz semigroup of an AI-algebra.
%===========================================

\Msubsubsection{Property $\rm{I}_{0}$}{Property I0}
\begin{dfn}\label{dfn_prop_I0}
 We say that a \CuSgp{} $S$ has the \emph{reduced  $\rm{I}_{0}$} property if, given any sequence $0=z_{l}\ll z_{l-1}\ll z_{l-2}\ll\cdots \ll z_{0}\ll z_{-\infty}= p\ll p$ and $r\ll t$ such that $r=\sum_{i\in I}z_{i}$ and $t=\sum_{j\in J}z_{j}$ with $I,J$ multisets of $\{-\infty ,0,\cdots ,l\}$,  there exists $k\in\NN$ and an $(n,e)$-chainable subsemigroup  $H\subset S$ with associated I-morphism $F$ and $p=ke$ such that:
 \begin{enumerate}[(i)]
  \item There exists a sequence $a_{l-1}\prec\cdots\prec a_{0}\prec k(-\infty ,\infty )$ in $X_{n}$ such that $z_{i}\ll F(a_{i-1})\ll z_{i-1}$ for every $i$.
  \item There exist $a, b\in X_{n}$ such that $\sum_{i\in I}a_{i}\simeq a\prec b\simeq \sum_{j\in J}a_{j}$.
 \end{enumerate}
\end{dfn}
%===========================================

Recall that given any $l,M\in\mathbb{N}$, the subset $C_l^M$ of $\LscI^{M}$ is defined as  $C_{l}^{M}=\{ \chi^{s}_{(i/l,1]} \}_{i,s}\cup\{ 1_{s}\}$, where $1_{s}$ and $\chi^{s}_{(i/l,1]}$ denote the unit and the element $\chi_{(i/l,1]}$ in the $s$-th summand respectively. Also recall that we denote by $B_{l}^{M}$  the additive span of the elements in $C_{l}^{M}$.

\begin{prp}\label{Easy_case}
 Let $\varphi\colon \LscI\to S$ be a Cu-morphism, and assume that $S$ has the reduced  $\rm{I}_{0}$ property. Then, for every $l\in\mathbb{N}$ and $x,x',y\in B_{l}^{1}$ with $x\ll x'$ and $\varphi (x')\ll \varphi (y)$, there exist $\Cu$-morphisms $\theta\colon \LscI\to \LscI$ and $\phi \colon \LscI\to S$ satisfying the conditions in \autoref{MainTh}. That is to say, we have
 \begin{enumerate}[(i)]
  \item $\varphi (\chi_{((i+1)/l,1]})\ll 
  \phi\theta (\chi_{(i/l,1]})$ and $\phi\theta (\chi_{((i+1)/l,1]})\ll 
  \varphi (\chi_{(i/l,1]})$ for every $i<l$.
  \item $\theta (x)\ll \theta (y)$.
  \item $\varphi (1)=\phi\theta (1)$.
 \end{enumerate}
\end{prp}
\begin{proof}
 Given a Cu-morphism $\varphi\colon \LscI\to S$ and a fixed $l\in\mathbb{N}$, consider the elements $x_{i}=\varphi (\chi_{(i/l,1]})$ and set $x_{-\infty }=\varphi (1)$. Also,
 set $r=\varphi (x')$ and $t=\varphi (y)$.
 
 Let $I,J$ be the multisets of $\{-\infty ,0,\cdots ,l\}$ such that $x'=\sum_{i\in I}\chi_{(i/l,1]}$ and $y=\sum_{j\in J}\chi_{(j/l,1]}$, where it is understood that $\chi_{(-\infty /l,1]}=1$. Thus, we can write $r,t$ as $r=\sum_{i\in I}x_{i}$ and $t=\sum_{j\in J}x_{j}$.
 
 Let $z_{i}=\varphi (\chi_{(i/2l,1]})$ for every $i\leq 2l$, and let $p=x_{-\infty }=\varphi (1)$, so that we have
 \[
  0=x_{l}=z_{2l}\ll z_{2l-1}\ll\cdots\ll z_{0}=x_{0}\ll x_{-\infty }=p.
 \]
 
Since $z_{2i}=x_{i}$ for every $i\leq l$, we also get $r=\sum_{i\in I}z_{2i}$ and $t=\sum_{j\in J}z_{2j}$.
 
 By the reduced  $\rm{I}_{0}$ property applied to $\{ z_{i}\}_{0\leq i\leq 2l}$ and $r\ll t$, there exist $k\in\NN$, a compact element $e$ such that $p=ke$, an $(n,e)$-chainable subsemigroup $H$ and an associated I-morphism $F$ with a sequence $a_{2l}\prec\cdots\prec a_{0}$ in $X_{n}$ such that 
 \[
 z_{i}\ll F(a_{i-1})\ll z_{i-1}.
 \]
 
 By condition (ii) in \autoref{dfn_prop_I0}, there also exist $a, b\in X_{n}$ with $\sum_{i\in I}a_{2i}\simeq a\prec b\simeq \sum_{j\in J}a_{2j}$. 
 
 Note, in particular, that
 \[
 x_{i}=z_{2i}\ll F(a_{2i-1})\ll F(a_{2(i-1)})\ll z_{2(i-1)}=x_{i-1}
 \]
 for every $i$.
 
 Consider the finite family of $\prec$-relations
 \[
 \{
 a_{i}\prec a_{i-1}
 \}_{i} \cup
 \{ a_{0}\prec k(-\infty ,\infty ) \}\cup
 \{
 a\prec b
 \}.
 \]
 
 By \autoref{Prec_Convert}, if $G\colon X_n\to\LscI$ is the I-morphism defined in \autoref{Linear_span}, there exists a large enough $m\in\NN$ such that the functions  $f_i=R_{\epsilon }(G(a_i))$, $f=R_{\epsilon }(G(a))$ and $g=R_{\epsilon }(G(b))$ satisfy
 \[
 \begin{split}
 f_{i}\ll f_{i-1}\ll k1,\\
  \rho_{m}(f_{i})\ll F(a_{i})\ll \rho_{m}(f_{i-1})\ll F(a_{i-1}),\\
  f\ll g,\\
  \rho_{m}(f )\ll F(a)\ll  \rho_{m}(g)\ll F(b)
 \end{split}
 \]
for every $i$, where recall that $\rho_m$ is the $\Cu$-morphism defined in \autoref{rho}.

Furthermore, since $\sum_{i\in I}a_{2i}\simeq a$, we get $\sum G(a_{2i} )=G(a)$. By the argument in \autoref{Retraction_eq} we get $\sum R_{\epsilon }(G(a_{2i} ))=R_{\epsilon }(G(a))$, that is, $\sum_{i\in I} f_{2i}=f$. Similarly, we also have $\sum_{j\in J} f_{2j}=g$.

Thus, for every $i\leq n$, one gets
\[
x_{i}\ll F(a_{2i-1})\ll \rho_{m}(f_{2(i-1)})\ll F(a_{2(i-1)})\ll x_{i-1}.
\]

Define the $\Cu$-morphism $\theta\colon\LscI\to \LscI$ as $\theta (\chi_{(i/l,1]})= f_{2i}$ and $\theta (1)=k1$. This can be done by \autoref{lifting_th} because $\{ f_{2i}\}_{i}$ is a decreasing sequence bounded by $k1$.

Note that we have
\[
\theta (x)\ll \theta (x')=\sum_{i\in I}f_{2i}=f\ll g=\sum_{j\in J}f_{2j}=\theta (y),
\]
and thus condition (ii) of the proposition is satisfied.

Set $\phi =\rho_{m}$, so that we get $\phi\theta (1)=k\rho_{m}(1)=ke=p=\varphi (1)$. By construction, one also has
\[
\begin{split}
 \varphi (\chi_{(i/l,1]})&=x_{i}\ll F(a_{2i-1})\ll \rho_{m}(f_{2(i-1)})\\
 &=\phi (\theta (\chi_{((i-1)/l,1]}))\ll F(a_{2(i-1)})\ll x_{i-1}=\varphi (\chi_{((i-1)/l,1]}),
\end{split}
\]
which implies condition (i) in the proposition.
\end{proof}
%===========================================

We now strengthen the previous property.
\begin{dfn}\label{dfn_I0}
We say that $S$ satisfies \emph{property $\rm{I}_{0}$} if, given any finite number of decreasing sequences of the form $0=z_{l,s}\ll z_{l-1,s}\ll\cdots \ll z_{0,s}\ll z_{-\infty ,s}=p_{s}\ll p_{s}$, $s\leq M$, and any pair $r\ll t$ such that $r=\sum_{s\leq M}\sum_{i\in I}z_{i,s}$ and $t=\sum_{s\leq M}\sum_{j\in J}z_{j,s}$ with $I,J$  (possibly empty) multisets of $\{-\infty ,0,\cdots ,l\}$, there exists an $(n,e)$-chainable subsemigroup $H$ with associated I-morphism $F$:
 \begin{enumerate}[(i)]
 \item For each $s$ there exists $k_{s}\in\NN$ such that $p_s =k_s e$.
  \item There exist sequences $a_{l-1,s}\prec\cdots\prec a_{0,s}\prec k_{s}(-\infty ,\infty )$ in $X_{n}$ such that $z_{i,s}\ll F(a_{i-1,s})\ll z_{i-1,s}$ for every $i,s$.
  
  \item There exist $a, b\in X_{n}$ such that $\sum_{s\leq M, i\in I}a_{i,s}\simeq a\prec b\simeq \sum_{s\leq M,j\in J}a_{j,s}$.
 \end{enumerate}
\end{dfn}
%===========================================

\begin{lma}\label{I0_InductiveLimits}
 Property $\rm{I}_{0}$ goes through inductive limits.
 \end{lma}
 \begin{proof}
 Let $S=\lim_{m} S_{m}$ be an inductive limit of Cu-semigroups $S_{m}$ satisfying property $\rm{I}_{0}$, and let $z_{i,s},p_{s},r,t$ be as in \autoref{dfn_I0}. Given $f\in S_{m}$, we will denote by $[f]$ the image of $f$ through the limit morphism from $S_{m}$ to $S$. 
 
 Then, there exists a large enough $m$ and elements $g_{i,s},f_{i,s}, e_{s}\in S_{m}$ such that:
 \[
  \begin{split}
  g_{i,s}\ll f_{i,s}\ll g_{i-1,s}\ll f_{i-1,s}\ll e_{s}\ll e_{s},\\
   z_{i,s}\ll [g_{i-1,s}]\ll [f_{i-1,s}]\ll z_{i-1,s}\,\, ,\, p_{s}=[e_{s}],\\
   \sum_s\sum_{i\in I}f_{i,s}\ll \sum_s\sum_{j\in J}f_{j,s},\\
   \sum_s\sum_{i\in I}[f_{i,s}]\ll r\ll
   \sum_s\sum_{j\in J}[f_{j,s}]
   \ll t.
  \end{split}
 \]

 Since $S_{m}$ satisfies property $\rm{I}_{0}$, we can find an $(n,e)$-chainable subset $H$ with $e_{s}=k_{s}e$ such that:
 \begin{enumerate}[(i)]
  \item There exist sequences $a_{l-1,s}\prec b_{l-1,s}\prec a_{l-2,s}\prec\cdots\prec a_{0,s}\prec k_{s}(-\infty ,\infty )$ in $X_{n}$ such that 
  \[
  g_{i,s}\ll F(a_{i,s})\ll f_{i,s}\ll F(b_{i,s})\ll g_{i-1,s}
  \] 
  in $S_m$ for every $i,s$.
  
  \item There exist $a, b\in X_{n}$ such that $\sum_{s\leq M, i\in I}a_{i,s}\simeq a\prec b\simeq \sum_{s\leq M,j\in J}a_{j,s}$.
 \end{enumerate}
 
 Therefore, one gets $f_{i,s}\ll [F(a_{i-1,s})]\ll f_{i-1,s}$ for every $i,s$ and $p_{s}=k_{s}[e]$ for every $s$.
 
 By \autoref{morphisms}, $[H]$ is an $(n,[e])$-chainable subset of $S$, and it satisfies the required properties by the previous considerations.
\end{proof}
%===========================================

\begin{exa}\label{N_has_I0}
 Let $S$ be an inductive limit of the form $S=\lim_{m}S_{m}$ , with $S_{m}=\NN$ for every $m$. Then $S$ satisfies property $\rm{I}_{0}$.
 
 To prove this, it is enough to show by \autoref{I0_InductiveLimits} that $\NN$ satisfies property $\rm{I}_{0}$.
 
 Thus, let $z_{i,s},p_{s},r,t$ be as in \autoref{dfn_I0}, and note that there exist positive integers $k_{i,s}\leq k_{i-1,s}$ such that $z_{i,s}=k_{i,s}1$ for every $i,s$. There also exist integers $k_{s}$ such that $p_{s}=k_{s}1$.
 
 Set $H=\mathbb{N}$, which is the additive span of $1$. Then, it follows from \autoref{algebraic} that $H$ is $(n,1)$-chainable for every $n$. Take $n=1$ and let $F\colon X_{1}\to H$ be as in \autoref{algebraic}, where note that condition (i) in \autoref{dfn_I0} is satisfied by construction.
 
 Consider the elements $a_{i,s}=k_{i-1,s}(-\infty ,\infty )$, whose image through $F$ is $k_{i-1,s}1$. This implies condition (ii) in \autoref{dfn_I0}.
 
 Also, since $\sum_{s\leq M, i\in I} k_{i,s}\leq \sum_{s\leq M, j\in J} k_{j,s}$, we clearly have 
 \[
 \sum_{s\leq M, i\in I} k_{i,s}(-\infty ,\infty )\prec  \sum_{s\leq M, j\in J} k_{j,s}(-\infty ,\infty ).
  \]

  Letting $a=\sum_{s\leq M, i\in I} k_{i,s}(-\infty ,\infty )$ and $b= \sum_{s\leq M, j\in J} k_{j,s}(-\infty ,\infty )$, condition (iii) also follows.
\end{exa}

\begin{exa}\label{Exa_Property_I0}
 Every \CuSgp{} $S$ of the form $S=\lim_{m}S_{m}$ , with $S_{m}=\LscI$ for every $m$, satisfies property $\rm{I}_{0}$.
 
 As in \autoref{N_has_I0}, by \autoref{I0_InductiveLimits} we only need to prove that $\LscI$ satisfies  property $\rm{I}_{0}$.
 
 Thus, take sequences 
 \[
 0=z_{l,s}\ll z_{l-1,s}\ll\cdots \ll z_{0,s}\ll z_{-\infty ,s}=p_{s}\ll p_{s},\, s\leq M
 \]
 and a pair $r\ll t$ with $r=\sum_{s\leq M, i\in I}z_{i,s}$ and $t=\sum_{s\leq M, j\in J}z_{j,s}$ in $\LscI$.
 
 Let $L_m$ be the subsets of $\LscI$ defined above \autoref{Notation}. Then, since $\cup_{m}L_{m}$ is dense in $\LscI$, there exist decreasing sequences $f_{l-1,s}\ll \cdots\ll f_{0,s}$ in $L_{m}$ such that
 \[
 z_{i,s}\ll f_{i-1,s}\ll z_{i-1,s}
 \]
 and $\sum_{s\leq M, i\in I}f_{i,s}\ll \sum_{s\leq M, j\in J}f_{j,s}$. Also take $k_{s}\in\mathbb{N}$ such that $k_{s}1=p_{s}$.
 
 Take a positive rational $\epsilon<1/2n$ small enough so that 
 \[
  z_{i,s}\ll R_{\epsilon} (f_{i-1,s})\ll z_{i-1,s}
 \]
and 
\[
 \sum_{s\leq M, i\in I}R_{\epsilon}(f_{i,s})=R_{\epsilon}\left(\sum_{s\leq M, i\in I} f_{i,s}\right)\ll R_{\epsilon}\left(\sum_{s\leq M, j\in J} f_{j,s}\right)
 =\sum_{s\leq M, j\in J} R_{\epsilon}(f_{j,s}),
\]
where the equalities follow from \autoref{Rmk_Retraction}.

Define $g_{i,s}=R_{\epsilon }(f_{i,s})$ and note that, by \autoref{Red_L0}, there exists some $n\in\mathbb{N}$ such that
\[
 g_{i,s}, \sum_{s\leq M, i\in I}g_{i,s}, \sum_{s\leq M, j\in J} g_{j,s}\in L_{n}^{0}.
\]
Recall from \autoref{Linear_span} that $L_n$ is $(n,1)$-chainable, and let $F\colon X_n\to L_n$ be the I-morphism obtained in its proof, that is, the I-morphism induced by $F((\alpha ,\beta ))=\chi_{(\alpha /n ,\beta /n)}$.

Set $G_{I}=\sum_{s\leq M, i\in I}g_{i,s}$ and $G_{J}=\sum_{s\leq M, j\in J} g_{j,s}$, so that $G_I\ll G_J$. By \autoref{L0_comp}, the elements $q_{g_{i,s}}$, $q_{G_{I}}$, $q_{ G_{J}}\in X_{n}$ satisfy
\[
 q_{g_{i,s}}\prec q_{g_{i-1,s}},
 \andSep
 q_{G_{I}}\prec q_{ G_{J}}.
\]

Set $a_{i,s}=q_{g_{i,s}}$, $a=q_{G_{I}}$ and $b=q_{ G_{J}}$. By \autoref{Canonical_element}, $F(a_{i,s})=g_{i,s}$, $F(a)=G_I$ and $F(b)=G_J$.

Using this, we now have 
$F(a)=\sum_{s\leq M, i\in I} F(a_{i,s})$, $F(b)=\sum_{s\leq M, j\in J} F(a_{j,s})$ and $z_{i,s}\ll g_{i-1,s}=F(a_{i-1,s})\ll z_{i-1,s}$ for every $i,s$. Since $F(\sum_{s,i}a_{i,s})=G_I$ and $F(\sum_{s,j}a_{j,s})=G_J$, it follows from \autoref{Canonical_element} that $\sum_{s, i} a_{i,s}\simeq a$ and, by \autoref{L0_comp}, that $\sum_{s, j} a_{j,s}\prec b$. Therefore, we have
\[
 \sum_{s\leq M, i\in I} a_{i,s}\simeq a\prec
 b\simeq\sum_{s\leq M, j\in J} a_{j,s},
\]
which finishes the proof.
\end{exa}
%===========================================

\begin{exa}\label{exa_propI0_sums}
 The \CuSgp{} $\LscI\oplus\LscI$ does not satisfy property $\rm{I}_{0}$. To see this, simply consider the sequences $(0,0)\ll (1,0)$ and $(0,0)\ll (0,1)$ and set $r=t=(0,0)$. 
 
 If $\LscI\oplus\LscI$ were to satisfy property $\rm{I}_{0}$, in particular it would have a compact element $e$ such that $(1,0)=k_1 e$ and $(0,1)= k_2 e$ for some $k_1,k_2\in\NN$. Clearly, this is not possible.
 
 Note that the same argument works for $\overline{\NN}\oplus \overline{\NN}$.
\end{exa}
%===========================================

\begin{thm}
 Let $S$ be a \CuSgp{}. Then, $S$ is $\Cu$-isomorphic to an inductive limit of the form $\lim \LscI$ if and only if $S$
  is countably based, compactly bounded (see \autoref{dfn_comp_bound}) and it satisfies (O5), (O6), weak cancellation, and  property $\rm{I}_{0}$.
\end{thm}
\begin{proof}
 The proof of one implication is analogous to the proof of \autoref{Easy_case}, so we omit it.
 
 The other implication is the previous example, together with the well known fact that (O5), (O6), weak cancellation and being countably based go through inductive limits (see, for example, \cite[Chapter 4]{AntoPereThie18}). Clearly, being compactly bounded also goes through inductive limits, which finishes the proof.
 \end{proof}
%===========================================

\Msubsubsection{Property $\rm{I}$}{Property I}
 
 \begin{dfn}\label{PropertyI}
  We say that a $\Cu$-semigroup $S$ satisfies \emph{property $\rm{I}$} if, given any finite number of decreasing sequences of the form $0=z_{l,s}\ll z_{l-1,s}\ll\cdots \ll z_{0,s}\ll z_{-\infty ,s}=p_{s}\ll p_{s}$, $1\leq s\leq M$, and any pair $r\ll t$ with $r=\sum_{s\leq M, i\in I}z_{i,s}$ and $t=\sum_{s\leq M, j\in J}z_{j,s}$, there exist finitely many subsemigroups $H_k\subset S$ such that each $H_k$ is $(n_{k},e_{k})$-chainable with associated I-morphism $F_k$ and:
 \begin{enumerate}[(i)]
  \item Each compact $p_{s}$ can be written as a  linear combination of $\{ e_{k}\}_{k}$, $p_{s}=\sum_{k}m_{k,s}e_{k}$.
  \item There exist sequences $a_{l-1,s}^{(k)}\prec\cdots\prec a_{0,s}^{(k)}\prec m_{k,s}(-\infty ,\infty )$ in $X_{n_{k}}$ such that
  \[
   z_{i,s}\ll \sum_{k}F_{k}(a_{i-1,s}^{(k)})\ll z_{i-1,s}
  \]
  for every $i,s$.
  \item For each $k$, there exist $a_{k}, b_{k}\in X_{n_k}$ such that 
  \[
  \sum_{s\leq M, i\in I}a_{i,s}^{(k)}\simeq a_{k}\prec b_{k}\simeq \sum_{s\leq M, j\in J}a_{j,s}^{(k)}.
  \]
 \end{enumerate}
 \end{dfn}
 
 \begin{rmk}
  If a \CuSgp{} satisfies property $\rm{I}_{0}$, then it clearly also satisfies property $\rm{I}$.
 \end{rmk}

 \begin{lma}\label{lma_I_limits}
  Property $\rm{I}$ goes through inductive limits.
 \end{lma}
 \begin{proof}
  We follow the proof of \autoref{I0_InductiveLimits}:
  \vspace{0.2cm}
 
  Let $S=\lim S_m$ with $S_m$ $\Cu$-semigroups satisfying property $\rm{I}$.
  
  As in \autoref{PropertyI}, consider decreasing sequences of the form 
  \[
   0=z_{l,s}\ll z_{l-1,s}\ll\cdots \ll z_{0,s}\ll z_{-\infty ,s}=p_{s}\ll p_{s}, \,1\leq s\leq M,
  \]
 and a pair of elements $r\ll t$ with $r=\sum_{s\leq M, i\in I}z_{i,s}$ and $t=\sum_{s\leq M, j\in J}z_{j,s}$.
 
 Then, for a large enough $m$, there exist elements $f_{i,s},g_{i,s},q_{s}\in S_m$ such that
 \[
  \begin{split}
   g_{i,s}\ll f_{i,s}\ll g_{i-1,s}\ll f_{i-1,s}\ll q_{s}\ll q_{s},\\
   z_{i,s}\ll [g_{i-1,s}]\ll [f_{i-1,s}]\ll z_{i-1,s}\,\, ,\, p_{s}=[q_{s}],\\
   \sum_s\sum_{i\in I}f_{i,s}\ll \sum_s\sum_{j\in J}f_{j,s},\\
   \sum_s\sum_{i\in I}[f_{i,s}]\ll r\ll
   \sum_s\sum_{j\in J}[f_{j,s}]
   \ll t,
  \end{split}
 \]
 where as in \autoref{I0_InductiveLimits} $[f]$ denotes the image of $f$ through the limit morphism from $S_{m}$ to $S$.

 Since $S_{m}$ satisfies property $\rm{I}$, we can find a $(n_k,e_k)$-chainable subsets $H_k$ such that:
 \begin{enumerate}[(i)]
 \item Each compact $q_{s}$ can be written as a  linear combination of $\{ e_{k}\}_{k}$, $q_{s}=\sum_{k}m_{k,s}e_{k}$.
  \item There exist sequences $a_{l-1,s}^{(k)}\prec b_{l-1,s}^{(k)}\prec a_{l-2,s}^{(k)}\prec\cdots\prec a_{0,s}^{(k)}\prec m_{k,s}(-\infty ,\infty )$ in $X_{n_k}$ such that 
  \[
  g_{i,s}\ll \sum_{k} F_k (a_{i,s}^{(k)})\ll f_{i,s}\ll \sum_{k}F_k (b_{i,s}^{(k)})\ll g_{i-1,s}
  \] 
  in $S_m$ for every $i,s$.
  
  \item There exist $a_k, b_k\in X_{n_k}$ such that $\sum_{s\leq M, i\in I}a_{i,s}^{(k)}\simeq a\prec b\simeq \sum_{s\leq M,j\in J}a_{j,s}^{(k)}$.
 \end{enumerate}
 
 It is now easy to check that the semigroups $[H_k]$, the elements $a_{i,s}^{(k)}$ and the compacts $[e_k]$ satisfy the desired properties for our original pair of elements and sequences.
 
 Recall that $[H_k]$ is an $(n_k,[e_k])$-chainable subsemigroup of $S$ by \autoref{morphisms}.
 \end{proof}

 The following lemma exemplifies one of the differences between property $\rm{I}_0$ and property $\rm{I}$, since, from \autoref{exa_propI0_sums}, we know that property $\rm{I}_0$ is not preserved under direct sums.
 
 \begin{lma}\label{lma_I_sums}
  If $S,T$ are \CuSgp{s} satisfying property $\rm{I}$, their direct sum $S\oplus T$ also satisfies property $\rm{I}$.
 \end{lma}
 \begin{proof}
  Given a finite number of decreasing sequences in $S\oplus T$ of the form 
  \[
   0=z_{l,s}\ll z_{l-1,s}\ll\cdots \ll z_{0,s}\ll z_{-\infty ,s}=p_{s}\ll p_{s},\, 1\leq s\leq M,
  \]
and a pair $r\ll t$ with $r=\sum_{s\leq M, i\in I}z_{i,s}$ and $t=\sum_{s\leq M, j\in J}z_{j,s}$, write $z_{i,s}=(v_{i,s},w_{i,s})$ with $v_{i,s}\in S$ and $w_{i,s}\in T$. Also set $q_{s}\in S_{c}$ and $t_{s}\in T_{c}$ such that $p_{s}=(q_s,t_s)$.

Thus, one gets the following decreasing sequences
\[
\begin{split}
 0=v_{l,s}\ll v_{l-1,s}\ll\cdots \ll v_{0,s}\ll v_{-\infty ,s}=q_{s}\ll q_{s},\, 1\leq s\leq M,\\
 0=w_{l,s}\ll w_{l-1,s}\ll\cdots \ll w_{0,s}\ll w_{-\infty ,s}=t_{s}\ll t_{s},\, 1\leq s\leq M,
 \end{split}
\]
in $S$ and $T$ respectively.

Using that $S$ satisfies property $I$, we obtain $(n_k,e_k)$-chainable subsemigroups $H_k$ of $S$ and elements $a_{i,s}^{(k)}\in X_{n_k}$ satisfying the conditions in \autoref{PropertyI} for the above sequences in $S$ and the pair $\sum_{s\leq M, i\in I}v_{i,s}$, $\sum_{s\leq M, j\in J}v_{j,s}$.

Similarly, we find $(m_l,f_l)$-chainable subsemigroups $U_l$ of $T$ and elements $b_{i,s}^{(l)}\in X_{m_l}$ satisfying the conditions in \autoref{PropertyI} for the sequences in $T$ and the pair $\sum_{s\leq M, i\in I}w_{i,s}$, $\sum_{s\leq M, j\in J}w_{j,s}$.

It is easy to check that $H_k\oplus 0$ and $0\oplus T_l$ are $(n_k,(e_k,0))$ and $(m_l,(0,f_l))$ chainable subsemigroups of $S\oplus T $ respectively. Moreover, it is also clear that these subsemigroups with the elements $a_{i,s}^{(k)}\in X_{n_k}$ and $b_{i,s}^{(l)}\in X_{m_l}$ satisfy the conditions of property $\rm{I}$ for our original sequences and pairs of elements.
 \end{proof}
%===========================================
 
 \begin{exa}\label{exa_prop_I}
 $ $
 \begin{enumerate}
  \item Let $S$ be the Cuntz semigroup of an AF-algebra. Then $S$ satisfies  property $\rm{I}$.
  \item The Cuntz semigroup of an AI-algebra satisfies property $\rm{I}$.
 \end{enumerate}
 
 Indeed, by combining \autoref{lma_I_limits} and \autoref{lma_I_sums} it suffices to show that $\overline{\NN}$ and $\LscI$ satisfy property $\rm{I}$. This follows from Examples \ref{N_has_I0} and \ref{Exa_Property_I0}, as both semigroups satisfy property $\rm{I}_{0}$.
 \end{exa}
%===========================================

\begin{thm}\label{thm:IiffAI}
  Let $S$ be a \CuSgp{}. Then, $S$ is $\Cu$-isomorphic to the Cuntz semigroup of an AI-algebra if and only if $S$ is countably based, compactly bounded and it satisfies (O5), (O6), weak cancellation, and property $\rm{I}$.
\end{thm}
\begin{proof}
  If $S$ is Cu-isomorphic to the Cuntz semigroup of an AI-algebra, then it is well known that $S$ is countably based, compactly bounded and satisfies (O5), (O6), and weak cancellation.
  
  It follows from \autoref{exa_prop_I} that $S$ satisfies property $\rm{I}$.
  
  \vspace{0.2cm}
  Conversely, assume that $S$ is countably based, compactly bounded and satisfies (O5), (O6), weak cancellation, and property $\rm{I}$. We will use \autoref{MainTh} to show that $S$ is the Cuntz semigroup of an AI-algebra.
  
  More explicitly, given a morphism $\varphi\colon \LscI^{M}\to S$, an integer $l\in\mathbb{N}$ and a triple $x, x', y\in B_{l}^{M}$  such that $x\ll x'$ with $\varphi (x')\ll \varphi (y)$, we will construct $\Cu$-morphisms 
  \[
   \theta\colon\LscI^{M}\to\LscI^{N},\andSep 
   \phi\colon \LscI^{N}\to S
  \]
  such that the following properties are satisfied:
  \begin{enumerate}[(i)]
 \item $\varphi (\chi^{s}_{((i+1)/l,1]})\ll \phi\theta(\chi^{s}_{(i/l,1]})$ and $\phi\theta(\chi^{s}_{((i+1)/l,1]})\ll \varphi (\chi^{s}_{(i/l,1]})$ for every $\chi^{s}_{(i/l,1]}\in C_{l}^{M}$.
 \item $\theta (x)\ll\theta (y)$.
 \item $\varphi (1_{s})=\phi\theta (1_{s})$ for every $s\leq M$.
\end{enumerate}

We generalize the proof of \autoref{Easy_case}:
\vspace{0.2cm}

For every $i\leq n$ and $s\leq M$, consider the elements $x_{i,s}=\varphi (\chi^{s}_{(i/l,1]})$ and set $x_{-\infty ,s}=\varphi (1_{s})$. Also,
 set $r=\varphi (x')$ and $t=\varphi (y)$.
 
 Let $I,J$ be the multisets of $\{-\infty ,0,\cdots ,l\}$ such that 
 \[
  x'=\sum_{s\leq M, i\in I}\chi_{(i/l,1]}^s,\andSep 
  y=\sum_{s\leq M, j\in J}\chi_{(j/l,1]}^s,
 \]
 where it is understood that $\chi_{(-\infty /l,1]}^s=1_s$. Thus, we can write $r,t$ as $r=\sum_{s\leq M, i\in I}x_{i,s}$ and $t=\sum_{j\leq M, j\in J}x_{j,s}$.
 
 Let $z_{i,s}=\varphi (\chi_{(i/2l,1]}^s)$ for every $i\leq 2l$, and let $p_s=x_{-\infty ,s}=\varphi (1_s)$, so that we have
 \[
  0=x_{l,s}=z_{2l,s}\ll z_{2l-1,s}\ll\cdots\ll z_{0,s}=x_{0,s}\ll x_{-\infty ,s}=p_s.
 \]
 
Since $z_{2i,s}=x_{i,s}$ for every $i\leq l$, we also get $r=\sum_{s\leq M, i\in I}z_{2i,s}$ and $t=\sum_{s\leq M,j\in J}z_{2j,s}$.

Apply property $\rm{I}$ to $(z_{i,s})_{i,s}$ and $r\ll t$ to obtain $(n_k,e_k)$-chainable subsemigroup $H_k$ and sequences $a_{2l,s}^{(k)}\prec\cdots\prec a_{0,s}^{(k)}$ in $X_{n_k}$ satisfying conditions (i)-(iii) in \autoref{PropertyI}. In particular, note that
\[
x_{i,s}=z_{2i,s}\ll \sum_{k}F_{k}(a_{2i-1,s}^{(k)})\ll z_{2i-1,s}\ll \sum_{k}F_{k}(a_{2(i-1),s}^{(k)})\ll z_{2(i-1),s}=x_{i-1,s}.
\]

Now fix $k<\infty $ and consider the finite family of $\prec$-relations
\[
\begin{split}
\{ a_{i,s}^{(k)}\prec a_{i-1,s}^{(k)} \}_{s,i}\cup \{ a_{0,s}^{(k)}\prec m_{k,s}(-\infty ,\infty ) \}_s \cup \{ a_k\prec b_k\}
\end{split}
\]
given by property $\rm{I}$.

By using the same argument as in \autoref{Easy_case}, we obtain an integer $N_{k}$ and elements $f_{i,s}^{(k)}$, $f_{k}$ and $g_{k}$ in $\LscI$ such that 
\[
\begin{split}
 f_{i,s}^{(k)}\ll f_{i-1,s}^{(k)}\ll m_{k,s}
 ,\\
 \rho_{N_{k}}(f_{i,s}^{(k)})\ll F_{k}(a_{i,s}^{(k)})\ll \rho_{N_{k}}(f_{i-1,s}^{(k)})\ll F_{k}(a_{i-1,s}^{(k)})
 ,\\
 f_{k}\ll g_{k}),\\
 \rho_{N_{k}}(f_{k})\ll F_{k}(a_{k})\ll 
 \rho_{N_{k}}(g_{k})\ll F_{k}(b_{k})
 ,
\end{split}
\]
where the $\Cu$-morphism $\rho_{N_{k}}\colon\LscI\to S$ is as defined in \autoref{rho} for $H_{k}$. Recall, in particular, that $\rho_{N_{k}}(1)=e_{k}$.

Combining \autoref{Prec_Convert} and \autoref{Retraction_eq}, we have 
\[
\sum_{(i,s)\in I}f_{2i,s}^{(k)}= f_{k}\ll g_{k}= \sum_{(j,s)\in J}f_{2j,s}^{(k)}.
\]

Let $\phi=\oplus_{k}\rho_{N_{k}}\colon\LscI^{N}\to S$, and define $f_{i,s}=\sum_{k}f_{i,s}^{(k)}1_{k}$ and $m_{s}=\sum_{k}m_{k,s}1_{k}$, where given $h\in\LscI$ we denote by $h1_{k}$ the function in $\LscI^{N}$ whose $k$-th component is $h$ and the rest are null.

Note that $f_{i,s}\ll f_{i-1,s}\ll m_{s}$ for each $i$ and $s$. Further, we have by construction that $\phi (1_{k})=e_{k}$.

By (ii) of property $\rm{I}$, we get
\[
 x_{i,s}\ll \phi (f_{2(i-1),s})\ll x_{i-1,s}
\]
for every $i,s$.

For each $s\leq M$, define $\theta_{s}\colon\LscI\to\LscI^{N}$ as the $\Cu$-morphism that sends $\chi_{(i/l,1]}$ to $f_{2i,s}$ and $1$ to $m_{s}$. As mentioned, $\theta_s$ exists by \autoref{lifting_th}, as $\{ f_{2i,s}\}_{i}$ is a $\ll$-decreasing sequence bounded by $1m_{s}$.

Let $\theta=\oplus_{s\leq M}\theta_{s}$. We have
\[
 \theta (x)\ll \theta (x')=\sum_{(i,s)\in I}f_{(2i,s)}=\sum_{k}f_{k}\ll \sum_{k}g_{k}=\sum_{(j,s)\in I}f_{(2j,s)}=\theta (y),
\]
which implies condition (ii) in \autoref{MainTh}.

Also, for every fixed $s$, we get
\[
\varphi (1_{s})=p_{s}= \sum_{k}m_{k,s}e_{k}= \sum_{k}m_{k,s}\phi (1_{k})=\phi (m_{s})=\phi (\theta (1)),
\]
which is condition (iii) in \autoref{MainTh}.

To see condition (i) simply note that, since $x_{i,s}\ll \phi (f_{2(i-1),s})\ll x_{i-1,s}$, we have by definition that
\[
 \varphi (\chi^{s}_{(i/l,1]})=x_{i,s}\ll \phi (\theta (\chi_{(i/l,1]}))\ll x_{i-1,s}=\varphi (\chi^{s}_{((i-1)/l,1]})
\]
as required.
\end{proof}

%===========================================
In \cite[Section~6]{Vila20} three new properties that every Cuntz semigroup of an AI-algebra satisfies are introduced. It would be interesting to know whether such properties lead to a characterization:

\begin{qst}\label{qst:SimplerCharac}
 Let $S$ be a countably based and compactly bounded \CuSgp{}. Is $S$ $\Cu$-isomorphic to the Cuntz semigroup of an AI-algebra if and only if $S$ satisfies (O5), (O6), weak cancellation, and the three properties defined in \cite[Definitions~6.1,~6.9~and~6.13]{Vila20}?
\end{qst}
%===========================================

Note that a similar characterization is indeed possible for unital commutative AI-algebras; see \cite[Theorem~5.19]{Vila20}.
%===========================================
%===========================================
%===========================================

\bibliographystyle{aomalphaMyShort}
\bibliography{References}

\providecommand{\etalchar}[1]{$^{#1}$}
\providecommand{\bysame}{\leavevmode\hbox to3em{\hrulefill}\thinspace}
\providecommand{\noopsort}[1]{}
\providecommand{\mr}[1]{\href{http://www.ams.org/mathscinet-getitem?mr=#1}{MR~#1}}
\providecommand{\zbl}[1]{\href{http://www.zentralblatt-math.org/zmath/en/search/?q=an:#1}{Zbl~#1}}
\providecommand{\jfm}[1]{\href{http://www.emis.de/cgi-bin/JFM-item?#1}{JFM~#1}}
\providecommand{\arxiv}[1]{\href{http://www.arxiv.org/abs/#1}{arXiv~#1}}
\providecommand{\doi}[1]{\url{http://dx.doi.org/#1}}
\providecommand{\MR}{\relax\ifhmode\unskip\space\fi MR }
% \MRhref is called by the amsart/book/proc definition of \MR.
\providecommand{\MRhref}[2]{%
  \href{http://www.ams.org/mathscinet-getitem?mr=#1}{#2}
}
\providecommand{\href}[2]{#2}
\begin{thebibliography}{GHK{\etalchar{+}}03}

\bibitem[ADPS14]{AntDadPerSan14RecoverElliott}
\bgroup\scshape{}R.~Antoine\egroup{}, \bgroup\scshape{}M.~Dadarlat\egroup{},
  \bgroup\scshape{}F.~Perera\egroup{}, and
  \bgroup\scshape{}L.~Santiago\egroup{}, Recovering the {E}lliott invariant
  from the {C}untz semigroup,  \emph{Trans. Amer. Math. Soc.} \textbf{366}
  (2014), 2907--2922.

\bibitem[APRT21]{AntPerRobThi21:CuntzSR1}
\bgroup\scshape{}R.~Antoine\egroup{}, \bgroup\scshape{}F.~Perera\egroup{},
  \bgroup\scshape{}L.~Robert\egroup{}, and \bgroup\scshape{}H.~Thiel\egroup{},
  \ca{s} of stable rank one and their {C}untz semigroups, Duke Math. J. (to
  appear), preprint (arXiv: arXiv:1809.03984 [math.OA]), 2021.

\bibitem[APS11]{AntPerSan11PullbacksCu}
\bgroup\scshape{}R.~Antoine\egroup{}, \bgroup\scshape{}F.~Perera\egroup{}, and
  \bgroup\scshape{}L.~Santiago\egroup{}, Pullbacks, {$C(X)$}-algebras, and
  their {C}untz semigroup,  \emph{J. Funct. Anal.} \textbf{260} (2011),
  2844--2880.

\bibitem[APT18]{AntoPereThie18}
\bgroup\scshape{}R.~Antoine\egroup{}, \bgroup\scshape{}F.~Perera\egroup{}, and
  \bgroup\scshape{}H.~Thiel\egroup{}, Tensor products and regularity properties
  of {C}untz semigroups,  \emph{Mem. Amer. Math. Soc.} \textbf{251} (2018),
  viii+191.

\bibitem[APT20]{AntPerThi20:AbsBivariantCu}
\bgroup\scshape{}R.~Antoine\egroup{}, \bgroup\scshape{}F.~Perera\egroup{}, and
  \bgroup\scshape{}H.~Thiel\egroup{}, Abstract bivariant {C}untz semigroups,
  \emph{Int. Math. Res. Not. IMRN} (2020), 5342--5386.

\bibitem[CET{\etalchar{+}}19]{CasEviTikWhiWin19arX:NucDimSimple}
\bgroup\scshape{}J.~Castillejos\egroup{},
  \bgroup\scshape{}S.~Evington\egroup{}, \bgroup\scshape{}A.~Tikuisis\egroup{},
  \bgroup\scshape{}S.~White\egroup{}, and \bgroup\scshape{}W.~Winter\egroup{},
  Nuclear dimension of simple \ca{s}, preprint (arXiv:1901.05853 [math.OA]),
  2019.

\bibitem[CE08]{CiupElli08}
\bgroup\scshape{}A.~Ciuperca\egroup{} and \bgroup\scshape{}G.~Elliott\egroup{},
  A remark on invariants for \ca{s} of stable rank one,  \emph{Int. Math. Res.
  Not. IMRN} (2008), Art. ID rnm 158, 33.

\bibitem[CES11]{CiupElliSant11}
\bgroup\scshape{}A.~Ciuperca\egroup{}, \bgroup\scshape{}G.~A. Elliott\egroup{},
  and \bgroup\scshape{}L.~Santiago\egroup{}, On inductive limits of type-{I}
  \ca{s} with one-dimensional spectrum,  \emph{Int. Math. Res. Not. IMRN}
  (2011), 2577--2615.

\bibitem[CEI08]{CowEllIva08CuInv}
\bgroup\scshape{}K.~T. Coward\egroup{}, \bgroup\scshape{}G.~A.
  Elliott\egroup{}, and \bgroup\scshape{}C.~Ivanescu\egroup{}, The {C}untz
  semigroup as an invariant for \ca{s},  \emph{J.\ Reine Angew.\ Math.}
  \textbf{623} (2008), 161--193.

\bibitem[Cun78]{Cun78DimFct}
\bgroup\scshape{}J.~Cuntz\egroup{}, Dimension functions on simple \ca{s},
  \emph{Math. Ann.} \textbf{233} (1978), 145--153.

\bibitem[EK86]{EffrKami86}
\bgroup\scshape{}E.~G. Effros\egroup{} and
  \bgroup\scshape{}J.~Kaminker\egroup{}, Homotopy continuity and shape theory
  for \ca{s},  in \emph{Geometric methods in operator algebras ({K}yoto,
  1983)}, \emph{Pitman Res. Notes Math. Ser.} \textbf{123}, Longman Sci. Tech.,
  Harlow, 1986, pp.~152--180.

\bibitem[EHS80]{EffHanShe80}
\bgroup\scshape{}E.~G. Effros\egroup{}, \bgroup\scshape{}D.~E.
  Handelman\egroup{}, and \bgroup\scshape{}C.~L. Shen\egroup{}, Dimension
  groups and their affine representations,  \emph{Amer. J. Math.} \textbf{102}
  (1980), 385--407.

\bibitem[GHK{\etalchar{+}}03]{GieHof+03Domains}
\bgroup\scshape{}G.~Gierz\egroup{}, \bgroup\scshape{}K.~H. Hofmann\egroup{},
  \bgroup\scshape{}K.~Keimel\egroup{}, \bgroup\scshape{}J.~D. Lawson\egroup{},
  \bgroup\scshape{}M.~Mislove\egroup{}, and \bgroup\scshape{}D.~S.
  Scott\egroup{}, \emph{Continuous lattices and domains}, \emph{Encyclopedia of
  Mathematics and its Applications} \textbf{93}, Cambridge University Press,
  Cambridge, 2003.

\bibitem[JS99]{JiaSu99}
\bgroup\scshape{}X.~Jiang\egroup{} and \bgroup\scshape{}H.~Su\egroup{}, On a
  simple unital projectionless \ca{},  \emph{Amer. J. Math.} \textbf{121}
  (1999), 359--413.

\bibitem[KR14]{KirRor14CentralSeq}
\bgroup\scshape{}E.~Kirchberg\egroup{} and
  \bgroup\scshape{}M.~R{\o}rdam\egroup{}, Central sequence \ca{s} and tensorial
  absorption of the {J}iang-{S}u algebra,  \emph{J. Reine Angew. Math.}
  \textbf{695} (2014), 175--214.

\bibitem[Nad92]{Nadl92}
\bgroup\scshape{}S.~Nadler\egroup{}, \emph{Continuum theory: An introduction},
  \emph{Monographs and textbooks in pure and applied mathematics} \textbf{158},
  M. Dekker, 1992.

\bibitem[Rob12]{Rob12LimitsNCCW}
\bgroup\scshape{}L.~Robert\egroup{}, Classification of inductive limits of
  1-dimensional {NCCW} complexes,  \emph{Adv. Math.} \textbf{231} (2012),
  2802--2836.

\bibitem[Rob13]{Rob13Cone}
\bgroup\scshape{}L.~Robert\egroup{}, The cone of functionals on the {C}untz
  semigroup,  \emph{Math. Scand.} \textbf{113} (2013), 161--186.

\bibitem[RS10]{RobSan10}
\bgroup\scshape{}L.~Robert\egroup{} and \bgroup\scshape{}L.~Santiago\egroup{},
  Classification of {$C\sp \ast$}-homomorphisms from {$C\sb 0(0,1]$} to a
  \ca{},  \emph{J. Funct. Anal.} \textbf{258} (2010), 869--892.

\bibitem[RLL00]{LarLauRor00KThy}
\bgroup\scshape{}M.~R{\o}rdam\egroup{}, \bgroup\scshape{}F.~Larsen\egroup{},
  and \bgroup\scshape{}N.~Laustsen\egroup{}, \emph{An introduction to
  {$K$}-theory for \ca{s}}, \emph{London Mathematical Society Student Texts}
  \textbf{49}, Cambridge University Press, Cambridge, 2000.

\bibitem[R{\o}r04]{Ror04StableRealRankZ}
\bgroup\scshape{}M.~R{\o}rdam\egroup{}, The stable and the real rank of
  {$\mathcal{Z}$}-absorbing \ca{s},  \emph{Internat. J. Math.} \textbf{15}
  (2004), 1065--1084.

\bibitem[RW10]{RorWin10ZRevisited}
\bgroup\scshape{}M.~R{\o}rdam\egroup{} and \bgroup\scshape{}W.~Winter\egroup{},
  The {J}iang-{S}u algebra revisited,  \emph{J. Reine Angew. Math.}
  \textbf{642} (2010), 129--155.

\bibitem[Sat12]{Sat12arx:TraceSpace}
\bgroup\scshape{}Y.~Sato\egroup{}, Trace spaces of simple nuclear \ca{s} with
  finite-dimensional extreme boundary, preprint (arXiv:1209.3000 [math.OA]),
  2012.

\bibitem[Sch18]{Scho18}
\bgroup\scshape{}C.~Schons\egroup{}, Categorical aspects of {C}untz semigroups,
  Master thesis, 2018.

\bibitem[She79]{Shen79}
\bgroup\scshape{}C.-L. Shen\egroup{}, On the classification of the ordered
  groups associated with the approximately finite dimensional \ca{s},
  \emph{Duke Math. J.} \textbf{46} (1979), 613--633.

\bibitem[Tho92]{Tho92:IndLimIntAlg}
\bgroup\scshape{}K.~Thomsen\egroup{}, Inductive limits of interval algebras:
  unitary orbits of positive elements,  \emph{Math. Ann.} \textbf{293} (1992),
  47--63.

\bibitem[Tom08]{Tom08InfFamily}
\bgroup\scshape{}A.~S. Toms\egroup{}, An infinite family of non-isomorphic
  \ca{s} with identical {$K$}-theory,  \emph{Trans. Amer. Math. Soc.}
  \textbf{360} (2008), 5343--5354.

\bibitem[TWW15]{TomWhiWin15ZStableFdBauer}
\bgroup\scshape{}A.~S. Toms\egroup{}, \bgroup\scshape{}S.~White\egroup{}, and
  \bgroup\scshape{}W.~Winter\egroup{}, {$\mathcal{Z}$}-stability and
  finite-dimensional tracial boundaries,  \emph{Int. Math. Res. Not. IMRN}
  (2015), 2702--2727.

\bibitem[Vil21]{Vila20}
\bgroup\scshape{}E.~Vilalta\egroup{}, The {C}untz semigroup of unital
  commutative {A}{I}-algebras, in preparation, 2021.

\bibitem[Win12]{Win12NuclDimZstable}
\bgroup\scshape{}W.~Winter\egroup{}, Nuclear dimension and
  {$\mathcal{Z}$}-stability of pure \ca{s},  \emph{Invent. Math.} \textbf{187}
  (2012), 259--342.

\end{thebibliography}

\end{document}